\documentclass[a4paper,reqno]{amsart}

\usepackage{graphicx}
\usepackage{tikz}
\usetikzlibrary{patterns,shapes,decorations.pathmorphing,decorations.pathreplacing,calc,arrows}
\usepackage{verbatim}
\usepackage{amssymb}
\usepackage{amstext}
\usepackage{amsmath}
\usepackage{amscd}
\usepackage{amsthm}
\usepackage{amsfonts}
\usepackage{enumerate}
\usepackage{latexsym}
\usepackage{mathrsfs}
\usepackage{hyperref}
\usepackage{mathtools}
\usepackage[all]{xy}
\usepackage{pgfplots}
\xyoption{all}

\date{\today}

\usepackage{pstricks}
\usepackage{lscape}
\usepackage{comment}

\newtheorem{theorem}{Theorem}[section]

\newtheorem{corollary}[theorem]{Corollary}
\newtheorem{lemma}[theorem]{Lemma}
\newtheorem{proposition}[theorem]{Proposition}
\newtheorem{definition-proposition}[theorem]{Definition-Proposition}

\newtheorem{conjecture}[theorem]{Conjecture}

\theoremstyle{definition}
\newtheorem{definition}[theorem]{Definition}

\newtheorem{remark}[theorem]{Remark}
\newtheorem{example}[theorem]{Example}

\newcommand{\CC}{\mathcal{C}}

\newcommand{\DDD}{\mathsf{D}}

\newcommand{\KKK}{\mathsf{K}}

\newcommand{\FF}{\mathcal{F}}

\newcommand{\TT}{\mathcal{T}}

\newcommand{\WW}{\mathcal{W}}

\newcommand{\Z}{\mathbb{Z}}

\newcommand{\Q}{\mathbb{Q}}
\newcommand{\R}{\mathbb{R}}
\newcommand{\N}{\mathbb{N}}

\newcommand{\h}{\operatorname{h}\nolimits}
\newcommand{\bo}{\operatorname{b}\nolimits}
\newcommand{\cone}{\operatorname{cone}\nolimits}

\renewcommand{\top}{\operatorname{top}\nolimits}

\DeclareMathOperator{\Hom}{\operatorname{Hom}}
\DeclareMathOperator{\cHom}{\mathcal{H}{\rm om}}
\newcommand{\End}{\operatorname{End}\nolimits}
\newcommand{\Aut}{\operatorname{Aut}\nolimits}

\newcommand{\op}{\operatorname{op}\nolimits}
\newcommand{\RHom}{\mathbf{R}\strut\kern-.2em\operatorname{Hom}\nolimits}
\newcommand{\Lotimes}{\mathop{\stackrel{\mathbf{L}}{\otimes}}\nolimits}

\newcommand{\Kernel}{\operatorname{Ker}\nolimits}
\newcommand{\Cokernel}{\operatorname{Cok}\nolimits}

\newcommand{\Cone}{\mathsf{Cone}}
\newcommand{\Wall}{\mathsf{Wall}}

\DeclareMathOperator{\moduleCategory}{\mathsf{mod}} \renewcommand{\mod}{\moduleCategory}

\DeclareMathOperator{\proj}{\mathsf{proj}}

\DeclareMathOperator{\ind}{{\rm ind}}
\DeclareMathOperator{\Sub}{\mathsf{Sub}}
\DeclareMathOperator{\thick}{\mathsf{thick}}

\DeclareMathOperator{\serre}{\mathsf{serre}}
\DeclareMathOperator{\add}{\mathsf{add}}
\DeclareMathOperator{\fl}{\mathsf{fl}}

\DeclareMathOperator{\TF}{\mathsf{TF}}

\DeclareMathOperator{\dimv}{\underline{dim}}

\DeclareMathOperator{\tors}{\mathsf{tors}}
\DeclareMathOperator{\ftors}{\mathsf{f-tors}}
\DeclareMathOperator{\torf}{\mathsf{torf}}
\DeclareMathOperator{\ftorf}{\mathsf{f-torf}}
\DeclareMathOperator{\wide}{\mathsf{wide}}
\DeclareMathOperator{\twosilt}{2\mathsf{-silt}}
\DeclareMathOperator{\twopresilt}{2\mathsf{-psilt}}

\DeclareMathOperator{\indtwopresilt}{\mathsf{ind-}2\mathsf{-psilt}}

\newcommand{\cut}{\ar@{-}@[|(5)]}

\DeclareMathOperator{\brick}{\mathsf{brick}}
\DeclareMathOperator{\Filt}{\mathsf{Filt}}
\DeclareMathOperator{\vecFilt}{\operatorname{\overrightarrow{\mathsf{Filt}}}}
\DeclareMathOperator{\Fac}{\mathsf{Fac}}
\DeclareMathOperator{\T}{\mathsf{T}}
\DeclareMathOperator{\F}{\mathsf{F}}

\DeclareMathOperator{\Mat}{Mat}

\renewcommand{\subset}{\subseteq}
\renewcommand{\supset}{\supseteq}

\numberwithin{equation}{section}

\begin{document}
\title{Semistable~torsion~classes and
Canonical~decompositions in Grothendieck groups}

\author{Sota Asai} 
\address{Sota Asai: Graduate School of Mathematical Sciences,
University of Tokyo,  
3-8-1 Komaba, Meguro-ku, Tokyo-to, 153-8914, Japan}
\email{sotaasai@g.ecc.u-tokyo.ac.jp}

\author{Osamu Iyama}
\address{Osamu Iyama: Graduate School of Mathematical Sciences,
University of Tokyo,  
3-8-1 Komaba, Meguro-ku, Tokyo-to, 153-8914, Japan}
\email{iyama@ms.u-tokyo.ac.jp}

\begin{abstract}
We study two classes of torsion classes which generalize functorially finite torsion classes, that is, semistable torsion classes and morphism torsion classes.
Semistable torsion classes are parametrized by the elements in the real Grothendieck group up to TF equivalence.
We give a close connection between TF equivalence classes and the cones given by canonical decompositions of the spaces of projective presentations due to Derksen-Fei. 
More strongly, for $E$-tame algebras and hereditary algebras, we prove that TF equivalence classes containing lattice points are exactly the cones given by canonical decompositions.
One of the key steps in our proof is a general description of semistable torsion classes in terms of morphism torsion classes.
We also answer a question by Derksen-Fei negatively by giving examples
of algebras which do not satisfy the ray condition.
As an application of our results, we give an explicit description of TF equivalence classes of preprojective algebras of type $\widetilde{\mathbb{A}}$.
\end{abstract}

\maketitle 
\setcounter{tocdepth}{1}
\tableofcontents

\section{Introduction}
Derived categories are basic in homological algebra and appear in many branches of mathematics,
and tilting theory is a powerful tool to study equivalences of the derived categories. There are two important notions in tilting theory, that is, tilting/silting complexes and t-structures.
Two rings $A$ and $B$ are derived equivalent if and only if there exists a tilting complex of $A$ whose endomorphism ring is isomorphic to $B$ \cite{Rickard}. The class of silting complexes is a generalization of the class of tilting complexes from the point of view of mutation, which is a categorical operation to construct a new silting complex from a given one by replacing a direct summand.
A t-structure is a pair of two full subcategories satisfying certain axioms, and intermediate t-structures correspond bijectively with torsion classes in the module category \cite{HRS}.
There is a bijection between silting complexes and algebraic t-structures \cite{KY}, which give bijections between 2-term silting complexes, intermediate algebraic t-structures and functorially finite torsion classes \cite{AIR}.
It plays a key role in the additive categorification of cluster algebras \cite{FZ} (e.g.\ \cite{BY,CKLP}).
There are a large number of works on torsion classes. It is known that a finite dimensional algebra $A$ is $g$-finite (i.e.\ $A$ has only finitely many basic 2-term silting complexes up to isomorphism) if and only if all torsion classes are functorially finite \cite{DIJ,ZZ}. If $A$ is not $g$-finite, then most torsion classes are not functorially finite.

The aim of this paper is to study two classes of torsion classes containing all functorially finite torsion classes.
The first one is the class of torsion classes determined by stability conditions, i.e.\ elements $\theta$ in the real Grothendieck group $K_0(\proj A)_\R:=K_0(\proj A) \otimes_\Z \R$ of the category $\proj A$ of finitely generated projective $A$-modules.
The notion of $\theta$-semistable modules naturally appears in geometric invariant theory of quiver representations \cite{K}. Each $\theta $ gives two torsion pairs $(\overline{\TT}_\theta,\FF_\theta)$ and $(\TT_\theta,\overline{\FF}_\theta)$ \cite{BKT,Bridgeland}, which we call \textit{semistable torsion pairs}. 
They satisfy $\overline{\TT}_\theta\supseteq\TT_\theta$ and $\FF_\theta\subseteq\overline{\FF}_\theta$, and the intersection $\overline{\TT}_\theta\cap\overline{\FF}_\theta$ is the wide subcategory of $\theta$-semistable modules.
The semistable torsion classes $\overline{\TT}_\theta,\TT_\theta$ of $\theta=[U]$ for a 2-term presilting complex $U$ are functorially finite \cite{Y,BST} and well-studied in tilting theory.

Using the semistable torsion pairs, the first author \cite{A} introduced an equivalence relation on the real Grothendieck group $K_0(\proj A)_\R$ as follows: We call $\theta,\eta\in K_0(\proj A)_\R$ \textit{TF equivalent} if
\begin{align*}\overline{\TT}_\theta=\overline{\TT}_\eta\ \mbox{ and }\ \TT_\theta=\TT_\eta.\end{align*}
We denote by $[\theta]_{\rm TF}$ the TF equivalence class of $\theta$. It is an important problem to give an explicit description of TF equivalence classes of an arbitrary element in $K_0(\proj A)$. 
For a subset $X$ of $K_0(\proj A)_\R$, let $X^\circ$ be the relative interior of $X$, and let
\begin{align*}
 \cone X:=\sum_{\theta \in X}\R_{\ge 0}\theta\supseteq\cone^\circ X:=(\cone X)^\circ=\sum_{\theta \in X}\R_{>0}\theta
\end{align*}
with $\cone^\circ\emptyset=\cone\emptyset:=\{0\}$.
If there exists a 2-term presilting complex $U$ such that $\theta=[U]$, then $[\theta]_{\rm TF}$ can be described as
\begin{equation}\label{TF of rigid}
[\theta]_{\rm TF}=\cone^\circ\{[U_1],\ldots,[U_\ell]\}
\end{equation}
where $U=U_1\oplus\cdots\oplus U_\ell$ is a decomposition into indecomposable direct summands.

In this paper, we show that there is a close connection between TF equivalence classes and the canonical decomposition of a space of projective presentations introduced by Derksen-Fei \cite{DF}. It is an analogue of the canonical decomposition of a space of representations of quivers (with relations) \cite{Kac,S,CS}, and played an important role in categorification of cluster algebras \cite{P}.
For example, if $U$ is a 2-term presilting complex in the homotopy category $\KKK^{\bo}(\proj A)$ and $U=U_1\oplus\cdots\oplus U_\ell$ is a decomposition into indecomposable direct summands, then $[U]=[U_1]\oplus\cdots\oplus[U_\ell]$ is a canonical decomposition.

Let $A$ be a finite dimensional algebra over an algebraically closed field $k$. For $\theta\in K_0(\proj A)$, we take a canonical decomposition $\theta=\theta_1\oplus\cdots\oplus\theta_\ell$, and set
\[\ind\theta:=\{\theta_1,\ldots,\theta_\ell\},\ |\theta|:=\#\ind\theta\ \mbox{ and }\ \ind\N\theta:=\bigcup_{\ell\ge1}\ind\ell\theta.\]
For example, if $\theta$ is rigid, then $\ind\theta=\ind\N\theta$ and $\dim\cone(\ind\theta)=|\theta|$ hold.
Our first main result shows that all elements in the cone given by a canonical decomposition are TF equivalent.

\begin{theorem}[{Theorem \ref{decomposition and T}}]\label{decomposition and T in intro}
Let $A$ be a finite dimensional algebra over an algebraically closed field $k$. For each $\theta\in K_0(\proj A)$, we have
\begin{align*}
[\theta]_{\rm TF}\supseteq\cone^\circ(\ind\theta).
\end{align*}
\end{theorem}

Since $[\theta]_{\rm TF}=[\ell\theta]_{\rm TF}$ holds for each $\ell\ge1$, Theorem \ref{decomposition and T in intro} implies $[\theta]_{\rm TF}\supseteq\cone^\circ(\ind\N\theta)$. Notice that $\cone(\ind\N\theta)\supseteq\cone(\ind\theta)$ holds clearly, but the equality does not necessarily hold, see Theorem \ref{counter to ray intro} below. 
It is natural to pose the following as a large generalization of \eqref{TF of rigid}.

\begin{conjecture}\label{canonical and TF}
For each $\theta\in K_0(\proj A)$, we have
\begin{align*}
[\theta]_{\rm TF}=\cone^\circ(\ind\N\theta).
\end{align*}
\end{conjecture}

Our second main result shows that Conjecture \ref{canonical and TF} is true for two classes of algebras. 
The first one is the class of hereditary algebras (that is, algebras whose global dimension is at most one), which contains the path algebras of acyclic quivers. 
The second one is defined in terms of \textit{$E$-invariants}: For $\theta,\eta\in K_0(\proj A)$, let
\begin{align*}
E(\eta,\theta)&:=\min\{\dim_k\Hom_{\KKK^{\bo}(\proj A)}(P_f,P_g[1])\mid (f,g)\in\Hom(\eta)\times\Hom(\theta)\}.
\end{align*}
An algebra $A$ is called \textit{$E$-tame} if $E(\theta,\theta)$ is zero for all $\theta\in K_0(\proj A)$. 
This class contains all $g$-finite algebras as well as representation-tame algebras \cite{GLFS,PY}.
We refer to Figure \ref{tame figure} in Section \ref{Section_ray} for relationship between some variations of finiteness and tameness, where there have been many recent works on these notions including
\cite{AAC,AMY,AHMW,AHIKM,AMV2,ArS,A-semi,AMN,Au,DIJ,FG,HW,IZ,KM,M,Mo,MP,Mu,P-finite,PY,STTVW,ST,STV,W,Z}.

Now we are ready to state our second main result.

\begin{theorem}[{Theorems \ref{describe TF for E-tame} and  \ref{describe TF for hereditary}}]\label{converse in intro}
Let $A$ be a finite dimensional algebra over an algebraically closed field $k$, and
$\theta\in K_0(\proj A)$.
If $A$ is either hereditary or $E$-tame, then
\begin{align*}[\theta]_{\rm TF}=\cone^\circ(\ind\theta).
\end{align*}
\end{theorem}

In the proof of Theorem \ref{converse in intro} for $E$-tame algebras, we prove the following characterization of $E$-tame algebras, which is interesting by itself.

\begin{theorem}[Theorem \ref{summand not TF}]
For a finite dimensional algebra $A$ over an algebraically closed field $k$, the following conditions are equivalent.
\begin{enumerate}[\rm(a)]
\item $A$ is $E$-tame. 
\item Let $\eta, \theta \in K_0(\proj A)$. Then $\eta$ and $\theta$ are TF equivalent if and only if $\ind\eta=\ind\theta$.
\end{enumerate}
\end{theorem}

Another class of torsion classes studied in this paper is given by morphisms between projective modules. For each morphism $f$ in the category $\proj A$, we obtain torsion pairs $(\TT_f,\overline{\FF}_f)$ and $(\overline{\TT}_f,\FF_f)$ which we call \textit{morphism torsion pairs}.
If $f$ is presilting as a 2-term complex, then the morphism torsion classes are functorially finite and well studied in tilting theory (e.g.~\cite{ASS,AIR,AMV}).
In this paper, we will show that semistable torsion classes of $\theta\in K_0(\proj A)$ can be described by using morphism torsion classes. More explicitly, by unifying morphism torsion pairs of each morphism $f$ in $\proj A$ satisfying $[f]=\theta$, we define torsion pairs $(\TT^{\h}_{\theta},\overline{\FF}^{\h}_{\theta})$ and $(\overline{\TT}^{\h}_{\theta},\FF^{\h}_{\theta})$.
We prove the equalities below, which are also used in the proof of Theorem \ref{converse in intro}. Note that they were obtained by Fei \cite{Fei} independently.

\begin{theorem}[{Theorem \ref{cup T_f}}]\label{cup T_f in intro}
Let $A$ be a finite dimensional algebra over an algebraically closed field $k$. For $\theta \in K_0(\proj A)$, we have
\begin{align*}
\TT_\theta=\bigcap_{\ell\ge1}\TT^{\h}_{\ell\theta},\quad
\FF_\theta=\bigcap_{\ell\ge1}\FF^{\h}_{\ell\theta},\quad
\overline{\TT}_\theta=\bigcup_{\ell\ge1}\overline{\TT}^{\h}_{\ell\theta},\quad
\overline{\FF}_\theta=\bigcup_{\ell\ge1}\overline{\FF}^{\h}_{\ell\theta},\quad
\WW_\theta=\bigcup_{\ell\ge1}\WW^{\h}_{\ell\theta}.
\end{align*}
Moreover, we can let $\ell=1$ above if $\theta$ is tame.
\end{theorem}

As an application of our results, we study 
the behavior of canonical decomposition under multiplication by a positive integer.
We say that an algebra $A$ satisfies the \emph{ray condition} if for each indecomposable wild element $\theta$ and $\ell\ge1$, the element $\ell\theta$ is indecomposable. 
We show that the ray condition is satisfied by $E$-tame algebras and hereditary algebras (Propositions \ref{linear independence2}, \ref{linear independence3}), and also give an example of an algebra which does not satisfy the ray condition, answering a question \cite[Question 4.7]{DF} negatively.

\begin{theorem}[{Example \ref{counter to ray 2}}]\label{counter to ray intro}
There exists a finite dimensional algebra $A$ and an indecomposable wild element $\theta\in K_0(A)$ such that $\cone(\ind\theta)\subsetneq\cone(\ind\N\theta)$. In particular, $\theta$ does not satisfy the ray condition.
\end{theorem}

In Section \ref{Section_preproj}, we give the following explicit descriptions of TF equivalence classes of the complete preprojective algebra $\Pi$ of type $\widetilde{\mathbb{A}}_{n-1}$, where $h:=\sum_{i=1}^n [S(i)] \in K_0(\fl \Pi)$ and $H:=\Kernel \langle ?,h \rangle\subset K_0(\proj\Pi)_\R$. We refer to Section \ref{Section_preproj} for details and other notations.

\begin{theorem}[Proposition \ref{half planes}, Theorem \ref{preproj TF}]\label{type A tilde}
Let $\Pi$ be the complete preprojective algebra $\Pi$ of type $\widetilde{\mathbb{A}}_{n-1}$. As subsets of $K_0(\proj \Pi)_\R$,
\begin{align*}
H^+ \cup H^-=\bigsqcup_{U \in (\twopresilt \Pi) \setminus \{0\}}C^\circ(U)
=\bigsqcup_{J \subset \{1,2,\ldots,n\}, \ w \in W/W_J}\sigma^*_w
(C^\circ(P_J))
\end{align*} 
is the decomposition into the TF equivalence classes.
The decomposition of $H \subset K_0(\proj \Pi)$ into the TF equivalence classes is
\begin{align*}
H=\bigsqcup_{U \in \twopresilt \Pi'}\iota(C^\circ(U))
=\bigsqcup_{J \subset \{1,2,\ldots,n-1\}, \ w \in W'/W'_J}\sigma^*_w
\left(\sum_{j \in J} \cone^\circ\{[P(j)]-[P(n)] \mid j \in J\}\right).
\end{align*}
\end{theorem}

It will be interesting to understand a connection between our Theorem \ref{type A tilde} and a realization of crystal due to Baumann-Kamnitzer-Tingley \cite{BKT,KS}.

\subsection{Convention}
In this paper, $k$ is an algebraically closed field,
and $A$ is a finite dimensional $k$-algebra.
We write $\mod A$ for the category of finitely generated right $A$-modules,
and $\proj A$ for the category of finitely generated projective right $A$-modules.
The bounded derived category of $\mod A$ is denoted by $\DDD(A):=\DDD^{\bo}(\mod A)$, and 
the homotopy category of the bounded complexes over $\proj A$ 
is denoted by $\KKK^{\bo}(\proj A)$.

Unless otherwise stated, any subcategory is assumed to be a full subcategory.

For any subcategory $\CC \subset \mod A$,
we set
\begin{align*}
\CC^\perp&:=\{ X \in \mod A \mid \Hom_A(\CC,X)=0 \}, \\
{^\perp \CC}&:=\{ X \in \mod A \mid \Hom_A(X,\CC)=0 \}, \\
\add \CC&:=\{ X \in \mod A \mid \text{$X$ is a direct summand of $C^{\oplus m}$ for some $C \in \CC$ and $m \ge 1$}\}, \\
\Fac \CC&:=\{ X \in \mod A \mid \text{$X$ is isomorphic to a factor module of some $C \in \add \CC$}\}, \\
\Sub \CC&:=\{ X \in \mod A \mid \text{$X$ is isomorphic to a submodule of some $C \in \add \CC$}\}, \\
\Filt \CC&:=\{ X \in \mod A \mid \text{there exist $0=X_0 \subset X_1 \subset \cdots \subset X_\ell=X$ such that $X_i/X_{i-1} \in \add \CC$} \}.
\end{align*}

\subsection*{Acknowledgments}
The authors thank Laurent Demonet for useful discussions at the first stage of this project.
They also thank Jiarui Fei for informing us of his paper \cite{Fei} on Theorem \ref{cup T_f in intro}.
S.A. was supported by JSPS KAKENHI Grant Numbers JP16J02249, JP19K14500 and JP20J00088. O.I. was supported by JSPS KAKENHI Grant Numbers JP15H05738, JP16H03923 and JP18K03209.

\section{Preliminaries}

\subsection{Torsion pairs and silting theory}

We first recall some terminology on torsion pairs.
Let $\TT,\FF$ be full subcategories of $\mod A$.
We call the pair $(\TT,\FF)$ a \textit{torsion pair} in $\mod A$
if and only if $\FF=\TT^\perp$ and $\TT={^\perp \FF}$.
This is equivalent to that the following two conditions hold:
\begin{itemize}
\item
$\Hom_A(\TT,\FF)=0$;
\item
for any $X \in \mod A$, there exists a short exact sequence 
$0 \to X' \to X \to X'' \to 0$ for some $X' \in \TT$ and $X'' \in \FF$.
\end{itemize}
A subcategory $\TT \subset \mod A$ is called a \textit{torsion class} if
there exists $\FF \subset \mod A$ such that $(\TT,\FF)$ is a torsion pair in $\mod A$.
We can check that $\TT \subset \mod A$ is a torsion class
if and only if $\TT$ is closed under taking factor modules and extensions.
Similarly, we can define \textit{torsion-free classes}.

We write $\tors A$ (resp.~$\torf A$) for the set of torsion classes (resp.~torsion-free) classes in $\mod A$.
$\tors A$ and $\torf A$ are lattices with respect to inclusions,
so we write $\vee$ for the joins in these lattices.

For any subcategory $\CC \in \mod A$, we can check that
$\T(\CC):={^\perp(\CC^\perp)}$ is the smallest torsion class containing $\CC$, and 
$\F(\CC):=({^\perp \CC})^\perp$ is the smallest torsion-free class containing $\CC$.

Here we also recall the definition of wide subcategories.
A subcategory $\WW \subset \mod A$ is called a \textit{wide subcategory} if $\WW$ is 
closed under taking kernels, cokernels and extensions.
We define $\wide A$ as the set of wide subcategories in $\mod A$.
If $\TT \in \tors A$ and $\FF \in \torf A$, then $\TT \cap \FF$ is closed under taking images.
It is not necessarily a wide subcategory, but we will often deal with wide subcategories obtained as the intersections of torsion classes and torsion-free classes in this paper.

Let $X \in \mod A$ and $\CC \subset \mod A$.
Then a homomorphism $f \colon X \to C$ is called a
\textit{left $\CC$-approximation} of $X$ if $C \in \CC$
and $f$ induces a surjection $\Hom_A(C,C') \to \Hom_A(X,C')$ for any $C' \in \CC$.
Dually, \textit{right $\CC$-approximations} of $X$ are also defined.
A full subcategory $\CC \subset \mod A$ is said to be \textit{functorially finite} in $\mod A$
if any $X \in \mod A$ admits a left $\CC$-approximation and a right $\CC$-approximation.

Thus we can consider functorially finite torsion(-free) classes in $\mod A$.
For any torsion pair $(\TT,\FF)$ in $\mod A$,
$\TT$ is functorially finite if and only if $\FF$ is functorially finite \cite{Smalo},
so we call such a torsion pair a \textit{functorially finite torsion pair} in $\mod A$.
We define $\ftors A$ (resp.~$\ftorf A$) as the set of functorially finite
torsion (resp. torsion-free) classes in $\mod A$.

Functorially finite torsion(-free) classes are strongly related to 
silting theory established by \cite{KV}.

In the definition below,
we say that a complex $U \in \KKK^{\bo}(\proj A)$ is \textit{2-term}
if its terms except $-1$st and $0$th ones vanish,
and a full subcategory of a triangulated category is said to be \textit{thick}
if it is closed under taking direct summands.

\begin{definition}
Let $U$ be a 2-term complex in $\KKK^{\bo}(\proj A)$.
\begin{enumerate}[\rm(a)]
\item
A 2-term complex $U$ in $\KKK^{\bo}(\proj A)$ is called 
\textit{presilting} if $\Hom_{\KKK^{\bo}(\proj A)}(U,U[1])=0$. 
We write $\twopresilt A$ for the set of isomorphism classes of 
basic 2-term presilting complexes in $\KKK^{\bo}(\proj A)$.
We set $\indtwopresilt A \subset \twopresilt A$ as the subset of indecomposable 2-term presilting complexes.
\item
A 2-term presilting complex $T$ is called \textit{silting} if 
the smallest thick subcategory containing $T$ is $\KKK^{\bo}(\proj A)$ itself.
We write $\twosilt A$ for the set of isomorphism classes of 
basic 2-term silting complexes in $\KKK^{\bo}(\proj A)$.
\end{enumerate}
\end{definition}

Any 2-term presilting complex is a direct summand of some 2-term silting complex by Bongartz-type Lemma \cite[Proposition 2.16]{Ai}\cite[Theorem 5.4]{DF}. 
Therefore a 2-term presilting complex $U$ is silting if and only if $|U|=|A|$ \cite[Proposition 3.3]{AIR},
where $|\cdot|$ denotes the number of nonisomorphic indecomposable direct summands.

For any 2-term presilting complex $U$,
\cite[Lemma 3.4]{AIR} and \cite{AS} tell us that we have two torsion pairs
$(\overline{\TT}_U,\FF_U)$ and $(\TT_U,\overline{\FF}_U)$ given by
\begin{align}\notag
\overline{\TT}_U&:={^\perp H^{-1}(\nu U)}, & \FF_U&:=\Sub H^{-1}(\nu U), \\ \label{define T_U}
\TT_U&:=\Fac H^0(U), & \overline{\FF}_U&:={H^0(U)}^\perp,
\end{align}
which are all functorially finite.
In general, these functorially finite torsion pairs $(\overline{\TT}_U,\FF_U)$ and $(\TT_U,\overline{\FF}_U)$
do not coincide; they coincide if and only if $U$ is 2-term silting by
\cite[Theorems 2.12, 3.2, Propositions 2.16, 3.6]{AIR}.

Now we can refer to the following important result by Adachi-Iyama-Reiten.

\begin{proposition}\label{AIR}\cite[Theorems 2.7, 3.2]{AIR}
We have bijections
\begin{align*}
\twosilt A \to \ftors A, \quad \twosilt A \to \ftorf A
\end{align*}
given by
\begin{align*}
T \mapsto \overline{\TT}_T=\TT_T, \quad T \mapsto \overline{\FF}_T=\FF_T.
\end{align*}
\end{proposition}

Based on this, the first named author of this paper proved the following properties.

\begin{lemma}\label{func fin presilt}
Let $U,V \in \twopresilt A$.
\begin{enumerate}[\rm(a)]
\item\cite[Lemma 3.13]{A} $U \in \add V$ if and only if 
$\TT_U\subseteq\TT_V\subseteq\overline{\TT}_V\subseteq\overline{\TT}_U$. In particular, $U=V$ holds if and only if $(\overline{\TT}_U,\FF_U)=(\overline{\TT}_V,\FF_V)$ and 
$(\TT_U,\overline{\FF}_U)=(\TT_V,\overline{\FF}_V)$.
\item If $U\oplus V\in\twopresilt A$, then $\TT_U\subset\overline{\TT}_V$, $\TT_V\subset\overline{\TT}_U$, $\FF_U\subset\overline{\FF}_V$ and $\FF_V\subset\overline{\FF}_U$.
\end{enumerate}
\end{lemma}

\begin{proof}
(b) We only prove the first assertion. By (a), we have $\TT_U\subseteq\TT_{U\oplus V}\subseteq\overline{\TT}_{U\oplus V}\subseteq\overline{\TT}_V$.
\end{proof}

\subsection{Semistable torsion pairs}

As we have seen in the previous section, 
functorially finite torsion pairs are important examples of torsion pairs.
We would like to extend the results on functorially torsion pairs to 
some wider class of torsion pairs.
In this paper, we mainly focus on \textit{semistable torsion pairs} introduced by \cite[Section 3.1]{BKT}, which are associated to the elements of 
the (real) Grothendieck group of $\proj A$.

For an exact category $\CC$,
the Grothendieck group of $\CC$ is denoted by $K_0(\CC)$ as usual.
If $\CC=\proj A$, the isoclasses of indecomposable projective modules $P(1),P(2),\ldots,P(n)$
give a canonical $\Z$-basis of $K_0(\proj A)$.
We set $S(i)$ as the simple top of $P(i)$ for each $i$.
Then the isoclasses of simple modules $S(1),S(2),\ldots,S(n)$
form a canonical $\Z$-basis of $K_0(\mod A)$.
In this paper, the \textit{Euler form} is the $\Z$-bilinear form
\begin{align*}\langle -,- \rangle \colon K_0(\proj A) \times K_0(\mod A) \to \Z\end{align*}
satisfying $\langle P(i),S(j) \rangle=\delta_{i,j}$ for any $i,j \in \{1,2,\ldots,n\}$.
Each element $\theta \in K_0(\proj A)$ defines
$\theta:=\langle \theta,- \rangle \colon K_0(\mod A) \to \Z$,
which means that we can regard $K_0(\proj A)$ as the dual of $K_0(\mod A)$.

Grothendieck groups can be defined for triangulated categories in a similar way.
We can check that $K_0(\proj A) \simeq K_0(\KKK^{\bo}(\proj A))$ and 
$K_0(\mod A) \simeq K_0(\DDD^{\bo}(\mod A))$.
The Euler form satisfies
\begin{align*}
\langle P,X \rangle=\sum_{\ell \in \Z}(-1)^\ell
\dim_k\Hom_{\DDD^{\bo}(\mod A)}(P,X[\ell])
\end{align*}
for any $P \in \KKK^{\bo}(\proj A)$ and $X \in \DDD^{\bo}(\mod A)$.

For any Grothendieck group $K_0(\CC)$,
we call $K_0(\CC)_\R:=K_0(\CC) \otimes_\Z \R$ 
the \textit{real Grothendieck group} of $\CC$.
Then $K_0(\proj A)_\R$ and $K_0(\mod A)_\R$ are identified with the Euclidean space $\R^n$.
Thus we can consider $K_0(\proj A)_\R$ and $K_0(\mod A)_\R$ as topological spaces.
Clearly, for each $\theta \in K_0(\proj A)_\R$, 
we have an $\R$-linear form 
$\theta:=\langle \theta,- \rangle \colon K_0(\mod A)_\R \to \R$.

Now we can recall the definition of semistable torsion pairs.

\begin{definition}\label{define T_theta}
Let $\theta\in K_0(\proj A)_\R$.
\begin{enumerate}[\rm(a)]
\item\cite[Section 3.1]{BKT}
We define two \textit{semistable torsion pairs} $(\overline{\TT}_\theta,\FF_\theta)$ and $(\TT_\theta,\overline{\FF}_\theta)$ by
\begin{align*}
\TT_\theta&:=\{X\in\mod A\mid\theta(X')>0\ \text{for all factor modules $X'\neq 0$ of $X$}\},\\
\overline{\TT}_\theta&:=\{X\in\mod A\mid\theta(X')\ge0\ \text{for all factor modules $X'$ of $X$}\},\\
\FF_\theta&:=\{X\in\mod A\mid\theta(X')<0\ \text{for all submodules $X'\neq 0$ of $X$}\},\\
\overline{\FF}_\theta&:=\{X\in\mod A\mid\theta(X')\le0\ \text{for all submodules $X'$ of $X$}\}.
\end{align*}
\item\cite[Subsection 1.1]{K}
We set $\WW_\theta:=\overline{\TT}_\theta \cap \overline{\FF}_\theta$ and call it the \textit{$\theta$-semistable subcategory}.
\end{enumerate}
\end{definition}

We remark that $\WW_\theta$ is a wide subcategory of $\mod A$ (see \cite{HR}); hence, it is an abelian length category.
The simple objects of $\WW_\theta$ are called \textit{$\theta$-stable} modules.
Thus $X$ is $\theta$-stable if and only if $\theta(X)=0$ and $\theta(X')>0$ for all factor modules $X' \ne 0$ of $X$.
If $X$ is a simple object of $\WW_\theta$, then $X$ is a \textit{brick};
that is, $\End_A(X) \simeq k$.
$\WW_\theta$ satisfies the Jordan-H\"older property, so for each $X \in \WW_\theta$, the \textit{composition factors} of $X$ in $\WW_\theta$ are well-defined.

We frequently use the following easy fact.

\begin{lemma}\label{additivity}
For each $\theta,\eta\in K_0(\proj A)_\R$ and $\epsilon>0$, we have
\[\overline{\TT}_{\epsilon\theta}=\overline{\TT}_\theta,\quad \TT_{\epsilon\theta}=\TT_\theta,\quad \overline{\FF}_{\epsilon\theta}=\overline{\FF}_\theta,\quad \FF_{\epsilon\theta}=\FF_\theta,\]
\[\overline{\TT}_{\eta}\cap\overline{\TT}_{\theta}\subseteq\overline{\TT}_{\eta+\theta},\quad \overline{\TT}_{\eta}\cap\TT_{\theta}\subset\TT_{\eta+\theta},\quad \overline{\FF}_{\eta}\cap\overline{\FF}_{\theta}\subseteq\overline{\FF}_{\eta+\theta},\quad \overline{\FF}_{\eta}\cap\FF_{\theta}\subset\FF_{\eta+\theta}.\]
\end{lemma}

For $\theta,\eta\in K_0(\proj A)_\R$, we write
\begin{align*}\theta\ge\eta\end{align*}
if $\theta-\eta\in \sum_{i=1}^n \R_{\ge 0}[P(i)]$.
The following is clear from Lemma \ref{additivity}.

\begin{lemma}
If $\theta\ge\eta$, then
\begin{align*}\overline{\TT}_\theta\supseteq\overline{\TT}_\eta,\quad\TT_\theta\supseteq\TT_\eta,\quad\overline{\FF}_\theta\subseteq\overline{\FF}_\eta,\quad\FF_\theta\subseteq\FF_\eta.\end{align*}
\end{lemma}

We say that a continuous map $\theta(t) \colon [0,1]\to K_0(\proj A)_\R$ is \textit{decreasing} if $0\notin\theta([0,1])$ and each $t,t'\in[0,1]$ with $t\le t'$ satisfy $\theta(t)\ge\theta(t')$.
By using this, we have the following canonical filtration of each $A$-module.

\begin{definition-proposition}\label{HN filtration}
Let $\theta\in K_0(\proj A)_\R$ and $\theta(t) \colon [0,1]\to K_0(\proj A)_\R$ be a decreasing map with $\theta(0)=\theta$ and $\theta(1)\in\sum_{i=1}^n \R_{<0}[P(i)]$. For each $X\in\overline{\TT}_\theta$, there exists $0\le t_1<t_2<\cdots<t_\ell\le 1$ and a filtration (called \textit{Harder-Narasimhan filtration} of $X$)
\begin{align*}X=X_0\supset X_1\supset\cdots\supset X_\ell=0\end{align*}
such that $X_{i-1}/X_{i}\in\WW_{\theta(t_i)}$ for each $1\le i\le\ell$.
\end{definition-proposition}

\begin{proof}
Assume $X\neq0$ and use the induction on $\dim_kX$. Define a function $f \colon [0,1]\to\R$ by
\begin{align*}f(t):=\min\{\theta(t)(Y)\mid\text{$Y$ is a non-zero factor module of $X$}\}.\end{align*}
Then $f$ is a continuous decreasing function. Since $0\neq X\in\overline{\TT}_\theta$, we have $f(0)\ge0$ and $f(1)<0$. Take $t_1\in[0,1]$ satisfying $f(t_1)=0$. Then there exists a submodule $X_1$ of $X$ such that $\theta(t_1)(X/X_1)=0$ and $X_1\neq X$. Then $X/X_1\in\WW_{\theta(t_1)}$ holds. The induction hypothesis shows that there exists $t_1<t_2<\cdots<t_\ell\le 1$ and $X_1\supset X_2\supset\cdots\supset X_\ell=0$ such that $X_{i-1}/X_{i}\in\WW_{\theta(t_i)}$ for each $2\le i\le\ell$.
Thus the assertion follows.
\end{proof}

Notice that, for given elements $\theta,\eta\in K_0(\proj A)_\R$ and $\eta\in \sum_{i=1}^n \R_{<0}[P(i)]$, Rudakov's Harder-Narashimhan filtration for $(\theta,\eta)$ \cite[Proposition 3.4]{Rudakov} coincides with the one given in Definition-Proposition \ref{HN filtration} for $\theta(t):=(1-t)\theta+t\eta$.

For a subset $I\subset K_0(\proj A)_\R$, let
\begin{align*}
\vecFilt_{\theta\in I} \limits \WW_\theta:=\bigcup_{
\ell\ge0,\ \theta_1<\cdots<\theta_\ell\ {\rm in}\ I
}\WW_{\theta_\ell}*\cdots*\WW_{\theta_2}*\WW_{\theta_1}.
\end{align*}
We immediately obtain the following (cf.\ \cite[Lemma 5.2]{T}).

\begin{proposition}\label{W filtration}
Let $\theta\in K_0(\proj A)_\R$.
\begin{enumerate}[\rm(a)]
\item Let $\theta(t) \colon [0,1]\to K_0(\proj A)_\R$ be a decreasing map with $\theta(0)=\theta$ and $\theta(1)\in \sum_{i=1}^n \R_{<0}[P(i)]$. Then
\begin{align*}\overline{\TT}_\theta=\vecFilt_{\eta\in\theta([0,1])} \limits \WW_{\eta}.\end{align*}
\item For $\theta\in K_0(\proj A)_\Q$, let $K_0(\proj A)_{\Q}^{\le\theta}:=\{\eta\in K_0(\proj A)_\Q\mid \eta\le\theta\}$. Then
\begin{align*}\overline{\TT}_\theta=\vecFilt_{\eta\in K_0(\proj A)_{\Q}^{\le\theta}} \limits \WW_\eta.
\end{align*}
\end{enumerate}
\end{proposition}

\begin{proof}
(a) The assertion is clear from Definition-Proposition \ref{HN filtration}.

(b) Fix $\eta\in \sum_{i=1}^n \Q_{<0}[P(i)]$ and let $\theta(t):=(1-t)\theta+t\eta$. Then the assertion follows from (a). Notice that $t_1$ in the proof of Definition-Proposition \ref{HN filtration} can be taken from $\Q$ since it is a solution of linear equations with rational coefficients.
\end{proof}

As in \cite[Definition 2.13]{A}, 
we consider the TF equivalence class of $\theta$ defined by
\begin{align*}[\theta]_{\rm TF}:=\{\eta\in K_0(\proj A)_\R\mid \overline{\TT}_\theta=\overline{\TT}_{\eta},\ \overline{\FF}_\theta=\overline{\FF}_{\eta}\}.\end{align*}
Its closure has the following description.

\begin{proposition}{\cite[Lemma 2.16]{A}}\label{closure of TF}
For $\theta\in K_0(\proj A)_\R$, we have
\begin{align*}\overline{[\theta]_{\rm TF}}=\{\eta\in K_0(\proj A)_\R\mid\overline{\TT}_\theta\subseteq\overline{\TT}_\eta,\ \overline{\FF}_\theta\subseteq\overline{\FF}_\eta\}.\end{align*}
In particular, $\overline{[\theta]_{\rm TF}}$ is a disjoint union of some TF equivalence classes.
\end{proposition}

Later we will use the following easy observations.
Note that $\epsilon$ below depends on $X$.

\begin{lemma}\label{extend opposite}
Assume that $\eta,\theta\in K_0(\proj A)_\R$ are TF equivalent.
\begin{enumerate}[\rm(a)]
\item Each $X\in\TT_\theta$ belongs to $\TT_{\theta-\epsilon\eta}$ for sufficiently small $\epsilon>0$.
\item Each $X\in\WW_\theta$ belongs to $\WW_{\theta-\epsilon\eta}$ for sufficiently small $\epsilon>0$.
\item Each $X\in\overline{\TT}_\theta$ belongs to $\overline{\TT}_{\theta-\epsilon\eta}$ for sufficiently small $\epsilon>0$.\end{enumerate}
\end{lemma}

\begin{proof}
(a) Each non-zero factor module $Y$ of $X$ satisfies $\theta(Y)>0$ and $\eta(Y)>0$. Since there are only finitely many dimension vectors of factor modules of $X$, we have
\begin{align*}\delta:=\min\{\theta(Y)/\eta(Y)\mid\text{$Y$ is a non-zero factor module of $X$}\}>0.\end{align*}
The desired inequality is satisfied if $\epsilon<\delta$.

(b) We can assume that $X$ is a simple object in $\WW_\theta=\WW_\eta$.
Then $\theta(X)=\eta(X)=0$, and $\theta(Y)>0$ and $\eta(Y)>0$ hold for each non-zero factor module $Y\neq X$ of $X$. Again we have
\begin{align*}\delta:=\min\{\theta(Y)/\eta(Y)\mid\text{$Y\neq X$ is a non-zero factor module of $X$}\}>0,\end{align*}
and the desired inequality is satisfied if $\epsilon<\delta$.

(c) Take an exact sequence $0\to T\to X\to W\to 0$ with $T\in\TT_\theta$ and $W\in\WW_\theta$. By (a) and (b), we have $T\in\TT_{\theta-\epsilon\eta}\subset\overline{\TT}_{\theta-\epsilon\eta}$ and $W\in\WW_{\theta-\epsilon\eta}\subset\overline{\TT}_{\theta-\epsilon\eta}$ for sufficiently small $\epsilon>0$, and hence $X\in\overline{\TT}_{\theta-\epsilon\eta}$.
\end{proof}

One of the systematic ways to obtain TF equivalence classes is to use 2-term presilting complexes.
For any $U=\bigoplus_{i=1}^m U_i \in \twopresilt A$ with $U_i$ indecomposable, we set \textit{cones}
\begin{align*}
C^\circ(U) := \cone^\circ\{[U_1],\ldots,[U_m]\} \subset
C(U) := \cone\{[U_1],\ldots,[U_m]\}.
\end{align*}
In particular, we set
\begin{align*}C^\circ(0)=C(0):=\{0\}.\end{align*}
These cones appear in many papers including \cite{DIJ,Y,BST}.
The following remark is crucial.

\begin{remark}\label{cone basis}\cite[Theorem 2.27, Corollary 2.28]{AI}.
If $U=\bigoplus_{i=1}^m U_i \in \twopresilt A$ with $U_i$ indecomposable
$[U_1],[U_2],\ldots,[U_m] \in K_0(\proj A)$ can be extended to a $\Z$-basis 
of $K_0(\proj A)$.
\end{remark}

Thus the dimensions of $C^\circ(U)$ and $C(U)$ in $K_0(\proj A)_\R$ are both $|U|$. Let
\begin{align*}
\Cone^\circ:=\bigcup_{T\in\twosilt A}C^\circ(T)\subset \Cone:=\bigcup_{T\in\twosilt A}C(T)=\bigcup_{U\in\twopresilt A}C^\circ(U)\subset K_0(\proj A)_\R.
\end{align*}
The first author proved that each $C^\circ(U)$ gives a TF equivalence class
by using \cite[Proposition 3.3]{Y} and \cite[Proposition 3.27]{BST}.

\begin{proposition}\label{cone-TF}\cite[Proposition 3.11]{A}
For any $U \in \twopresilt A$,
the cone $C^\circ(U)$ is a TF equivalence class satisfying
\begin{align*}
C^\circ(U) &= \{ \theta \in K_0(\proj A)_\R \mid 
\overline{\TT}_\theta=\overline{\TT}_U, \ \overline{\FF}_\theta=\overline{\FF}_U \},\\
C(U) &= \{ \theta \in K_0(\proj A)_\R \mid 
\overline{\TT}_\theta\supset\overline{\TT}_U, \ \overline{\FF}_\theta\supset\overline{\FF}_U \}.
\end{align*}
\end{proposition}

In particular, $C^\circ(T)$ for $T \in \twosilt A$ is a \textit{full-dimensinonal} TF equivalence class, that is, $C^\circ(T)$ is a TF equivalence class whose interior is not empty.
Set $\TF_n(A)$ as the set of full-dimensional TF equivalence classes. On these notions, there are the following results.

\begin{proposition}\label{cone-wall}\cite[Theorem 3.17]{A}
The following properties hold.
\begin{enumerate}[\rm(a)]
\item
For $\theta \in K_0(\proj A)$, $\theta \in \Cone^\circ$ if and only if $\WW_\theta=\{0\}$.
\item
For $\theta \in K_0(\proj A)_\R$, $\theta \in \Cone^\circ$ if and only if there exists an open neighborhood $V$ of $\theta$ such that $\WW_{\theta'}=\{0\}$ for all $\theta' \in V$.
\item
There exists a bijection $\twosilt A \to \TF_n(A)$ given by $T \mapsto C^\circ(T)$.
\end{enumerate}
\end{proposition}

\subsection{Wall-chamber structures}

In \cite{BST} and \cite{Bridgeland},
they defined a wall-chamber structure on $K_0(\proj A)_\R$
by using $\theta$-semistable subcategories $\WW_\theta$.

\begin{definition}\cite[Definition 3.2]{BST}\cite[Definition 6.1]{Bridgeland}
Let $X \in \mod A$ be a non-zero module.
Then we call
\begin{align*}
\Theta_X:=\{\theta\in K_0(\proj A)_\R\mid X\in\WW_\theta\}
\end{align*}
the \textit{wall} associated to $X$.
By considering the walls $\Theta_X$ for all non-zero modules,
we define a \textit{wall-chamber structure} on $K_0(\proj A)_\R$.
\end{definition}

Since there are only finitely many dimension vectors of factor modules of $X$, $\Theta_X$ is a rational polyhedral cone in the Euclidean space $K_0(\proj A)_\R$. 
The \emph{dimension} $\dim C$ of a cone $C$ in $\R^n$ is the dimension of the subspace generated by $C$. The \emph{codimension} of $C$ is $n-\dim C$. A convex subset $C \subset \R^n$ is called \textit{strongly convex} if $C \cap (-C)=\{0\}$.

The following basic properties are useful. 

\begin{lemma}\label{simple and interior}
Let $X\in\mod A$ and $\theta\in K_0(\proj A)_\R$.
\begin{enumerate}[\rm(a)]
\item \cite[Lemma 2.5]{A} $\Theta_X$ is strongly convex if and only if $X$ is sincere.
\item \cite[Lemma 2.2]{A} Assume $X\in\WW_\theta$. Then $X$ is a simple object in $\WW_\theta$ if and only if $\theta(Y)>0$ holds for each non-zero proper factor module of $X$. 
\item \cite[Lemma 2.7]{A} If $X$ is a simple object in $\WW_\theta$, then $\dim\Theta_X=|A|-1$ and $\theta\in\Theta_X^\circ$ hold. 
\item \cite[Lemma 2.7]{A} Assume $X \in \WW_\theta$. Then $\theta\in F^\circ$ holds for some face $F$ of $\Theta_X$ with $\dim F=|A|-\dim_\R W_{\theta,X}$, where
\begin{align*}
W_{\theta,X}:=\langle[S] \mid 
\text{$S$ is a composition factor of $X$ in $\WW_\theta$}\rangle_\R\subset K_0(\mod A)_\R.
\end{align*}
\item Assume $\dim\Theta_X=|A|-1$. Then $\Theta_X^\circ$ consists of all $\theta\in\Theta_X$ such that $\theta(Y)>0$ holds for each factor module $Y$ of $X$ satisfying $\dimv Y\notin\R\dimv X$.
\end{enumerate}
\end{lemma}

\begin{proof}
(e) We need the following basic fact: For a finite dimensional $\R$-vector space $V$ and non-zero $\R$-linear forms $d_1,\ldots,d_m \colon V \to \R$, let $H_i^{\ge0}:=\{x\in V\mid d_i(x)\ge0\}\supset H_i^{>0}:=\{x\in V\mid d_i(x)>0\}$ and $C:=\bigcap_{i=1}^mH_i^{\ge0}$. If $\R C=V$ holds, then we have
\begin{equation}\label{X^circ}
C^\circ=\bigcap_{i=1}^mH_i^{>0}.
\end{equation}
Now let $V:=\Kernel\langle-,X\rangle\subset K_0(\proj A)_\R$ and $d_1,\ldots,d_m$ the dimension vectors of the factor modules of $X$ which does not belong to $\R\dimv X$. Then \eqref{X^circ} shows the assertion since $C=\Theta_X$.
\end{proof}

There may be some inclusions $\Theta_X \subset \Theta_Y$ for $X,Y \in \mod A$;
for example, $\Theta_{X \oplus X'}=\Theta_X \cap \Theta_{X'}$.
Thus some walls $\Theta_X$ are redundant.
Actually it is enough to consider bricks to obtain the wall-chamber structure.

\begin{proposition}\label{maximal wall}\cite[Proposition 2.8]{A}
Let $X \in \mod A$.
Take $\theta \in \Theta_X^\circ$ and $S \in \WW_\theta$ such that $S$ is a composition factor of $M$ in the abelian length category $\WW_\theta$.
Then $\Theta_S \supseteq \Theta_X$ and the codimension of $\Theta_S$ is one.
\end{proposition}

Let $\Wall$ be the union of all walls;
\begin{align*}
\Wall&:=\bigcup_{0\neq X\in\mod A}\Theta_X=\{\theta\in K_0(\proj A)_\R\mid\WW_\theta\neq0\}\ \text{ and}\\
\brick_{\rm s}A&:=\{X\in\mod A\mid\text{$X$ is $\theta$-stable for some $\theta\in K_0(\proj A)_\R$}\}.
\end{align*}
Then Proposition \ref{maximal wall} implies
\begin{align*}
\Wall=\bigcup_{X\in\brick_{\rm s}A}\Theta_X.
\end{align*}
Moreover, Proposition \ref{cone-wall} is rewritten as 
\begin{align*}
K_0(\proj A)\subset\Cone^\circ\sqcup\Wall\ \text{ and }\ K_0(\proj A)_\R=\Cone^\circ\sqcup\overline{\Wall}.
\end{align*}

The wall-chamber structure and the TF equivalence classes are related as follows.
For $\theta,\theta' \in K_0(\proj A)_\R$, we set
\begin{align*}
[\theta,\theta']:=\{(1-r)\theta+r\theta' \mid r \in [0,1]\}.
\end{align*}

\begin{proposition}\label{Asai TF}\cite[Theorem 2.17]{A}
For $\theta \ne \theta' \in K_0(\proj A)_\R$, the following are equivalent.
\begin{enumerate}[\rm(a)]
\item $\theta$ and $\theta'$ are TF equivalent.
\item For any $\theta'' \in [\theta,\theta']$, $\WW_\theta''$ is constant.
\item There exists no brick $S$ such that $\Theta_S \cap [\theta,\theta']$ is one point.
\end{enumerate}
\end{proposition}

\subsection{Canonical decompositions}

Any element $\theta\in K_0(\proj A)$ can be written uniquely as
\begin{align*}\theta=P_0^\theta-P_1^\theta\end{align*}
for some $P_0^\theta,P_1^\theta\in\proj A$ which do not have non-zero common direct summands. Following \cite{DF}, we write
\begin{align*}\Hom(\theta):=\Hom_A(P_1^\theta,P_0^\theta)\end{align*}
and call it the \textit{presentation space} of $\theta$.
Clearly, $\Hom(\theta)$ is an irreducible algebraic variety, so we consider the Zariski topology there.

For each morphism $f \colon P_1\to P_0$ in $\proj A$,
we set $P_f$ as the 2-term complex given by $f$:
\begin{align*}P_f:=(P_1\xrightarrow{f} P_0).\end{align*}
Also we write
\begin{align*}[f]=[P_f]=[P_0]-[P_1]\in K_0(\proj A).\end{align*}

\begin{definition}\cite[Definition 4.3]{DF}
\begin{enumerate}[\rm(a)]
\item For $\theta_1,\ldots,\theta_\ell\in K_0(\proj A)$, we write $\theta_1\oplus\cdots\oplus\theta_\ell$ if for each general element in $f\in\Hom(\theta_1+\cdots+\theta_\ell)$, there exist $f_1,\ldots,f_\ell\in\Hom(\theta_i)$ such that $P_f\simeq P_{f_1}\oplus\cdots\oplus P_{f_\ell}$ as complexes.
\item
Let $\theta\in K_0(\proj A)$, 
then $\theta$ is said to be \textit{indecomposable} in $K_0(\proj A)$
if $P_f$ is indecomposable for each general element in $f\in\Hom(\theta)$.
\item
We call $\theta_1\oplus\cdots\oplus\theta_\ell$ \textit{a canonical decomposition} 
(of $\theta_1+\cdots+\theta_\ell$)
if all $\theta_i$ are indecomposable in $K_0(\proj A)$. 
\end{enumerate}
\end{definition}

If $\theta=\theta_1+\cdots+\theta_\ell$ and $\theta_1\oplus\cdots\oplus\theta_\ell$,
we write $\theta=\theta_1\oplus\cdots\oplus\theta_\ell$.
We remark that $\theta \in K_0(\proj A)$ is indecomposable if and only if 
$\theta \ne 0$ and $\theta=\theta_1 \oplus \theta_2$ implies $\theta_1=0$ or $\theta_2=0$.
Moreover if $\theta=\theta_1\oplus\cdots\oplus\theta_\ell$ is a canonical decomposition,
then for general $f \in \Hom(\theta)$, 
$P_f$ is isomorphic to $P_{f_1}\oplus\cdots\oplus P_{f_\ell}$ 
with each $P_{f_i}$ is indecomposable. 

The following invariant is useful to understand canonical decompositions, and was originally introduced in \cite{DWZ} for Jacobian algebras of quivers with potential.

\begin{definition}
For morphisms $f \colon P_1\to P_0$ and $g \colon Q_1\to Q_0$ in $\proj A$, 
let
\begin{align*}E(f,g):=\dim_k\Hom_{\KKK^{\bo}(\proj A)}(P_f,P_g[1])\end{align*}
For $\eta,\theta\in K_0(\proj A)$, let
\begin{align*}E(\eta,\theta):=\min\{E(f,g)\mid (f,g)\in\Hom(\eta)\times\Hom(\theta)\}.\end{align*}
Clearly the map
\begin{align*}E \colon K_0(\proj A)\times K_0(\proj A)\to\Z\end{align*}
is subadditive for both entries.
\end{definition}

Notice that $E(-,-)$ is not symmetric, even $E(\eta,\theta)=0$ does not imply $E(\theta,\eta)=0$ in general.

For morphisms $f \colon P_1\to P_0$ and $g \colon Q_1\to Q_0$ in $\proj A$, we have
\begin{equation}\label{Ext 1}
\Hom_{\KKK^{\bo}(\proj A)}(P_f,P_g[1])=\Cokernel(\Hom_A(P_1,Q_1)\oplus\Hom_A(P_0,Q_0)\xrightarrow{(g\circ-\ -\circ f)}\Hom_A(P_1,Q_0)).
\end{equation}
We obtain the following basic observation.

\begin{proposition}\cite[Section 3]{DF}\label{upper semi}
Let $\eta=[P_0]-[P_1]$ and $\theta=[Q_0]-[Q_1]$. The map
\begin{align*}E(-,-) \colon \Hom(\eta)\times\Hom(\theta)\to\Z\end{align*}
is upper semi-continuous, and the subset $\{(f,g)\in\Hom(\eta)\times\Hom(\theta)\mid E(f,g)=E(\eta,\theta)\}$ is open dense in $\Hom(\eta)\times\Hom(\theta)$.
\end{proposition}

By using $E(\eta,\theta)$, canonical decompositions are characterised as in (a) below. In particular, the existence of canonical decompositions is guaranteed.

\begin{proposition}\label{decomposition}
The following assertions hold.
\begin{enumerate}[\rm(a)] 
\item \cite[Theorem 4.4]{DF} Let $\theta_1,\theta_2,\ldots,\theta_\ell \in K_0(\proj A)$, then
$\theta_1\oplus\cdots\oplus\theta_\ell$ holds
if and only if $E(\theta_i,\theta_j)=0$ for all $i\neq j$.
\item \cite{DF} For any $\theta \in K_0(\proj A)$, there exists a unique canonical decomposition $\theta_1\oplus\cdots\oplus\theta_\ell$ of $\theta$ up to reordering.
\end{enumerate}
\end{proposition}

\begin{proof}
(b) The existence follows from (a), and the uniqueness is clear.
\end{proof}

We immediately have the following properties.

\begin{proposition}\label{decompose_2step}
Let $\eta,\theta_1,\ldots,\theta_\ell \in K_0(\proj A)$ and $\theta:=\theta_1+\cdots+\theta_\ell$.
\begin{enumerate}[\rm(a)]
\item The condition $\theta_1\oplus\cdots\oplus\theta_\ell$ holds if and only if $\theta_i\oplus\theta_j$ holds for each $1\le i\neq j\le\ell$.
\item $E(\eta,\theta)\le E(\eta,\theta_1)+\cdots+E(\eta,\theta_\ell)$ and $E(\theta,\eta)\le E(\theta_1,\eta)+\cdots+E(\theta_\ell,\eta)$ hold.
Both equalities hold if $\theta_1 \oplus\cdots\oplus \theta_\ell$.
\item If $\eta\oplus\theta_i$ for each $1\le i\le\ell$, then $\eta\oplus\theta$. The converse holds if $\theta_i\oplus\theta_j$ for each $1\le i\neq j\le\ell$.
\item Assume $\bigoplus_{i=1}^\ell\theta_i$ and $\theta_i=\bigoplus_{j=1}^{\ell_i}\theta_{ij}$. Then $\bigoplus_{1\le i\le\ell,\ 1\le j\le\ell_i}\theta_{ij}$ holds, that is, $\theta_{ij}\oplus\theta_{i'j'}$ holds for each $(i,j)\neq(i',j')$.
\item For each $m\ge1$, we have $\cone(\ind \theta)\subseteq\cone(\ind m\theta)$.
\end{enumerate}
\end{proposition}

\begin{proof}
(a) is immediate from Proposition \ref{decomposition}(a).

(b) 
The first statement is immediate from definition.
To show the second one, assume $\theta_1 \oplus\cdots\oplus \theta_\ell$.
We only prove $E(\eta,\theta)\ge \sum_{i=1}^\ell E(\eta,\theta_i)$.
Since $\theta=\theta_1\oplus\cdots\oplus \theta_\ell$, 
\begin{align*}
X:=\left\{ g \in \Hom(\theta) \ \middle|\ 
\text{there exists $(g_i)_{i=1}^\ell \in \prod_{i=1}^\ell\Hom(\theta_i)$ such that
$P_g \simeq\bigoplus_{i=1}^\ell P_{g_i}$} \right\}
\end{align*}
is an open dense subset of $\Hom(\theta)$. By Proposition \ref{upper semi}, the subset 
\begin{align*}
Y:=\left\{(f,g)\in\Hom(\eta)\times\Hom(\theta)\ \middle|\  E(f,g)=E(\eta,\theta)\right\}
\end{align*}
is open dense in $\Hom(\eta)\times\Hom(\theta)$. Let $\pi_2:\Hom(\eta)\times\Hom(\theta)\to\Hom(\theta)$ be the projection to the second entry. Take a point $(f,g)$ in an open dense subset $\pi_2^{-1}(X)\cap Y$, and $(g_i)_{i=1}^\ell\in \prod_{i=1}^\ell\Hom(\theta_i)$ such that $P_g\simeq \bigoplus_{i=1}^\ell P_{g_i}$.
Then we obtain the desired inequality
\begin{align*}
E(\eta,\theta)=E(f,g)=\sum_{i=1}^\ell E(f,g_i) \ge\sum_{i=1}^\ell E(\eta,\theta_i).
\end{align*}

(c)
We prove the first statement. By (b), we have $E(\eta,\theta)\le\sum_{i=1}^\ell E(\eta,\theta_i)$, which is zero by Proposition \ref{decomposition}(a). Thus $E(\eta,\theta)=0$ holds, and dually $E(\theta,\eta)=0$ holds. By Proposition \ref{decomposition}(a) again, we obtain $\eta\oplus\theta$.

To show the second one, assume $\eta\oplus\theta$ and $\theta_i \oplus \theta_j$ for each $1 \le i \ne j \le \ell$.
By Proposition \ref{decomposition}(a) and the second statement of (b), we have
$0=E(\eta,\theta)=\sum_{i=1}^\ell E(\eta,\theta_i)$.
Thus $E(\eta,\theta_i)=0$ for each $i$.
Dually $E(\theta_i,\eta)=0$ for each $i$.
Thus Proposition \ref{decomposition}(a) gives $\eta \oplus \theta_i$.

(d) is immediate from (c).

(e) Let $\theta=\bigoplus_{i=1}^\ell\theta_i$ be a canonical decomposition. It suffices to show $\theta_i\in\cone(\ind m\theta)$ for each $i$.
We have $m\theta=\bigoplus_{i=1}^\ell m\theta_i$ by (a) and the first statement of (c).
Let $m\theta_i=\bigoplus_{j=1}^{\ell_i}\theta_{ij}$ be a canonical decomposition. 
Then $m\theta=\bigoplus_{i=1}^\ell\bigoplus_{j=1}^{\ell_i}\theta_{ij}$ is a canonical decomposition by (d).
Therefore each $\theta_{ij}$ 
belong to $\ind(m \theta)$, so we have $ \theta_i=m^{-1}\sum_{j=1}^{\ell_i} \theta_{i\ell_i} \in \cone(\ind m \theta)$, as desired.
\end{proof}

We also need the following observation.

\begin{proposition}\cite[Lemma 2.16]{P}\label{codim=E}
For $P_0,P_1\in\proj A$, let $G=\Aut_A(P_1)\times\Aut_A(P_0)$. 
For $f\in\Hom_A(P_1,P_0)$, the codimension of $Gf$ in $\Hom_A(P_1,P_0)$ is $E(f,f)$.
\end{proposition}

Following \cite{DF}, we introduce the next notions, where we do \textit{not} assume that $\theta$ is indecomposable.

\begin{definition}\cite[Definition 4.6]{DF}
Let $\theta \in K_0(\proj A)$.
\begin{enumerate}[\rm(a)]
\item \begin{enumerate}[\rm(i)]
\item
$\theta$ is said to be \textit{rigid} if there exists $f \in \Hom(\theta)$ such that $\Hom_{\KKK^{\bo}(\proj A)}(P_f,P_f[1])=0$.
\item
$\theta$ is said to be \textit{tame} if $E(\theta,\theta)=0$.
\item
$\theta$ is said to be \textit{wild} if $E(\theta,\theta) \ne 0$.
\end{enumerate}
\item $\theta$ is said to be \textit{positive} if $P_0^\theta \ne 0$ and $P_1^\theta=0$, and \textit{negative} if $P_0^\theta=0$ and $P_1^\theta\ne0$.
\end{enumerate}
\end{definition}

Therefore $\theta \in K_0(\proj A)$ is rigid if and only if there exists $f \in \Hom(\theta)$ such that $P_f$ is presilting, and $\theta$ is tame if and only if $\theta \oplus \theta$ holds.

Typical examples of direct sums in $K_0(\proj A)$ are given in silting theory.

\begin{example}\label{presilting case}
For $f \in \Hom(\theta)$, $P_f$ is presilting if and only if the orbit of $f$ 
with respect to the action of the group $\Aut(P_1^\theta) \times \Aut(P_0^\theta)$
on $\Hom(\theta)$ is dense by Proposition \ref{codim=E}.
In this case, $P_f \simeq P_{f'}$ holds for each general $f' \in \Hom(\theta)$. 

Let $U=U_1\oplus\cdots\oplus U_\ell\in\twopresilt A$ with $U_i$ indecomposable. 
From the previous paragraph, we have a canonical decomposition $[U]=[U_1] \oplus \cdots \oplus [U_\ell]$ in $K_0(\proj A)$ and the obvious equalities
\begin{align*}\overline{\TT}_U=\bigcap_{i=1}^\ell\overline{\TT}_{U_i}\ \text{ and }\ \TT_U=\bigvee_{i=1}^\ell\TT_{U_i}.\end{align*}
By Remark \ref{cone basis}, any $\eta \in C(U) \cap K_0(\proj A)$ has a canonical decomposition of the form
$\theta=[U_1]^{\oplus s_1} \oplus \cdots \oplus [U_\ell]^{\oplus s_\ell}$ with $s_i \in \Z_{\ge 0}$.
\end{example}

\begin{lemma}\label{silting direct sum}
Let $U_1,U_2$ be 2-term presilting complexes 
in $\KKK^{\bo}(\proj A)$.
Then $U_1 \oplus U_2$ is 2-term presilting 
if and only if $[U_1] \oplus [U_2]$ in $K_0(\proj A)$.
\end{lemma}

\begin{proof}
The ``only if'' part follows from Example \ref{presilting case}.

For the ``if'' part, we have $\Hom_{\KKK^{\bo}(\proj A)}(P_f,P_g[1])=0$ and $\Hom_{\KKK^{\bo}(\proj A)}(P_g,P_f[1])=0$ for any general 
$(f,g) \in \Hom([U_1]) \times \Hom([U_2])$
from Propositions \ref{upper semi} and \ref{decomposition}.
By Example \ref{presilting case}, we may assume that $P_f \simeq U_1$ and $P_g \simeq U_2$.
Then $U_1 \oplus U_2$ is 2-term presilting.
\end{proof}

\section{Morphism torsion pairs and semistable torsion pairs}

The aim of this section is to introduce a class of torsion classes called morphism torsion classes and observe their basic properties.
This class contains the functorially finite torsion classes.
Throughout this section, $A$ is a finite dimensional algebra over an algebraically closed field $k$.

\subsection{Basic properties}

We first define morphism torsion pairs as follows as a generalization of \eqref{define T_U}.

\begin{definition}
For a morphism $f$ in $\proj A$, let
$C_f:=\Cokernel f$ and $K_{\nu f}:=\Kernel \nu f$.
We define \textit{morphism torsion classes}
\begin{align*}\TT_f:=\T(C_f)\ \text{ and }\ \overline{\TT}_f:={^\perp K_{\nu f}}\end{align*}
and \textit{morphism torsion-free classes}
\begin{align*}\FF_f:=\F(K_{\nu f})\ \text{ and }\ \overline{\FF}_f:={C_f}^\perp.\end{align*}
Clearly they give two torsion pairs
\begin{align*}(\overline{\TT}_f,\FF_f)\ \text{ and }\ (\TT_f,\overline{\FF}_f)\end{align*}
called \textit{morphism torsion pairs}. 
We also set $\WW_f:=\overline{\TT}_f\cap\overline{\FF}_f.$
\end{definition}

We will later show that $\WW_f$ is always a wide subcategory of $\mod A$ in Proposition \ref{f category}.

\begin{remark}
In contrary to semistable torsion classes, the inclusions $\TT_f\subseteq\overline{\TT}_f$ do not necessarily hold.
Actually, $\TT_f \subset \overline{\TT}_f$ holds if and only $P_f$ is 2-term presilting; see Proposition \ref{P Q[1]}.
\end{remark}

By definition, for a direct sum $f\oplus g$ of morphisms $f,g$ in $\proj A$, we have
\begin{align*}
\TT_{f\oplus g}=\TT_f\vee\TT_g,\quad\FF_{f\oplus g}=\FF_f\vee\FF_g,&\\
\overline{\TT}_{f\oplus g}=\overline{\TT}_f\cap\overline{\TT}_g,\quad\overline{\FF}_{f\oplus g}=\overline{\FF}_f\cap\overline{\FF}_g,&\quad\WW_{f\oplus g}=\WW_f\cap\WW_g.
\end{align*}

To understand morphism torsion pairs, the Nakayama functor is useful.

\begin{lemma}\label{Nakayama T_f}
Let $f$ be a morphism in $\proj A$, and $X \in \mod A$. Consider the homomorphism
\begin{align*}
\Hom_A(f,X) \colon \Hom_A(P_0,X) \to \Hom_A(P_1,X).
\end{align*}
\begin{enumerate}[\rm(a)]
\item
There exist isomorphisms
\begin{align*}
\Cokernel \Hom_A(f,X) &\simeq \Hom_{\DDD(A)}(P_f,X[1]) 
\simeq D \Hom_A(X,K_{\nu f}),\\
\Kernel \Hom_A(f,X) &\simeq \Hom_{\DDD(A)}(P_f,X)
\simeq \Hom_A(C_f,X).
\end{align*}
\item We have
\begin{align*}
\overline{\TT}_f&=\{X\in\mod A\mid \textup{$\Hom_A(f,X)$ is surjective}\}, \\
\overline{\FF}_f&=\{X\in\mod A\mid \textup{$\Hom_A(f,X)$ is injective}\}, \\
\WW_f&=\{X\in\mod A\mid \textup{$\Hom_A(f,X)$ is isomorphic}\}.
\end{align*}
\item
For $\theta:=[f] \in K_0(\proj A)$, we have 
\begin{align*}
\theta(X)&=\dim_k \Hom_{\DDD(A)}(P_f,X) - \dim_k \Hom_{\DDD(A)}(P_f,X[1])\\
&=\dim_k \Hom_A(C_f ,X) - \dim_k \Hom_A(X,K_{\nu f}).
\end{align*}
Thus if $\theta(X)=0$, then $X \in \WW_f$ is equivalent to $X \in \overline{\TT}_f$ and also to $X \in \overline{\FF}_f$.
\end{enumerate}
\end{lemma}

\begin{proof}
(a) The first isomorphisms follow from
\begin{align*}
\Cokernel \Hom_A(f,X) &\simeq \Hom_{\DDD(A)}(P_f,X[1]) \simeq D \Hom_{\DDD(A)}(X,\nu P_f[-1]) \simeq D \Hom_A(X,K_{\nu f}).
\end{align*}
The second isomorphisms are immediate.

(b) and  (c) follow from (a).
\end{proof}

We give an example coming from silting theory.

\begin{example}
Let $U=(P_1\xrightarrow{f}P_0)\in\twopresilt A$. 
Then $C_f=H^0(U)$ and $K_{\nu f}=H^{-1}(\nu U)$.
Thus
\begin{align*}\overline{\TT}_f={}^\perp H^{-1}(\nu U),\  
\overline{\FF}_f=H^0(U)^\perp\ \text{ and }\ 
\WW_f={}^\perp H^{-1}(\nu U)\cap H^0(U)^\perp.\end{align*}
Moreover, if $P_f=P_{f_1}\oplus\cdots\oplus P_{f_\ell}$ with indecomposable $P_{f_i}$, we have
\begin{align*}\overline{\TT}_f=\overline{\TT}_\theta,\ \overline{\FF}_f=\overline{\FF}_\theta,\ \WW_f=\WW_\theta\end{align*}
for all $\theta\in\cone^\circ\{[f_1],\ldots,[f_\ell]\}$ by Proposition \ref{cone-TF}.
\end{example}

We also remark that rigid elements are characterized as follows.

\begin{remark}
Let $\theta \in K_0(\proj A)$.
Then $\theta$ is rigid
if and only if there exists $f \in \Hom(\theta)$ satisfying $\TT_f=\TT_\theta$ and $\FF_f=\FF_\theta$.
\end{remark}

\begin{proof}
The ``only if'' part follows from Proposition \ref{cone-TF}.
For the ``if'' part, $C_f \in \TT_\theta$ and $K_{\nu f} \in \FF_\theta$ implies that
$\Hom_{\KKK^{\bo}(\proj A)}(P_f,\nu P_f[-1])=0$, 
which means that $P_f$ is a presilting complex.
Thus $\theta$ is rigid.
\end{proof}

We have the following relationship between morphism torsion pairs and semistable torsion pairs.

\begin{proposition}\label{W_f and W_theta}
For a morphism $f \colon P_1\to P_0$ in $\proj A$ and $\theta=[f]$, we have
\begin{align*}\TT_f\supseteq\TT_\theta, \quad
\FF_f\supseteq\FF_\theta, \quad
\overline{\TT}_f\subseteq\overline{\TT}_\theta, \quad
\overline{\FF}_f\subseteq\overline{\FF}_\theta, \quad
\WW_f\subseteq\WW_\theta,\ .\end{align*}
\end{proposition}

\begin{proof}
We prove $\overline{\TT}_f\subseteq\overline{\TT}_\theta$. Let $X\in\overline{\TT}_f$. Then any factor module $Y$ of $X$ belongs to $\overline{\TT}_f$. Thus the map $\Hom_A(P_0,Y)\to\Hom_A(P_1,Y)$ is surjective by Lemma \ref{Nakayama T_f}, and hence $\theta(Y)=\dim_k\Hom_A(P_0,Y)-\dim_k\Hom_A(P_1,Y)\ge0$ holds. Thus $X\in\overline{\TT}_\theta$.

By definition, we have $\FF_f=(\overline{\TT}_f)^\perp\supseteq(\overline{\TT}_\theta)^\perp=\FF_\theta$.
The dual argument shows $\overline{\FF}_f\subseteq\overline{\FF}_\theta$ and $\TT_f\supseteq\TT_\theta$.
Consequently, $\WW_f=\overline{\TT}_f\cap\overline{\FF}_f\subseteq\overline{\TT}_\theta\cap\overline{\FF}_\theta=\WW_\theta$ hold.
\end{proof}

Now we can prove that $\WW_f$ is a wide subcategory of $\mod A$.

\begin{proposition}\label{f category}
For any morphism $f \colon P_1 \to P_0$ in $\proj A$, we have $\WW_f\in\wide A$.
\end{proposition}

\begin{proof}
Note first that $\WW_f$ is closed under taking images.
Thus it is enough to show that $\WW_f$ satisfies the 2-out-of-3 property for short exact sequences in $\mod A$.

Let $0\to X\to Y\to Z\to0$ be an exact sequence in $\mod A$. We have a commutative diagram
\begin{align*}\xymatrix@R1em{
0\ar[r]&\Hom_A(P_0,X)\ar[r]\ar[d]&\Hom_A(P_0,Y)\ar[r]\ar[d]&\Hom_A(P_0,Z)\ar[r]\ar[d]&0\\
0\ar[r]&\Hom_A(P_1,X)\ar[r]&\Hom_A(P_1,Y)\ar[r]&\Hom_A(P_1,Z)\ar[r]&0
}.\end{align*}
If two of $X,Y,Z$ are in $\WW_f$, then the corresponding two vertical maps are isomorphic by Lemma \ref{Nakayama T_f},
so the other vertical map is also isomorphic, which means the remaining one of $X,Y,Z$ also belongs to $\WW_f$ by Lemma \ref{Nakayama T_f} again.
\end{proof}

We also have the following properties on $\WW_f$.

\begin{lemma}\label{W_f and W_theta 2}
Let $f$ be a morphism in $\proj A$ and $\theta=[f]$.
\begin{enumerate}[\rm(a)]
\item
$\overline{\FF}_f\cap\WW_\theta=\WW_f=\overline{\TT}_f\cap\WW_\theta$.
\item
$\WW_f$ is a Serre subcategory of $\WW_\theta$.
\item
If $C_f\simeq K_{\nu f}$ are isomorphic bricks, then $C_f$ is a simple object of $\WW_\theta$. 
\end{enumerate}
\end{lemma}

\begin{proof}
(a) We only prove the first equality. The inclusion ``$\supseteq$'' follows from Proposition \ref{W_f and W_theta}. Take $X\in\overline{\FF}_f\cap\WW_\theta$. Then by Lemma \ref{Nakayama T_f}(c),
\begin{align*}0\stackrel{X\in\WW_\theta}{=}\theta(X)=\dim_k \Hom_A(C_f,X) - \dim_k \Hom_A(X,K_{\nu f}).\end{align*}
Since $\Hom_A(C_f,X)=0$ by $X\in\overline{\FF}_f$, 
we have $\Hom_A(X,K_{\nu f})=0$. Thus $X\in\overline{\FF}_f \cap \overline{\TT}_f = \WW_f$.

(b) Let $0\to X\to Y\to Z\to0$ be an exact sequence in $\WW_\theta$. It suffices to show that $Y\in\WW_f$ implies $X\in\WW_f$. Since $Y\in\WW_f\subseteq\overline{\FF}_f$, we have $X\in\overline{\FF}_f$. Thus $X\in\WW_f$ by (a).

(c) For any non-zero proper factor module $X$ of $C_f$, 
we have $\Hom_A(X,K_{\nu f})=\Hom_A(X,C_f)=0$ since $C_f$ is a brick. Then $\theta(X)>0$ by Lemma \ref{Nakayama T_f}(c).
By Lemma \ref{simple and interior}(b), the assertion holds.
\end{proof}

Now we fix $P_0,P_1\in\proj A$ and consider $\WW_f$ for each $f\in\Hom_A(P_1,P_0)$. For the morphism $0 \colon P_1 \to P_0$, we obtain
$\WW_0=(P_0\oplus P_1)^\perp$, which is contained in each $\WW_f$. In fact, we show that $\WW_f$ is bigger if $f$ is more general in $\Hom_A(P_1,P_0)$.

\begin{proposition}\label{closure and T}
For $f,g\in\Hom_A(P_1,P_0)$, assume that $g$ is contained in the Zariski closure of $Gf$, where $G=\Aut_A(P_0)\times\Aut_A(P_1)$. Then we have
\begin{align*}\TT_f\subseteq\TT_g,\quad\FF_f\subseteq\FF_g,\quad\overline{\TT}_f\supseteq\overline{\TT}_g,\quad\overline{\FF}_f\supseteq\overline{\FF}_g,\quad\WW_f\supseteq\WW_g.\end{align*}
\end{proposition}

To prove this, we need the following upper semi-continuous condition obtained similarly to Proposition \ref{upper semi}.

\begin{lemma}\label{open}
For $P_0,P_1\in\proj A$ and $X \in \mod A$, the following subsets of $\Hom_A(P_1,P_0)$ are open:
\begin{align*}
\{ f \in \Hom_A(P_1,P_0) \mid X\in\overline{\TT}_f\} \
\text{and} \ \{ f \in \Hom_A(P_1,P_0) \mid X\in\overline{\FF}_f\}.
\end{align*}
\end{lemma}

\begin{proof}
We prove the assertion for $\overline{\TT}$. Consider the natural map
\begin{align*}F \colon \Hom_A(P_1,P_0)\to H:=\Hom_k(\Hom_A(P_0,X),\Hom_A(P_1,X)),\end{align*}
and let $U$ be the subset of $H$ consisting of all surjections. Then $U$ is an open subset of $H$, and hence $F^{-1}(U)=\{ f \in \Hom_A(P_1,P_0) \mid X\in\overline{\TT}_f\}$ is an open subset of $\Hom_A(P_1,P_0)$.
\end{proof}

We are ready to prove Proposition \ref{closure and T}.

\begin{proof}[Proof of Proposition \ref{closure and T}]
We prove the assertion for $\overline{\TT}$. For $X\in\overline{\TT}_g$, let $U:=\{h\in\Hom_A(P_1,P_0)\mid
X\in\overline{\TT}_h\}$. By Lemma \ref{open}, $U$ is an open subset of $\Hom_A(P_1,P_0)$ containing $g$.
Since $g$ belongs to the Zariski closure of $Gf$, we have $Gf\cap U\neq\emptyset$. Thus $X\in\overline{\TT}_f$.

The assertion for $\overline{\FF}$ is shown similarly, and the remaining assertions follow.
\end{proof}

\subsection{TF equivalence classes and canonical decompositions}

In this subsection, we consider the relationship between TF equivalence classes and canonical decompositions.
By Serre duality, we have
\begin{align}\label{Serre duality}
\Hom_{\DDD(A)}(P_{f},P_{g}[1]) \simeq
D \Hom_{\DDD(A)}(P_{g}[1],\nu P_{f}) \simeq
\Hom_A(C_{g}, K_{\nu f}).
\end{align}
Thus we have the following observation, which will be used frequently.

\begin{proposition}\label{P Q[1]}
For morphisms $f$ and $g$ in $\proj A$, the following conditions are equivalent:
\begin{enumerate}[\rm(a)]
\item $E(f,g)=0$, that is, $\Hom_{\DDD(A)}(P_{f},P_{g}[1])=0$.
\item $\Hom_A(C_{g}, K_{\nu f})=0$,
\item $\FF_{f} \subseteq \overline{\FF}_{g}$ (or equivalently, $K_{\nu f}\in\overline{\FF}_{g}$),
\item $\TT_{g} \subseteq \overline{\TT}_{f}$ (or equivalently, $C_g\in\overline{\TT}_{f}$).
\end{enumerate}
Moreover, if $E(\eta,\theta)=0$, these conditions hold for each general pair 
$(f,g) \in \Hom(\eta)\times\Hom(\theta)$.
\end{proposition}

In particular, $P_f$ is presilting if and only if $\TT_f\subseteq\overline{\TT}_f$ if and only if $\FF_f\subseteq\overline{\FF}_f$.

\begin{proof}
(a) and (b) are equivalent by \eqref{Serre duality}.
(b), (c) and (d) are equivalent by definition. The last assertion follows from Proposition \ref{upper semi}.
\end{proof}

Now we prove the following key property, where the part (b) is a generalization of Lemma \ref{func fin presilt}(b).

\begin{proposition}\label{signs of X}
Let $\eta,\theta\in K_0(\proj A)$.
\begin{enumerate}[\rm(a)]
\item If $E(\eta,\theta)=0$, then
\begin{align*}\TT_\theta\subseteq\overline{\TT}_\eta\ \text{ and }\ \FF_\eta\subseteq\overline{\FF}_\theta.\end{align*}
\item If $\eta\oplus\theta$, then
\begin{align*}\TT_\theta\subseteq\overline{\TT}_\eta,\ \TT_\eta\subseteq\overline{\TT}_\theta,\ 
\FF_\theta\subseteq\overline{\FF}_\eta\ \text{ and }\ \FF_\eta\subseteq\overline{\FF}_\theta.
\end{align*}
\end{enumerate}
\end{proposition}

\begin{proof}
(a) There exist $f\in\Hom(\eta)$ and $g\in\Hom(\theta)$ such that $\Hom_{\DDD(A)}(P_f,P_g[1])=0$. Then
\begin{align*}
\TT_\theta\stackrel{\text{Prop. \ref{W_f and W_theta}}}{\subseteq}\TT_g\stackrel{\text{ Prop. \ref{P Q[1]}}}{\subseteq}\overline{\TT}_f\stackrel{\text{Prop. \ref{W_f and W_theta}}}{\subseteq}\overline{\TT}_\eta.\end{align*}
(b) is immediate from (a) and Proposition \ref{decomposition}.
\end{proof}

We directly obtain the following observation.

\begin{lemma}\label{E and T}
Let $\eta,\theta_i\in K_0(\proj A)$ for $1\le i\le\ell$ such that $E(\eta,\theta_i)=0$ for each $i$.
Then, for each $a\in\R_{>0}$ and $\theta' \in \cone^\circ\{ \theta_1,\ldots,\theta_\ell\}$, 
we have
\begin{align*}\overline{\TT}_{a\eta+\theta'}\subseteq\overline{\TT}_\eta\ \text{ and  }\ \FF_\eta\subseteq\FF_{a\eta+\theta'}.\end{align*}
\end{lemma}

\begin{proof}
It suffices to show the second assertion. 
By Proposition \ref{signs of X}, we have 
$\FF_\eta\subseteq\overline{\FF}_{\theta_i}$ for each $i$.
By Lemma \ref{additivity}, we get
$\FF_\eta\subseteq\overline{\FF}_{\theta'}.$
By Lemma \ref{additivity} again, we have $\FF_{\eta}=\FF_{a\eta}\cap\overline{\FF}_{\theta'}\subseteq\FF_{a\eta+\theta'}$.
\end{proof}

Recall that $\tors A$ and $\torf A$ are lattices,
and that we denote by $\vee$ the join in $\tors A$ and $\torf A$.
The following is the main result in this section, which was obtained in the discussion with Laurent Demonet \cite{D} when the authors were in Nagoya University.

\begin{theorem}\label{decomposition and T}
Let $A$ be a finite dimensional algebra over an algebraically closed field $k$, and $\theta,\theta_1,\ldots,\theta_\ell\in K_0(\proj A)$. Assume that $\theta=\theta_1\oplus\cdots\oplus\theta_\ell$ holds.
\begin{enumerate}[\rm(a)]
\item For each $\eta\in\cone^\circ\{\theta_1,\ldots,\theta_\ell\}$,
we have
\begin{align*}\overline{\TT}_\eta=\bigcap_{i=1}^\ell\overline{\TT}_{\theta_i},\ \TT_\eta=\bigvee_{i=1}^\ell\TT_{\theta_i},\ \overline{\FF}_\eta=\bigcap_{i=1}^\ell\overline{\FF}_{\theta_i},\ \FF_\eta=\bigvee_{i=1}^\ell\FF_{\theta_i}\ \text{ and }\ \WW_\eta=\bigcap_{i=1}^\ell\WW_{\theta_i}.\end{align*}
\item We have
\begin{align*}
[\theta]_{\rm TF}\supseteq\cone^\circ\{\theta_1,\ldots,\theta_\ell\}.
\end{align*}
\end{enumerate}
\end{theorem}

\begin{proof}
It suffices to prove (a).
We prove the equality for $\overline{\TT}$. Since ``$\supseteq$'' is clear from Lemma \ref{additivity}, it suffices to prove ``$\subseteq$''.
Writing $\eta=a\theta_i+\theta'$ with $a\in\R_{>0}$ and $\theta'\in\cone^\circ\{\theta_j\mid j\neq i\}$ and applying Lemma \ref{E and T}, we obtain $\overline{\TT}_\eta\subseteq\overline{\TT}_{\theta_i}$ for each $i$. Thus the assertion holds.

The dual argument shows the equality for $\overline{\FF}$. They give the other equalities.
\end{proof}

Immediately we obtain the following result.

\begin{corollary}\label{ind and TF}
Let $A$ be a finite dimensional algebra over an algebraically closed field $k$. For $\theta,\eta\in K_0(\proj A)$, if $\ind\theta=\ind\eta$, then $\theta$ and $\eta$ are TF equivalent.
\end{corollary}

Note that Conjecture \ref{canonical and TF} means that the two subsets in Theorem \ref{decomposition and T}(b) coincide.

\begin{remark}
The converse of Theorem \ref{decomposition and T}(b) is not  true, that is, for indecomposable elements $\theta_1,\ldots,\theta_\ell\in K_0(\proj A)$, the condition 
$\cone^\circ\{\theta_1,\ldots,\theta_\ell\} \subset [\theta_1+\cdots+\theta_\ell]_{\mathrm{TF}}$
does not imply $\theta_1\oplus\cdots\oplus\theta_\ell$ in general.
For example, if $\theta$ is wild in $K_0(\proj A)$,
then $\theta$ is TF equivalent to itself, but $\theta\oplus \theta$ does not hold.

The converse of Corollary \ref{ind and TF} is not true either, that is, TF equivalence of $\theta$ and $\eta$ does not imply $\ind\theta=\ind\eta$. For example, if $\theta$ is wild in $K_0(\proj A)$, then $\theta$ and $2\theta$ are TF equivalent, but $\ind\theta=\{\theta\}\neq\{2\theta\}=\ind 2\theta$.
\end{remark}

As an application, we obtain the following observation.

\begin{proposition}\label{wall direct summand}
For each $X\in\mod A$, the subset $\Theta_X\cap K_0(\proj A)$ is closed under direct summands of canonical decompositions.\end{proposition}

\begin{proof}
For each $\theta\in \Theta_X\cap K_0(\proj A)$, let $\theta=\theta_1\oplus\cdots\oplus\theta_\ell$ be a canonical decomposition. By Theorem \ref{decomposition and T}(a), we have $X\in\WW_\theta\subset\WW_{\theta_i}$ for each $1\le i\le\ell$. Thus $\theta_i\in\Theta_X$.
\end{proof}

Theorem \ref{decomposition and T} recovers the sign-coherence due to \cite{P}:
We say $\theta_1,\ldots,\theta_\ell \in K_0(\proj A)$ are \textit{sign-coherent}
if $\add P_1^{\theta_i} \cap \add P_0^{\theta_j}=\{0\}$ for all $i,j$.

\begin{corollary}\cite[Lemma 2.10]{P}
If $\theta_1\oplus\cdots\oplus\theta_\ell$ holds in $K_0(\proj A)$, 
then $\theta_1,\ldots,\theta_\ell$ are sign-coherent.
\end{corollary}

\begin{proof}
We assume that $\add P_1^{\theta_i} \cap \add P_0^{\theta_j}\ne\{0\}$ for some $i,j$.
Take an indecomposable object $P \in \add P_1^{\theta_i} \cap \add P_0^{\theta_j}$ and the corresponding simple module $S:=\top P \in \mod A$. Then  $S \in \FF_{\theta_i} \cap \TT_{\theta_j}$ holds.
For $\theta:=\sum_{i=1}^\ell \theta_i$, we have
\begin{align*}S \in \FF_{\theta_i} \cap \TT_{\theta_j} \stackrel{\text{Thm.   \ref{decomposition and T}}}{\subseteq} \FF_\theta \cap \TT_\theta = \{0\},\end{align*}
a contradiction. Thus $\add P_1^{\theta_i} \cap \add P_0^{\theta_j}=\{0\}$ as desired.
\end{proof}

\begin{remark}
In Theorem \ref{decomposition and T}, $\WW_\theta$ is not a Serre subcategory of $\WW_{\theta_i}$.
For example, let $A=k(1 \to 2 \to 3)$, $\theta_1=[P(1)]-[P(2)]$ and $\theta_2=[P(1)]-[P(3)]$.
Then $\theta_1 \oplus \theta_2$ holds.
Moreover $\WW_{\theta_1+\theta_2}$ has only one simple object $P(1)$, but it is not simple in $\WW_{\theta_1}$ or $\WW_{\theta_2}$.
\end{remark}

\subsection{Functorial properties}

Let $A$ and $B$ be finite dimensional $k$-algebras, and let $\phi \colon A\to B$ be a $k$-algebra homomorphism which is not necessarily unital. Thus $1_A:=\phi(1_A)$ is an idempotent of $B$, and $\phi$ is a composite of a unital $k$-algebra homomorphism $A\to 1_AB1_A$ and the natural inclusion $1_AB1_A\to B$. We denote by
\begin{align*}(-)1_A \colon \mod B\to\mod A\end{align*}
the restriction functor. We also have an additive functor
\begin{align*}-\otimes_A1_AB \colon \proj A\to\proj B\end{align*}
and a triangle functor
\begin{align*}-\otimes_A1_AB \colon \KKK^{\bo}(\proj A)\to\KKK^{\bo}(\proj B).\end{align*}
We also have morphisms
\begin{align*}-\otimes_A1_AB \colon K_0(\proj A)\to K_0(\proj B)\ \text{ and } \ {-}\otimes_A1_AB \colon K_0(\proj A)_\R\to K_0(\proj B)_\R,\end{align*}
which make the following diagram commutative:
\begin{align*}\xymatrix{
\KKK^{\bo}(\proj A)\ar[d]^{[-]}\ar[rr]^{-\otimes_A1_AB}&&\KKK^{\bo}(\proj B)\ar[d]^{[-]}\\
K_0(\proj A)\ar[rr]^{-\otimes_A1_AB}&&K_0(\proj B).
}\end{align*}
We give the following useful basic properties.

\begin{proposition}\label{W^A and W^B}
Under the above setting, we have the following assertions.
\begin{enumerate}[\rm(a)]
\item For each morphism $f \colon P_1\to P_0$ in $\proj A$, we have
\begin{align*}
\overline{\TT}_{f\otimes 1_AB}^B&=\{X\in\mod B\mid X1_A\in\overline{\TT}_{f}^A\},&\overline{\FF}_{f\otimes 1_AB}^B=\{X\in\mod B\mid X1_A\in\overline{\FF}_{f}^A\},\\
\WW_{f\otimes 1_AB}^B&=\{X\in\mod B\mid X1_A\in\WW_{f}^A\}.\end{align*}
\item For each $\theta\in K_0(\proj A)_\R$, we have
\begin{align*}
\overline{\TT}_{\theta\otimes 1_AB}^B&\supseteq\{X\in\mod B\mid X1_A\in\overline{\TT}_\theta^A\},&\overline{\FF}_{\theta\otimes 1_AB}^B\supseteq\{X\in\mod B\mid X1_A\in\overline{\FF}_\theta^A\},\\
\WW_{\theta\otimes 1_AB}^B&\supseteq\{X\in\mod B\mid X1_A\in\WW_\theta^A\}.
\end{align*}
\item The equalities in (b) hold if there exists an idempotent $e \in A$ satisfying the following conditions.
\begin{enumerate}[\rm(i)]
\item $\phi(eA(1_A-e))=eB(1_A-e)$,
\item $\theta \in \sum_{P_+ \in \add eA}\R_{\ge 0}[P_+]-\sum_{P_- \in \add (1_A-e)A}\R_{\ge 0}[P_-] 
\subset K_0(\proj A)_\R$.
\end{enumerate}
\end{enumerate}
\end{proposition}

To prove part (c), we need the following observation.

\begin{lemma}\label{sub A and sub B}
Let $\phi \colon A\to B$ be a morphism of $k$-algebras, and $e\in A$ an idempotent. Assume $\phi(eA(1_A-e))=eB(1_A-e)$ holds. Then for each $X\in\mod B$ and an $A$-submodule $Y$ of $X1_A$, there exists a $B$-submodule $Y'$ of $X$ satisfying $Y'e\supseteq Ye$ and $Y'(1_A-e)\subseteq Y(1_A-e)$.
\end{lemma}

\begin{proof}
Since $\phi\colon A\to B$ is a composite of a unital $k$-algebra homomorphism $A\to 1_AB1_A$ and the natural inclusion $1_AB1_A\to B$, it suffices to consider the following two cases.

(i) Consider the case $A=1_AB1_A$. 
Then $Y':=YB\subset X$ is a $B$-submodule of $X$ satisfying $Y'1_A=Y(1_AB1_A)=YA=Y$. Thus $Y'e=Ye$ and $Y'(1_A-e)=Y(1_A-e)$.

(ii) Consider the case $1_A=1_B$. 

Let $f:=1_B-e$. For an $A$-submodule $Y$ of $X1_A$, let
\begin{align*}V:=YeBe \subseteq Xe\ \mbox{ and }\ W:=\{w\in Yf\mid w \cdot fBe \subseteq V\}\subseteq Yf.\end{align*}
We prove that
\begin{align*}
Y':=V\oplus W
\end{align*}
is a $B$-submodule of $X$, that is, $Y'B\subseteq Y'$ holds. Then the assertion follows from $Y'e=V\supseteq Ye$ and $Y'f=W\subseteq Yf$.
The inclusions
\begin{align*}V \cdot eBe \subseteq V,\ W \cdot fBf\subseteq W\ \textrm{ and }\ 
W \cdot fBe \subseteq V\end{align*}
are clear from the definitions. Thus it suffices to show $V \cdot eBf \subseteq W$.
Since $\phi(eAf)=eBf$ holds and $Y$ is an $A$-submodule of $X$, we have $YeBf=YeAf \subseteq Yf$ and hence
\begin{align*}V \cdot eBf=YeBeBf = YeBf \subseteq Yf.\end{align*}
Since $V \cdot eBf \cdot fBe \subseteq V \cdot eBe=V$, we have $V \cdot eBf\subseteq W$, as desired.
\end{proof}

Now we are ready to prove Proposition \ref{W^A and W^B}.

\begin{proof}[Proof of Proposition \ref{W^A and W^B}]
(a) For each $X\in\mod B$, we have a commutative diagram
\begin{align*}\xymatrix@R1em@C7em{
\Hom_B(P_1\otimes_A1_AB,X)\ar[r]^{\Hom_B(f\otimes 1_AB,X)}\ar[d]^\wr&\Hom_B(P_0\otimes_A1_AB,X)\ar[d]^\wr\\
\Hom_A(P_1,X1_A)\ar[r]^{\Hom_A(f,X1_A)}&\Hom_A(P_0,X1_A).
}\end{align*}
Thus the assertions follow immediately.

(b) We only prove the second equality since the first one is a dual and the third one follows from others. 
Fix $X$ in the right-hand side. Then each $B$-submodule $Y$ of $X$ gives an $A$-submodule $Y1_A$ of $X1_A$. Since $\Hom_B(P \otimes_A1_AB,Y) \simeq \Hom_A(P,Y1_A)$ holds for any $P \in \proj A$, our assumption $X1_A\in\overline{\FF}_\theta^A$ implies $(\theta\otimes 1_AB)(Y)=\theta(Y1_A)\le0$. Thus $X\in\overline{\FF}_{\theta\otimes 1_AB}^B$ holds.

(c) Again we only prove the second equality. Fix $X\in\overline{\FF}_{\theta\otimes 1_AB}^B$. To prove $X1_A\in\overline{\FF}_\theta^A$, let $Y$ be an $A$-submodule of $X1_A$. By our asssumption (i) and Lemma \ref{sub A and sub B}, there exists a $B$-submodule $Y'$ of $X$ satisfying $Y'e\supseteq Ye$ and $Y'(1_A-e)\subseteq Y(1_A-e)$. By our assumption (ii), we have
\begin{align*}\theta(Y)\le\theta(Y'1_A)=(\theta\otimes 1_AB)(Y')\le0.\end{align*}
Thus $X1_A\in\overline{\FF}^A_{\theta}$ holds.
\end{proof}

We apply the results above to some special cases.

\begin{example}\label{W^A and W^B 2}
Let $B$ be a finite dimensional $k$-algebra, $e\in B$ an idempotent and $A=eBe$. Then we have a fully faithful functor
\begin{align*}{-}\otimes_A eB \colon \proj A\to\proj B,\end{align*}
which induces embeddings
\begin{align*}{-}\otimes_A eB \colon K_0(\proj A)\to K_0(\proj B) \ \text{ and } \  
{-}\otimes_A eB \colon K_0(\proj A)_\R\to K_0(\proj B)_\R.\end{align*}
The following observations are special cases of Proposition \ref{W^A and W^B}. 
\begin{enumerate}[\rm(a)]
\item For each morphism $f$ in $\proj A$, we have 
\begin{align*}
\overline{\TT}_{f\otimes eB}^B&=\{X\in\mod B\mid Xe\in\overline{\TT}_{f}^A\},&\overline{\FF}_{f\otimes eB}^B=\{X\in\mod B\mid Xe\in\overline{\FF}_{f}^A\},\\
\WW_{f\otimes eB}^B&=\{X\in\mod B\mid Xe\in\WW_{f}^A\}.\end{align*}
\item For each $\theta\in K_0(\proj A)_\R$, we have
\begin{align*}
\overline{\TT}_{\theta\otimes eB}^B&=\{X\in\mod B\mid Xe\in\overline{\TT}_{\theta}^A\},&\overline{\FF}_{\theta\otimes eB}^B=\{X\in\mod B\mid Xe\in\overline{\FF}_{\theta}^A\},\\\WW_{\theta\otimes eB}^B&=\{X\in\mod B\mid Xe\in\WW_\theta^A\}.\end{align*}
\end{enumerate}
\end{example}

We also consider the following case e.g.~\cite{DIRRT}.

\begin{example}\label{pi theta}
Let $A$ and $B$ be finite dimensional $k$-algebras, and $\phi\colon A\to B$ a surjective $k$-algebra homomorphism. Then the restriction functor $(-)_A \colon \mod B\to\mod A$ is fully faithful,
so we regard $\mod B$ as a full subcategory of $\mod A$.
Moreover, the group homomorphisms ${-}\otimes_AB \colon K_0(\proj A) \to K_0(\proj B)$ and ${-}\otimes_AB \colon K_0(\proj A)_\R \to K_0(\proj B)_\R$ are surjective.
In this case, Proposition \ref{W^A and W^B} becomes the following form. 
\begin{enumerate}[\rm(a)]
\item For each morphism $f$ in $\proj A$, we have
\begin{align*}
\TT_{f\otimes B}^B=\TT_{f}^A\cap\mod B,\quad\FF_{f\otimes B}^B=\FF_{f}^A\cap\mod B,&\\
\overline{\TT}_{f\otimes B}^B=\overline{\TT}_{f}^A\cap\mod B,\quad\overline{\FF}_{f\otimes B}^B=\overline{\FF}_{f}^A\cap\mod B,&\quad \WW_{f \otimes B}^B=\WW_f^A \cap \mod B.
\end{align*}
\item For each $\theta\in K_0(\proj A)_\R$, we have
\begin{align*}
\TT_{\theta\otimes B}^B=\TT_{\theta}^A\cap\mod B,\quad\FF_{\theta\otimes B}^B=\FF_{\theta}^A\cap\mod B,&\\
\overline{\TT}_{\theta\otimes B}^B=\overline{\TT}_{\theta}^A\cap\mod B,\quad\overline{\FF}_{\theta\otimes B}^B=\overline{\FF}_{\theta}^A\cap\mod B,&\quad \WW_{\theta\otimes B}^B=\WW_\theta^A\cap\mod B.
\end{align*}
\end{enumerate}
Moreover, let $P$ be a 2-term presilting (resp.~2-term silting) complex in $\KKK^{\bo}(\proj A)$. 
\begin{enumerate}[\rm(a)]
\setcounter{enumi}{2}
\item $P \otimes_AB$ is a 2-term presilting (resp.~2-term silting) complex in $\KKK^{\bo}(\proj B)$.
\item If $\theta \in C^\circ(P)$, then we have $\theta \otimes_AB \in C^\circ(P \otimes_AB)$. 
\end{enumerate}
\end{example}

\begin{proof}
(c) and (d) are known to experts, but we include the proof for the convenience of the reader.

(c) For each 2-term complex $P$ in $\KKK^{\bo}(\proj A)$, we have a morphism of Hom-complexes
\begin{align*}\cHom_A(P,P)\to\cHom_B(P\otimes_AB,P\otimes_AB)\end{align*}
which is term-wise surjective. Since the degree 1 terms of both complexes are zero, the morphism
\begin{align*} H^1(\cHom_A(P,P)) \to H^1(\cHom_B(P\otimes_AB,P\otimes_AB))\end{align*}
is surjective. Since $P$ is presilting in $\KKK^{\bo}(\proj A)$, $H^1(\cHom_A(P,P))=\Hom_{\KKK^{\bo}(\proj A)}(P,P[1])=0$ hold. Thus $\Hom_{\KKK^{\bo}(\proj B)}(P\otimes_AB,(P\otimes_AB)[1])=H^1(\cHom_B(P\otimes_AB,P\otimes_AB))=0$ as desired.

If $P$ is 2-term silting in $\KKK^{\bo}(\proj A)$, then, since $A\in\thick P$ and $-\otimes_AB$ is a triangle functor, we have $B=A\otimes_AB\in (\thick P) \otimes_AB \subset\thick(P\otimes_AB)$.
Thus $P \otimes_AB$ is 2-term silting in $\KKK^{\bo}(\proj B)$.

(d) follows from (c) immediately.
\end{proof}

We remark that even if $\theta \in K_0(\proj B)$ is indecomposable rigid, 
$\theta \otimes_AB\in K_0(\proj B)$ is not necessarily indecomposable.
For example, if $A$ is the Kronecker quiver algebra $k(1 \rightrightarrows 2)$
and $B$ is $k(1 \to 2)$,
then $\theta:=2[P_A(1)]-[P_A(2)] \in K_0(\proj A)$ is indecomposable rigid,
but $\theta \otimes_AB=2[P_B(1)]-[P_B(2)]=[P_B(1)] \oplus ([P_B(1)]-[P_B(2)])$ is not indecomposable (but rigid).

\section{Constructing semistable torsion pairs from morphism torsion pairs}\label{section 4}

\subsection{Gluing morphism torsion pairs}

In this subsection, we consider the relationship between morphism torsion pairs $(\overline{\TT}_f,\FF_f)$ and $(\TT_f,\overline{\FF}_f)$ and
semistable torsion pairs $(\overline{\TT}_\theta,\FF_\theta)$ and $(\TT_\theta,\overline{\FF}_\theta)$.
We prepare the following symbols.

\begin{definition}\label{uni mor tors}
For $\theta\in K_0(\proj A)$, let
\begin{align*}
\TT^{\h}_\theta:=\bigcap_{[f]=\theta} \TT_f, \quad
\FF^{\h}_\theta:=\bigcap_{[f]=\theta} \FF_f,&\\
\overline{\TT}^{\h}_\theta:=\bigcup_{[f]=\theta} \overline{\TT}_f, \quad
\overline{\FF}^{\h}_\theta:=\bigcup_{[f]=\theta} \overline{\FF}_f,&\quad
\WW^{\h}_\theta:=\bigcup_{[f]=\theta} \WW_f,
\end{align*}
where $f$ runs over all morphisms $f \colon P_1 \to P_0$ in $\proj A$ such that $[f]:=[P_0]-[P_1]=\theta$.
\end{definition}

It is immediate from Proposition \ref{W_f and W_theta} that we have
\begin{equation}\label{W_f and W_theta 3}
\TT^{\h}_\theta\supseteq\TT_\theta, \quad
\FF^{\h}_\theta\supseteq\FF_\theta, \quad
\overline{\TT}^{\h}_\theta\subseteq\overline{\TT}_\theta, \quad
\overline{\FF}^{\h}_\theta\subseteq\overline{\FF}_\theta, \quad
\WW^{\h}_\theta\subseteq\WW_\theta.
\end{equation}

Note that $P_1$ and $P_0$ above may have common indecomposable direct summands, but such cases are redundant.

\begin{proposition}\label{exclude common summand}
For $\theta\in K_0(\proj A)$, we have
\begin{align*}
\TT^{\h}_\theta=\bigcap_{f\in\Hom(\theta)} \TT_f, \quad
\FF^{\h}_\theta=\bigcap_{f\in\Hom(\theta)} \FF_f,&\\
\overline{\TT}^{\h}_\theta=\bigcup_{f\in\Hom(\theta)} \overline{\TT}_f, \quad
\overline{\FF}^{\h}_\theta=\bigcup_{f\in\Hom(\theta)} \overline{\FF}_f,&\quad
\WW^{\h}_\theta=\bigcup_{f\in\Hom(\theta)} \WW_f.
\end{align*}
\end{proposition}

\begin{proof}
We only prove the assertion for $\overline{\TT}$ since others can be shown similarly.
Write $\theta=[P_0]-[P_1]$, where $P_0$ and $P_1$ do not have a non-zero common direct summand.
For any morphism $f$ in $\proj A$ such that $[f]=\theta$, there exists $Q\in\proj A$ such that $f\in\Hom_A(P_1\oplus Q,P_0\oplus Q)$. Let
\begin{align*}\pi \colon \Hom_A(P_1\oplus Q,P_0\oplus Q)\to\End_A(Q)\end{align*}
be a natural projection. For $G=\Aut_A(P_0\oplus Q)\times\Aut_A(P_1\oplus Q)$, let
\begin{align*}U:=G\{g\oplus 1_Q\mid g\in\Hom(\theta)\}\subset\Hom_A(P_1\oplus Q,P_0\oplus Q).\end{align*}
Then $U \supset \pi^{-1}(\Aut_A(Q))$ holds. Since $\Aut_A(Q)$ is an open dense subset of $\End_A(Q)$, $U$ is a dense subset of $\Hom_A(P_1\oplus Q,P_0\oplus Q)$.
Thus Proposition \ref{closure and T} implies
\begin{align*}\overline{\TT}_f\subseteq\bigcup_{g\in\Hom(\theta)}\overline{\TT}_{g\oplus 1_Q}=\bigcup_{g\in\Hom(\theta)}\overline{\TT}_{g}.\end{align*}
Thus the assertion follows.
\end{proof}

We also remark that it does not directly follow from the definition that $\overline{\TT}^{\h}_\theta$ is a torsion class. Our proof of this property will be given in Proposition \ref{T_PHom_tors}.

Now we can state the following.
We remark that this property has been independently proved by Fei \cite[Lemma 3.13]{Fei}.
Our proof is given in the last subsection in this section.

\begin{theorem}\label{cup T_f}
Let $\theta \in K_0(\proj A)$.
\begin{enumerate}[\rm (a)]
\item We have
\begin{align*}
\TT_\theta=\bigcap_{\ell\ge1}\TT^{\h}_{\ell\theta},\quad
\FF_\theta=\bigcap_{\ell\ge1}\FF^{\h}_{\ell\theta},\quad
\overline{\TT}_\theta=\bigcup_{\ell\ge1}\overline{\TT}^{\h}_{\ell\theta},\quad
\overline{\FF}_\theta=\bigcup_{\ell\ge1}\overline{\FF}^{\h}_{\ell\theta},\quad
\WW_\theta=\bigcup_{\ell\ge1}\WW^{\h}_{\ell\theta}.
\end{align*}
\item If $\theta$ is tame, then 
\begin{align*}
\TT_\theta=\TT^{\h}_{\theta},\quad
\FF_\theta=\FF^{\h}_{\theta},\quad
\overline{\TT}_\theta=\overline{\TT}^{\h}_{\theta},\quad
\overline{\FF}_\theta=\overline{\FF}^{\h}_{\theta},\quad
\WW_\theta=\WW^{\h}_{\theta}.
\end{align*}
\end{enumerate}
\end{theorem}

The equality $\WW^{\h}_\theta=\WW^{\h}_{\ell\theta}$ does not necessarily hold in general if $\theta$ is wild, see Example \ref{kronecker}.

We state some applications of Theorem \ref{cup T_f}.

\begin{corollary}\label{C f T bar}
For $\eta,\theta \in K_0(\proj A)$, the following conditions are equivalent.
\begin{enumerate}[\rm(a)]
\item There exists $f \in \Hom(\eta)$ satisfying $\TT_f \subseteq \overline{\TT}_{\theta}$ and $\FF_f \subseteq \overline{\FF}_{\theta}$.
\item There exists $\ell \in \Z_{\ge 1}$ satisfying $\eta \oplus \ell\theta$.
\end{enumerate}
Moreover, if $\theta$ is tame, then the following condition is also equivalent.
\begin{enumerate}[\rm(c)]
\item $\eta \oplus \theta$.
\end{enumerate}
\end{corollary}

Corollary \ref{C f T bar} follows immediately from the following more explicit result.

\begin{corollary}\label{C f T bar 2}
For $\eta,\theta \in K_0(\proj A)$ and $f\in\Hom(\theta)$, the following conditions are equivalent.
\begin{enumerate}[\rm(a)]
\item $\TT_f \subseteq \overline{\TT}_{\theta}$ and $\FF_f \subseteq \overline{\FF}_{\theta}$.
\item There exist $\ell \in \Z_{\ge 1}$ and $g\in\Hom(\ell\theta)$ satisfying $E(f,g)=0=E(g,f)$.
\end{enumerate}
Moreover, if $\theta$ is tame, then the following condition is also equivalent.
\begin{enumerate}[\rm(c)]
\item There exists $g\in\Hom(\theta)$ satisfying $E(f,g)=0=E(g,f)$.
\end{enumerate}
\end{corollary}

\begin{proof}
(a)$\Leftrightarrow$(b) By Proposition \ref{P Q[1]}, the condition (b) is equivalent to the following condition.

\begin{enumerate}[\rm(b$'$)]
\item There exist $\ell \in \Z_{\ge 1}$ and $g\in\Hom(\ell\theta)$ satisfying $\TT_f\subset\overline{\TT}_g$ and $\FF_f\subset\overline{\FF}_g$.
\end{enumerate}

On the other hand, the condition (a) is equivalent to $C_f\in\overline{\TT}_\theta$ and $K_{\nu f}\in\overline{\FF}_\theta$. By Theorem \ref{cup T_f}, this is equivalent to that there exist $\ell,\ell'\ge1$, $g\in\Hom(\ell\theta)$ and $g'\in\Hom(\ell'\theta)$ satisfying $C_f\in\overline{\TT}_g$ and $K_{\nu f}\in\overline{\FF}_{g'}$. We may assume $\ell=\ell'$ by replacing $g$ and $g'$ by $g^{\oplus\ell'}$ and $(g')^{\oplus\ell}$ respectively. By Lemma \ref{open}, we may assume $g=g'$.
Consequently, the conditions (a) is equivalent to (b$'$).

(b)$\Leftrightarrow$(c) Since $\theta\oplus\theta$, we have $\ell\theta=\theta^{\oplus\ell}$. Thus the assertion follows.
\end{proof}

Theorem \ref{cup T_f} can be extended for an arbitrary element $\theta\in K_0(\proj A)_\R$ as follows.

\begin{corollary}\label{cap T(C_f)R}
For $\theta \in K_0(\proj A)_\R$, take $\theta_i,\theta^i\in K_0(\proj A)_\Q$ for each $i\in\N$ such that $\theta_i\le\theta\le\theta^i$ and $\displaystyle\lim_{i\to\infty}\theta_i=\theta=\lim_{i\to\infty}\theta^i$. Then we have
\begin{align*}
\TT_\theta&=\bigcup_{i \ge 1}\TT_{\theta_i}=\bigcup_{i \ge 1}\TT^{\h}_{\theta_i},&
\FF_\theta&=\bigcup_{i \ge 1}\FF_{\theta^i}=\bigcup_{i \ge 1}\FF^{\h}_{\theta^i},\\
\overline{\TT}_\theta&=\bigcap_{i \ge 1}\overline{\TT}_{\theta^i}=\bigcap_{i \ge 1}\overline{\TT}^{\h}_{\theta^i},&
\overline{\FF}_\theta&=\bigcap_{i \ge 1}\overline{\FF}_{\theta_i}=\bigcap_{i \ge 1}\overline{\FF}^{\h}_{\theta_i}.
\end{align*}
\end{corollary}

\begin{proof}
We prove the equalities for $\overline{\TT}$ since the others can be shown in a similar way.
Since $\displaystyle\theta=\liminf_{i\to\infty}\theta^i$, for each $X\in\mod A$, 
$\theta(X)\ge0$ if and only if $\theta^i(X)\ge0$ for each $i\ge1$. Thus the first equality holds. The second one is immediate from the first one and Theorem \ref{cup T_f}.
\end{proof}

\subsection{Proof of Theorem \ref{cup T_f}}

In this section, we give our proof of Theorem \ref{cup T_f}
connecting morphism torsion pairs $(\overline{\TT}_f,\FF_f)$ and $(\TT_f,\overline{\FF}_f)$ and semistable torsion pairs $(\overline{\TT}_\theta,\FF_\theta)$ and $(\TT_\theta,\overline{\FF}_\theta)$.
Our strategy is that we first prove this theorem for 
generalized Kronecker quivers, and then consider general cases.

We first need to show that $\overline{\TT}^{\h}_\theta$,
which is the union of $\overline{\TT}_f$ for $f \in \Hom(\theta)$, 
is surely a torsion class.

\begin{proposition}\label{T_PHom_tors}
Let $\theta \in K_0(\proj A)$. Then we have torsion pairs
\begin{align*}(\overline{\TT}^{\h}_\theta,\FF^{\h}_\theta)\ \text{ and }\  (\TT^{\h}_\theta,\overline{\FF}^{\h}_\theta)\end{align*}
in $\mod A$, and we have
\begin{align*}\WW^{\h}_\theta\in\serre\WW_\theta\subseteq\wide A.\end{align*}
\end{proposition}

\begin{proof}
We prove $\overline{\TT}^{\h}_\theta\in\tors A$. Clearly $\overline{\TT}^{\h}_\theta$ is closed under factor modules. It remains to show that, for each exact sequence $0\to X\to Y\to Z\to0$ with $X,Z\in\overline{\TT}^{\h}_\theta$, we have $Y\in\overline{\TT}^{\h}_\theta$. Take $f,g\in\Hom(\theta)$ such that $X\in\overline{\TT}_f$ and $Z\in\overline{\TT}_g$. By Lemma \ref{open}, both
\begin{align*}\{h\in\Hom(\theta)\mid X\in\overline{\TT}_h\}\ \text{ and }\ \{h\in\Hom(\theta)\mid Z\in\overline{\TT}_h\}\end{align*}
are non-empty open subsets of $\Hom(\theta)$. Since $\Hom(\theta)$ is irreducible, there exists $h\in\Hom(\theta)$ such that both $X,Z$ belong to $\overline{\TT}_h$. Since $\overline{\TT}_h$ is a torsion class, $Y\in\overline{\TT}_h\subset\overline{\TT}^{\h}_\theta$ hold, as desired,

Dually, one can prove $\overline{\FF}^{\h}_\theta\in\torf A$. Finally, one can show $\WW^{\h}_\theta\in\serre\WW_\theta$ in a similar way by using Lemma \ref{W_f and W_theta 2}(b).
\end{proof}

We have the following clear observation.

\begin{lemma}\label{theta and l theta}
Let $\eta,\theta\in K_0(\proj A)$. Then we have
\begin{align*}
\TT^{\h}_{\eta}\vee\TT^{\h}_{\theta}\supseteq\TT^{\h}_{\eta+\theta},\quad
\FF^{\h}_{\eta}\vee\FF^{\h}_{\theta}\supseteq\FF^{\h}_{\eta+\theta},\quad&\\
\overline{\TT}^{\h}_{\eta}\cap\overline{\TT}^{\h}_{\theta}\subseteq\overline{\TT}^{\h}_{\eta+\theta},\quad
\overline{\FF}^{\h}_{\eta}\cap\overline{\FF}^{\h}_{\theta}\subseteq\overline{\FF}^{\h}_{\eta+\theta},\quad&
\WW^{\h}_{\eta}\cap\WW^{\h}_{\theta}\subseteq\WW^{\h}_{\eta+\theta}.
\end{align*}
In particular, for each $\ell\in\Z_{\ge1}$, we have
\begin{align*}\TT^{\h}_\theta\supseteq\TT^{\h}_{\ell\theta}\supseteq\TT_\theta,\quad
\FF^{\h}_\theta\supseteq\FF^{\h}_{\ell\theta}\supseteq\FF_\theta,\quad
\overline{\TT}^{\h}_\theta\subseteq\overline{\TT}^{\h}_{\ell\theta}\subseteq\overline{\TT}_\theta,\quad
\overline{\FF}^{\h}_\theta\subseteq\overline{\FF}^{\h}_{\ell\theta}\subseteq\overline{\FF}_\theta, \quad
\WW^{\h}_\theta\subseteq\WW^{\h}_{\ell\theta}\subseteq\WW_\theta.\end{align*}
\end{lemma}

\begin{proof}
We first to prove the assertion for $\overline{\TT}^{\h}$.
For each $X\in\overline{\TT}^{\h}_{\eta}\cap\overline{\TT}^{\h}_{\theta}$, take $f\in\Hom(\eta)$ and $g\in\Hom(\theta)$ such that $X\in\overline{\TT}_f\cap\overline{\TT}_g$. Then $X\in\overline{\TT}_{f\oplus g}$ and hence $X\in\overline{\TT}^{\h}_{\eta+\theta}$. Thus the assertion for $\overline{\TT}^{\h}$ follows.
Dually we obtain the assertion for $\overline{\FF}^{\h}$, and the equality for $\WW^{\h}$ also follows.
By Proposition \ref{T_PHom_tors}, the assertions for $\TT^{\h}$ and $\FF^{\h}$ follows.
\end{proof}

We additionally define the following notations for our proof:
\begin{align*}
\TT^{\h}_{\N\theta}:=\bigcap_{\ell\ge1}\TT^{\h}_{\ell\theta},\quad\FF^{\h}_{\N\theta}:=\bigcap_{\ell\ge1}\FF^{\h}_{\ell\theta},&\\
\overline{\TT}^{\h}_{\N\theta}:=\bigcup_{\ell\ge1}\overline{\TT}^{\h}_{\ell\theta},\quad\overline{\FF}^{\h}_{\N\theta}:=\bigcup_{\ell\ge1}\overline{\FF}^{\h}_{\ell\theta},&\quad\WW^{\h}_{\N\theta}:=\bigcup_{\ell\ge1}\WW^{\h}_{\ell\theta}.
\end{align*}
These definitions can be extended to $\theta\in K_0(\proj A)_\Q$ in an obvious way.

\begin{lemma}\label{cup T_f pre}
For $\eta,\theta\in K_0(\proj A)$, we have torsion pairs
\begin{align*}(\overline{\TT}^{\h}_{\N\theta},\FF^{\h}_{\N\theta})\ \text{ and }\ (\TT^{\h}_{\N\theta},\overline{\FF}^{\h}_{\N\theta})\end{align*}
in $\mod A$. Moreover we have
\begin{align*}
\TT^{\h}_{\N\theta}\supseteq\TT_\theta,\quad\FF^{\h}_{\N\theta}\supseteq\FF_\theta,\quad
\overline{\TT}^{\h}_{\N\theta}\subseteq\overline{\TT}_\theta,\quad\overline{\FF}^{\h}_{\N\theta}\subseteq\overline{\FF}_\theta,\quad&\WW^{\h}_{\N\theta}\in\serre\WW_\theta\subseteq\wide A,
\end{align*}
and
\begin{align*}
\TT^{\h}_{\N\eta}\vee\TT^{\h}_{\N\theta}\supseteq\TT^{\h}_{\N(\eta+\theta)},\quad
\FF^{\h}_{\N\eta}\vee\FF^{\h}_{\N\theta}\supseteq\FF^{\h}_{\N(\eta+\theta)},\quad&\\
\overline{\TT}^{\h}_{\N\eta}\cap\overline{\TT}^{\h}_{\N\theta}\subseteq\overline{\TT}_{\N(\eta+\theta)},\quad
\overline{\FF}^{\h}_{\N\eta}\cap\overline{\FF}^{\h}_{\N\theta}\subseteq\overline{\FF}^{\h}_{\N(\eta+\theta)},\quad&
\WW^{\h}_{\N\eta}\cap\WW^{\h}_{\N\theta}\subseteq\WW^{\h}_{\N(\eta+\theta)}.
\end{align*}
\end{lemma}

\begin{proof}
To prove the first statement, it suffices to prove that $\overline{\TT}^{\h}_{\N\theta}\in\tors A$ and $\overline{\FF}^{\h}_{\N\theta}\in\torf A$ thanks to Proposition \ref{T_PHom_tors}.
We only prove $\overline{\TT}^{\h}_{\N\theta}\in\tors A$ since the other one is the dual.
It suffices to show that $\overline{\TT}^{\h}_{\N\theta}$ is closed under extensions. Let $0\to X\to Y\to Z\to0$ be an exact sequence in $\mod A$ such that $X\in\overline{\TT}^{\h}_{\ell\theta}$ and $Z\in\overline{\TT}^{\h}_{m\theta}$ for some $\ell,m\in\Z_{\ge1}$. Then $X,Z\in\overline{\TT}^{\h}_{\ell m\theta}$ holds by Lemma \ref{theta and l theta}, and hence $Y\in\overline{\TT}^{\h}_{\ell m\theta}$.

The second statement is immediate from Proposition \ref{T_PHom_tors} and Lemma \ref{theta and l theta}. 

We prove the third statement. We only show the inclusion for $\overline{\TT}^{\h}$. Then the inclusion for $\overline{\FF}^{\h}$ follows dually, and the other assertions do.
For $X\in\overline{\TT}^{\h}_{\N \eta}\cap\overline{\TT}^{\h}_{\N \theta}$, take $f\in\Hom(\ell\eta)$ and $g\in\Hom(m\theta)$ with $\ell,m\in\Z_{\ge1}$ such that $X\in\overline{\TT}_f\cap\overline{\TT}_g$. Then $X\in\overline{\TT}_{f^{\oplus m}\oplus g^{\oplus\ell}}$ and therefore $X\in\overline{\TT}_{\ell m(\eta+\theta)}$, as desired.
\end{proof}

The following special case of Theorem \ref{cup T_f} is proved by using geometric invariant theory of quiver representations.

\begin{lemma}\label{W_theta=W_f}
Let $Q$ be an acyclic quiver and $A=kQ$. For any $\theta\in K_0(\proj A)$, we have
\begin{align*}\WW_\theta=\WW^{\h}_{\N\theta}.\end{align*}
\end{lemma}

\begin{proof}
Fix a dimension vector $d$, and consider the module variety $V:=\mod (A,d)$, its coordinate algebra $k[V]$, and the group $G:={\rm GL}(d)$ acting on $V$. For $p\in V$, we denote by $X_p$ the corresponding $A$-module.
For a character $\chi \colon G\to k^\times$, we denote by
\begin{align*}k[V]^\chi:=\{a\in k[V]\mid a(gp)=\chi(g)a(p)\ \text{ for all $p\in V$}\}\end{align*}
the space of $\chi$-semi-invariants. Recall that a point $p\in V$ is called \textit{$\chi$-semistable} if there exist $\ell\ge 1$ and $a\in k[V]^{\chi^\ell}$ such that $a(p)\neq0$.

For $\theta\in K_0(\proj A)$, we consider the character
\begin{align*}\chi_\theta \colon {\rm GL}(d)\to k^\times\ \text{ given by }\ \chi_\theta(g)=\prod_{i\in Q_0}(\det g_i)^{\theta_i}\end{align*}
given by our $\theta$. Then a point $p\in V$ is $\chi_\theta$-semistable if and only if $X_p\in\WW_\theta$ \cite[Proposition 3.1]{K}.

For each $\theta \in K_0(\proj A)$ with $\theta(d)=0$ and $f \in \Hom(\theta)$, we consider $a_f\in k[V]^{\chi_\theta}$ given by
\begin{align*}a_f(p):={\rm det}(f\circ{-} \colon \Hom_A(P_0,X_p)\to\Hom_A(P_1,X_p)).\end{align*}
Then the $k$-vector space $k[V]^{\chi_\theta}$ is spanned by $a_f$ for all $\theta \in K_0(\proj A)$ with $\theta(d)=0$ and $f \in \Hom(\theta)$ \cite[Theorem 1]{DW2}.

We are ready to prove the assertion. For $X\in\WW_\theta$ with $d:=\dimv X$, take a point $p\in \mod(A,d)$ such that $X\simeq X_p$. Then there exists $\ell\ge1$ and $a\in k[V]^{\chi_\theta^\ell}$ such that $a(p)\neq0$. Thus there exists $f\in\Hom_A(P_1^{\oplus\ell},P_0^{\oplus\ell})$ such that $a_f(p)\neq0$. In particular, $f\circ{-}\colon \Hom_A(P_0^{\oplus\ell},X_p)\to\Hom_A(P_1^{\oplus\ell},X_p)$ is an isomorphism, and hence $X\in\WW_f$.
\end{proof}

As an application of Lemma \ref{W_theta=W_f}, we prove the following linear algebraic statement.

\begin{example}\label{W_f linear alg}
Let $V,W$ be finite dimensional $k$-vector spaces, $H$ be a $k$-vector subspace of $\Hom_k(V,W)$, and $(a,b):=(\dim_kW,\dim_k V)$.
Assume that, for any $k$-vector subspace $V'$ of $V$, the $k$-vector subspace
\begin{align*}
HV':=\left\{\sum_{i=1}^mh_i(v_i) \mid m \in \Z_{\ge 1}, \ 
h_i\in H, \ v_i \in V'\right\}
\end{align*}
of $W$ satisfies $a \dim_k V' \le b \dim_k(HV')$. Then there exists $\ell\ge1$ such that the $k$-vector subspace $\Mat_{b\ell,a\ell}(H)$ of $\Hom_k(V^{\oplus a\ell},W^{\oplus b\ell})$ contains a $k$-linear isomorphism $V^{\oplus a\ell}\to W^{\oplus b\ell}$.
\end{example}

\begin{proof}
Consider the $k$-algebra $A:=\left[\begin{smallmatrix}k&H\\0&k\end{smallmatrix}\right]$. Then $A\simeq kQ$ for some generalized Kronecker quiver $Q$. Using the $k$-bilinear map $V\otimes_kH\to W$, we regard $X:=\left[\begin{smallmatrix}V & W\end{smallmatrix}\right]$ as an $A$-module.
We consider the functor $F:=\Hom_A(-,X) \colon \mod A\to\mod k$. We have obvious identifications
\begin{align*}\Hom_A(P(2),P(1))\simeq H,\ F(P(1))\simeq V\ \text{ and }\ F(P(2))\simeq W.\end{align*}
For each $i,j\ge0$, the composition
\begin{align*}
\Mat_{i,j}(H)&\simeq\Hom_A(P(2)^{\oplus j},P(1)^{\oplus i})\\
&\xrightarrow{F}\Hom_k(F(P(1)^{\oplus i}),F(P(2)^{\oplus j}))\simeq\Hom_k(V^{\oplus i},W^{\oplus j})\simeq\Mat_{j,i}(\Hom_k(V,W))
\end{align*}
coincides with the natural map $\Mat_{i,j}(H)\to\Mat_{j,i}(\Hom_k(V,W))$ induced by the inclusion $H\to\Hom_k(V,W)$ and transposes of matrices.

Now let $\theta:=a[P(1)]-b[P(2)] \in K_0(\proj A)$. Then $\theta(X)=0$ holds. Any $A$-submodule $X'$ of $X$ can be written as
$X'=\left[\begin{smallmatrix}V'&W'\end{smallmatrix}\right]$ for $k$-vector subspaces $V'$ and $W'$ of $V$ and $W$ respectively satisfying $HV' \subseteq W'$.
Our assumption implies that  $\theta(X') = a \dim_k V' - b \dim_k W' \le 0$. Thus $X \in \WW_\theta$ holds.
By Lemma \ref{W_theta=W_f}, there exist $\ell\ge1$ and $f\in\Hom(\ell\theta)$ such that $X \in \WW_f$.
This means that $f \colon P(2)^{\oplus b\ell}\to P(1)^{\oplus a\ell}$ induces a $k$-linear isomorphism $F(f) \colon V^{\oplus a\ell}=F(P(1)^{\oplus a\ell})\to W^{\oplus b\ell}=F(P(2)^{\oplus b\ell})$ in $\Mat_{b\ell,a\ell}(H)$.
\end{proof}

Now we can show the following special case of Theorem \ref{cup T_f}. 

\begin{lemma}\label{rank 2 case}
Let $A$ be a finite dimensional $k$-algebra with $|A|=2$ and $\theta \in K_0(\proj A)$. Then we have
\begin{align*}\WW_\theta=\WW^{\h}_{\N\theta}.\end{align*}
\end{lemma}

\begin{proof}
We can assume $A$ is basic.
The assertion is clear if $\theta\ge0$ or $\theta\le0$. Thus we can assume $\theta=a[eA]-b[fA]$ for a primitive idempotent $e\in A$, $f=1-e$ and $a,b\in\Z_{\ge0}$.

Consider a subalgebra $B$ of $A$ given by
\begin{align*}B:=\left[\begin{smallmatrix}k&eAf\\0&k\end{smallmatrix}\right]\subseteq A=\left[\begin{smallmatrix}eAe&eAf\\ fAe&fAf\end{smallmatrix}\right].\end{align*}
Let $\eta:=a[eB]-b[fB]\in K_0(\proj B)$.  Then $\theta=\eta\otimes A$ holds.
Applying Proposition \ref{W^A and W^B}, to prove $\WW^A_\theta=\WW^{A,\h}_{\N\theta}$, it suffices to prove $\WW^{B}_{\eta}=\WW^{B,\h}_{\N\eta}$. By replacing $A$ by $B$, we may assume $A=\left[\begin{smallmatrix}k&eAf\\0&k\end{smallmatrix}\right]$, 
and apply Lemma \ref{W_theta=W_f}.
\end{proof}

The following is a crucial step.
 
\begin{lemma}\label{Theta sub ThetaH}
Let $\theta\in K_0(\proj A)$. Then $\WW_\theta=\WW^{\h}_{\N\theta}$.
\end{lemma}

\begin{proof}
Let $X \in \mod A$ and $\theta\in K_0(\proj A)$ such that $X\in\WW_\theta$.
We prove $X\in\WW^{\h}_{\N\theta}$ by using the induction on $\dim_kX$.

Thanks to Proposition \ref{pi theta}, we can assume that $X$ is a sincere $A$-module by replacing $A$ by some $A/\langle e \rangle$. 
Moreover, if $|A|\le 2$, then $X\in\WW^{\h}_{\N\theta}$ holds by Lemma \ref{rank 2 case}. Thus we can assume $|A|\ge3$.

(i) First we assume that $\theta$ is a ray (i.e. one-dimensional subface) of $\Theta_X$.

If $X$ is simple in $\WW_\theta$, then Lemma \ref{simple and interior}(c) implies that $\Theta_X$ has dimension $|A|-1$ and $\theta\in\Theta_X^\circ$. This is not possible since $\theta$ is a ray of $\Theta_X$ and $|A|\ge3$ by our assumption.
Therefore $X$ is not simple in $\WW_\theta$.
Since the composition factors of $X$ in $\WW_\theta$ have smaller dimensions, they belong to $\WW^{\h}_{\N\theta}$ by the induction hypothesis. Since $\WW^{\h}_{\N\theta}$ is wide, we obtain $X\in\WW^{\h}_{\N\theta}$, as desired.

(ii) Now we consider general cases.

Let $\theta_1,\ldots,\theta_m\in K_0(\proj A)$ be the rays of $\Theta_X$. By (i), $X\in\WW^{\h}_{\N\theta_i}$ holds for each $i$.
Since $X$ is sincere by our assumption, $\Theta_X$ is strongly convex by Lemma \ref{simple and interior}(a). 
Thus there exist $a\in\Z_{\ge1}$ and $a_i\in\Z_{\ge0}$ such that $a\theta=\sum_{i=1}^ma_i\theta_i$. By the last assertion in Lemma \ref{cup T_f pre}, we obtain $X\in\WW^{\h}_{\N a\theta}=\WW^{\h}_{\N\theta}$ as desired.
\end{proof}

Now we are ready to prove Theorem \ref{cup T_f}.

\begin{proof}[Proof of Theorem \ref{cup T_f}]
We prove the assertion for $\overline{\TT}_\theta$. We have
\begin{align*}\overline{\TT}_\theta\stackrel{\text{Prop. \ref{W filtration}(b)}}{=}
\vecFilt_{\eta\in K_0(\proj A)_{\Q}^{\le\theta}} \limits \WW_\eta
\stackrel{\text{Lem. \ref{Theta sub ThetaH}}}{=}
\vecFilt_{\eta\in K_0(\proj A)_{\Q}^{\le\theta}} \limits \WW^{\h}_{\N\eta}
\subseteq\overline{\TT}^{\h}_{\N\theta},\end{align*}
where the last inclusion follows from $\WW^{\h}_{\N\eta}\subseteq\overline{\TT}^{\h}_{\N\eta}\subseteq\overline{\TT}^{\h}_{\N\theta}$ for each $\eta\le\theta$.

The assertion for $\FF_\theta$ follows from that for $\overline{\TT}_\theta$ and Lemma \ref{cup T_f pre}.
The remaining assertions are shown dually.
\end{proof}

\section{Ray condition and examples}\label{Section_ray}

Let $A$ be a finite dimensional algebra over an algebraically closed field $k$. In the rest of this subsection, we give information on Conjecture \ref{canonical and TF}.
In particular, we give an example of $\theta\in K_0(\proj A)$ such that $\cone(\ind\N\theta)$ is strictly bigger than $\cone(\ind\theta)$.

\subsection{Ray condition}

We recall that for a canonical decomposition $\theta=\theta_1\oplus\cdots\oplus\theta_\ell$, we set $\ind\theta=\{\theta_1\ldots,\theta_\ell\}$ and $|\theta|=\#\ind\theta$.
We consider the following condition.

\begin{definition}
We say that $A$ satisfies the \textit{ray condition} if, for each indecomposable wild element $\theta\in K_0(\proj A)$ and $\ell\ge1$, the element $\ell\theta$ is indecomposable.
\end{definition}

Later we show that the ray condition is satisfied by $E$-tame algebras and hereditary algebras (see Propositions \ref{linear independence2}, \ref{linear independence3}), and also  give an example of an algebra which does not satisfy the ray condition (see Example \ref{counter to ray 2}(c)).

In this subsection, we apply the ray condition to give more information on Conjecture \ref{canonical and TF}.

\begin{definition}
For $\theta\in K_0(\proj A)_\R$,
we set the $\R$-vector subspace 
\begin{align*}
W_\theta:=\langle[X]\mid X\in\WW_\theta\rangle_\R\subseteq K_0(\mod A)_\R.
\end{align*}
\end{definition}
Clearly, we have
\begin{equation}\label{W_theta}
W_\theta\subset\bigcap_{i=1}^\ell\Kernel\langle\theta_i,-\rangle.\end{equation}
The ray condition implies the following useful properties.

\begin{proposition}\label{linear independence}
Assume that $A$ satisfies the ray condition, and 
let $\theta=\theta_1\oplus\cdots\oplus\theta_\ell$ be a canonical decomposition such that $\theta_i \ne \theta_j$ for each $i \ne j$.
\begin{enumerate}[\rm(a)]
\item  $\theta_1,\ldots,\theta_\ell$ are linearly independent.
\item
Assume that $\theta_i$ for $1\le i\le \ell'$ is tame and $\theta_i$ for $\ell'<i\le\ell$ is wild.
Then for each $m \ge 1$, the canonical decomposition of $m \theta$ is
$m\theta=(\theta_1)^{\oplus m} \oplus \cdots \oplus (\theta_{\ell'})^{\oplus m}
\oplus (m \theta_{\ell'+1}) \oplus \cdots \oplus (m \theta_{\ell})$.
\item We have 
\begin{align*}
|A|\ge|\theta|=\dim\cone(\ind\theta)=\cone(\ind\N\theta)\ \mbox{ and }\ |A|-|\theta|\ge\dim_\R W_\theta.
\end{align*}
\end{enumerate}
\end{proposition}

\begin{proof}
(a)(i) First, we prove that $\R\theta_i\neq\R\theta_j$ for each $i\neq j$.

If $\R\theta_i=\R\theta_j$, then $a\theta_i=b\theta_j$ holds for some non-zero integers $a\neq b\in\Z$.
If $\theta_i$ is wild, then $a\theta_i$ is wild by the ray condition.
Since $\theta_i\oplus\theta_j$, we obtain $a\theta_i\oplus b\theta_j=a\theta_i\oplus a\theta_i$, a contradiction.
Thus $\theta_i$ is tame. Similarly, $\theta_j$ is also tame. Thus $\theta_i^{\oplus a}=\theta_j^{\oplus b}$ holds. This contradicts to the uniqueness of canonical decompositions since $\theta_i\neq\theta_j$ by our assumption.

(ii) We prove the assertion.

If they are not linearly independent, a certain non-trivial $\Z$-linear combination is zero. In particular, by changing indices, there is a relation
\begin{align*}\theta':=\sum_{i=1}^ma_i\theta_i=\sum_{i=m+1}^\ell a_i\theta_i\end{align*}
with $a_i\in\Z_{\ge0}$.
It suffices to show $\theta'=0$. Otherwise, thanks to the ray condition, by replacing each $a_i\theta_i$ by $\theta_i^{\oplus a_i}$ (if $\theta_i$ is tame) or $a_i\theta_i$ (if $\theta_i$ is wild), we obtain two canonical decompositions of $\theta'$, which are distinct by (i). This is a contradiction, and we obtain $\theta'=0$.

(b) By Proposition \ref{decompose_2step}(c),
$m \theta=m \theta_1 \oplus \cdots \oplus m \theta_\ell$ holds.
The ray condition implies that the canonical decomposition of $m \theta_i$ is
$(\theta_i)^{\oplus m}$ if $i \le \ell'$ and $m \theta_i$ if $i>\ell'$.
Then Proposition \ref{decompose_2step}(d) gives the assertion.

(c) follows from (a)(b) and \eqref{W_theta}.
\end{proof}

The following gives some relationship between $W_\theta$ and Conjecture \ref{canonical and TF}.

\begin{proposition}\label{dim of wide}
If $A$ satisfies the ray condition, then
\begin{align*}
\textup{(a)} \Longleftrightarrow \textup{(b)} \Longleftarrow \textup{(b)+(d)} 
\Longleftrightarrow \textup{(c)}
\end{align*}
hold, where
\begin{enumerate}[\rm(a)]
\item $[\theta]_{\rm TF}=\cone^\circ(\ind\theta)$ holds for each $\theta\in K_0(\proj A)$. (Equivalently, Conjecture \ref{canonical and TF} holds for $A$.)
\item $\dim_\R\langle[\theta]_{\rm TF}\rangle_\R=|\theta|$ holds for each $\theta\in K_0(\proj A)$.
\item $\dim_\R W_\theta=|A|-|\theta|$ holds for each $\theta\in K_0(\proj A)$.
\item $W_\theta=([\theta]_{\rm TF})^\perp$ holds for each $\theta\in K_0(\proj A)$.
\end{enumerate}
\end{proposition}

\begin{proof}
(b)+(d)$\Leftrightarrow$(c) Without loss of generality, we can assume $\theta=\theta_1\oplus\cdots\oplus\theta_\ell$ is a canonical decomposition with $\ell=|\theta|$. By Theorem \ref{decomposition and T}, we have $[\theta]_{\rm TF}\supseteq\cone^\circ\{\theta_1,\ldots,\theta_\ell\}$. By Proposition \ref{linear independence}(a), we have $\dim_\R\langle[\theta]_{\rm TF}\rangle_\R\ge\dim_\R\langle\theta_1,\ldots,\theta_\ell\rangle_\R=\ell$. Thus
\begin{align*}\dim_\R W_\theta\le|A|-\dim_\R\langle[\theta]_{\rm TF}\rangle_\R\le|A|-\ell.
\end{align*}
hold. Clearly, (b) holds if and only if the right equality holds, (d) holds if and only if the left equality holds, and the left-hand side equals the right-hand side if and only if (c) holds. Thus the assertion follows.

(a)$\Rightarrow$(b) is a direct consequence of Proposition \ref{linear independence}.

It remains to prove (b)$\Rightarrow$(a).
Without loss of generality, we can assume $\theta=\theta_1\oplus\cdots\oplus\theta_\ell$ is a canonical decomposition with $\ell=|\theta|$.
By Theorem \ref{decomposition and T}, we get $\cone^\circ\{\theta_1,\ldots,\theta_\ell\}\subset [\theta]_{\rm TF}$. 
By Proposition \ref{linear independence}, 
we have $\theta_1,\theta_2,\ldots,\theta_\ell$ are linearly independent,
so the assumption (b) tells us that $[\theta]_{\rm TF} \subset \langle\theta_1,\ldots,\theta_\ell\rangle_\R$.
These and the convexity of $[\theta]_{\rm TF}$ imply that it remains to show
\begin{align*}[\theta]_{\rm TF}\cap \cone\{\theta_1,\ldots,\theta_\ell\}\subseteq\cone^\circ\{\theta_1,\ldots,\theta_\ell\}.\end{align*}
Let $\eta=\sum_{i=1}^\ell a_i\theta_i$ in the left-hand side. By Theorem \ref{decomposition and T}, $\eta$ is TF equivalent to $\sum_{a_i\neq0}\theta_i$, and so is $\theta$. By assumption, we can apply (b) to the direct summand $\sum_{a_i\neq0}\theta_i$ of $\theta$. Then we have $\dim_\R\langle[\eta]_{\rm TF}\rangle_\R=\{i\mid a_i\neq0\}$. Thus, if $[\theta]_{\rm TF}=[\eta]_{\rm TF}$, then $a_i\neq 0$ holds for each $i$. Thus the assertion holds.
\end{proof}

It is natural to pose the following conjecture.

\begin{conjecture}\label{span TF and W_theta}
For each $\theta\in K_0(\proj A)$, we have
\begin{align*}
\dim_\R W_\theta=|A|-\dim\cone(\ind\N\theta).
\end{align*}
\end{conjecture}

Notice that, under the ray condition, this conjecture is equivalent to
\begin{equation}\label{span TF and W_theta 2}
\dim_\R W_\theta=|A|-|\theta|.
\end{equation}
Thus it is equivalent to the equality in \eqref{W_theta}, and implies Conjecture \ref{canonical and TF} by Proposition \ref{dim of wide}. 
When $\theta$ is indecomposable, the validity of the equality \eqref{span TF and W_theta 2} was asked in \cite[Question 5.5]{Fei}. In Example \ref{counter to ray 2}(e) below, we will see that \eqref{span TF and W_theta 2} does not necessarily hold (without assuming the ray condition).

Now we verify Conjecture \ref{span TF and W_theta} for rigid elements.

\begin{proposition}\label{W_theta=A-theta}
Assume that $\theta$ is rigid. Then $\WW_\theta$ has $|A|-|\theta|$ isoclasses of simple objects, which are linearly independent in $K_0(\mod A)$. In particular, Conjecture \ref{span TF and W_theta} holds true.
\end{proposition}

\begin{proof}
Let $\theta$ be rigid. 
Take the 2-term presilting complex $U$ with $[U]=\theta$ and its Bongartz completion $T$.
Then Proposition \ref{cone-TF} and the argument in \cite[Subsection 4.1]{A} give that 
the set of isoclasses of simple objects in $\WW_\theta$ has $|A|-|U|=|A|-|\theta|$ elements, and that is contained in the 2-term simple-minded collection in $\DDD^{\bo}(\mod A)$ corresponding to $T$ in \cite[Corollary 4.3]{BY}.
Any 2-term simple-minded collection in $\DDD^{\bo}(\mod A)$ gives a $\Z$-basis of $K_0(\mod A)$ by \cite[Lemma 5.3]{KY}.
Therefore $\dim_\R W_\theta=|A|-|\theta|$ holds.
\end{proof}

In Theorem \ref{describe TF for hereditary}, we will show that Conjecture \ref{span TF and W_theta} holds for hereditary algebras.

\subsection{Example}

It was asked in \cite[Question 4.7]{DF} that if an arbitrary finite dimensional $k$-algebra satisfies the ray condition.
In this section, we show that this is not the case by giving an explicit example. On the other hand, the ray condition is satisfied by $E$-tame algebras and hereditary algebras (see Propositions \ref{linear independence2}, \ref{linear independence3}).

Our example which does not satsify the ray condition is closely related to the comparison of $\WW_\theta$ and $\WW^{\h}_\theta$. Let $A$ be a finite dimensional algebra and $\theta\in K_0(\proj A)$. For $X\in\mod A$, let 
\begin{align*}
S^A_{X,\theta}=S_{X,\theta}:=\{\ell\in\Z_{\ge0}\mid X\in\overline{\TT}^{\h}_{\ell\theta}\}.
\end{align*}
Clearly this is a submonoid of $\Z_{\ge0}$. Moreover, by Theorem \ref{cup T_f}, $X\in\overline{\TT}_\theta$ holds if and only if $S_{X,\theta}$ contains a non-zero element. 
It is in general hard to determine the monoid $S_{X,\theta}$.
The following example is a generalization of \cite[Example 3.7]{Fei} for $n=3$.

\begin{example}\label{kronecker}
Let $n\ge 3$ be an odd integer.
Let $A=k\left[\xymatrix{1\ar@<4pt>[r]^{a_1}\ar@{}[r]|{\cdots}\ar@<-4pt>[r]_{a_n}&2}\right]=\begin{bmatrix}k&k^n\\0&k\end{bmatrix}$, 
$\theta=P(1)-P(2)$, and $X=[X_1\ X_2]$ the $A$-module given by
\begin{align*}
X_1=X_2=V:=k^n,\ X_{a_i}=F_i:=E_{i,i+1}-E_{i+1,i},
\end{align*}
where $E_{ij}$ is a matrix of size $n$ whose $(i,j)$-entry is $1$ and the others are zero, and $n+1:=1$.
Then $X\in\overline{\TT}_\theta$ and $S_{X,\theta}=\Z_{\ge0}\setminus\{1\}$ hold.
\end{example}

\begin{proof}
Since $\theta(X)=0$, $X \in \overline{\TT}_{\ell\theta}^{\h}$ is equivalent to $X \in \overline{\FF}_{\ell\theta}^{\h}$ by Lemma \ref{Nakayama T_f}(c).

(i) We prove $X\in\overline{\FF}_\theta$ directly, that is, $\dim_k U\le\dim_k\sum_{i=1}^nF_i(U)$ holds for any subspace $U$ of $V$. 
Assume the contrary $\dim_k U>\dim_k\sum_{i=1}^nF_i(U)$.
For each $i\in\Z/n\Z$, let $G_i:=\sum_{j=1}^{\frac{n-1}{2}}F_{i+2j-1}$. Then $\Kernel G_i$ is spanned by $e_i$,
where $e_i$ is the element of $k^n$ whose $i$th entry is 1 and the others are zero. Since $\dim_kU > \dim_k\sum_{i=1}^nF_i(U) \ge \dim_kG_i(U)$, we have $e_i\in U$ for each $i$. Thus $U=V$ holds. Since $e_i,e_{i+1}\in F_i(V)$, we have $V=\sum_{i=1}^nF_i(V)$, a contradiction.

(ii) We prove $X \notin \overline{\TT}^{\h}_\theta$. We set $f_x := (x \cdot) \colon P(2) \to P(1)$ for each $x \in e_1Ae_2$.
Any element $\Hom_A(P(2),P(1))$ is of the form $f_x$,
where $x=\sum_{i=1}^np_ia_i$ for some $p_i\in k$. Then
\begin{align*}
\left[\Hom_A(f_x,X)\colon\Hom_A(P(1),X)\to\Hom_A(P(2),X)\right]=\left[
\sum_{i=1}^np_iF_i:V\to V\right],\end{align*}
which is never an isomorphism, since the matrix in the right hand side is a skew symmetric matrix of odd size and hence the determinant is zero. Thus $X\notin\overline{\TT}_{f_x}$ holds.

(iii) We prove $S_{X,\theta}=\Z_{\ge0}\setminus\{1\}$. Since $S_{X,\theta}$ is a monoid, it suffices to show $X\in\WW^{\h}_{\ell\theta}$ for $\ell=2,3$. 
Let $x:=\sum_{i=1}^{\frac{n-1}{2}}a_{2i},\ y:=\sum_{i=1}^{\frac{n-1}{2}}a_{2i-1}\in e_1Ae_2$,
\[z_2:=\left[\begin{smallmatrix}x&a_n\\ a_n&y\end{smallmatrix}\right]\in M_2(e_1Ae_2)\ \mbox{ and }\ z_3:=\left[\begin{smallmatrix}x&a_n&O\\ a_n&y&y\\ O&y+a_n&x\end{smallmatrix}\right]\in M_3(e_1Ae_2).\]
For $\ell=2,3$, the morphism $f_{z_\ell}:=(z_\ell\cdot):P(2)^{\oplus\ell}\to P(1)^{\oplus\ell}$ induces an isomorphism $\Hom_A(f_{z_\ell},X):\Hom_A(P(1)^{\oplus\ell},X)\to\Hom_A(P(2)^{\oplus\ell},X)$. 
In fact, it is easily checked that the matrices $\left[\begin{smallmatrix}G_1&F_n\\ F_n&G_n\end{smallmatrix}\right]\in M_{2n}(k)$ and $\left[\begin{smallmatrix}G_1&F_n&O\\ F_n&G_n&G_n+F_n\\ O&G_n&G_1\end{smallmatrix}\right]\in M_{3n}(k)$ are invertible.
\end{proof}

Now we apply the monoid $S_{X,\theta}$ to construct an example of exotic behavior of canonical decompositions.
For $\eta,\theta\in K_0(\proj A)$, let 
\begin{align*}
S^A_{\eta,\theta}=S_{\eta,\theta}:=\{\ell\in\Z_{\ge0}\mid\eta\oplus(\ell\theta)\}.
\end{align*}
Clearly this is a submonoid of $\Z_{\ge0}$ too.
To explain a connection between this type of monoids $S_{\eta,\theta}$ and the previous one $S_{X,\theta}$, we consider the following setting.

Let $B$ be a finite dimensional $k$-algebra, $e\in B$ an idempotent and $A:=eBe$.
We have a fully faithful functor $-\otimes_A(eB):\proj A\to\proj B$, which induces an inclusion $-\otimes_A(eB):K_0(\proj A)\to K_0(\proj B)$. 

\begin{proposition}\label{compare monoid}
Under the setting above, let $\theta\in K_0(\proj A)$ and $P\in\proj B$. 
\begin{enumerate}[\rm(a)]
\item $Pe\in\overline{\TT}^{A,\h}_\theta$ holds if and only if $E(\theta\otimes eB,[P])=0$ holds if and only if $[P]\oplus(\theta\otimes eB)$ holds.
\item $S^A_{Pe,\theta}=S^B_{[P],\theta\otimes eB}$ holds. Thus for $\ell\in\Z_{\ge0}$, $Pe\in\overline{\TT}^{A,\h}_{\ell\theta}$ holds if and only if $E(\ell(\theta\otimes eB),[P])=0$ holds if and only if $[P]\oplus \ell(\theta\otimes eB)$ holds.
\end{enumerate}
\end{proposition}

\begin{proof}
(a) 
Fix $f\in\Hom_A(\theta)$.
By the first equality of Example \ref{W^A and W^B 2}(a), $Pe\in\overline{\TT}^{A}_f$ holds if and only if $Pe\in\overline{\TT}^{B}_{f\otimes eB}$ holds if and only if $\Hom_{\DDD(B)}(P_{f\otimes eB},P[1])=0$.
The map $\Hom_A(\theta)\to\Hom_B(\theta\otimes eB)$, $f\mapsto f\otimes eB$ is bijective. Thus $Pe\in\overline{\TT}^{A}_f$ holds for some $f\in\Hom_A(\theta)$ if and only if $\Hom_{\DDD(B)}(P_g,P[1])=0$ holds for some $g\in\Hom_B(\theta\otimes eB)$, that is, $E(\theta\otimes eB,[P])=0$. This is equivalent to $[P]\oplus(\theta\otimes eB)$ by Proposition \ref{decomposition}(a) since $E([P],\theta\otimes eB)=0$ clearly holds.

(b) Immediate from (a).
\end{proof}

Now we are ready to prove Theorem \ref{counter to ray intro}. A concrete example is given as follows.

\begin{example}\label{counter to ray 2}
Let $A$ and $X$ be the $k$-algebra and the $A$-module given in Example \ref{kronecker} respectively, and let
\[B:=\begin{bmatrix}k&X\\0&A\end{bmatrix}=\begin{bmatrix}k&X_1&X_2\\0&k&k^n\\ 0&0&k\end{bmatrix}\ni e_0:=\left[\begin{smallmatrix}1&0&0\\ 0&0&0\\ 0&0&0\end{smallmatrix}\right],\  e_1:=\left[\begin{smallmatrix}0&0&0\\ 0&1&0\\ 0&0&0\end{smallmatrix}\right],\ 
e_2:=\left[\begin{smallmatrix}0&0&0\\ 0&0&0\\ 0&0&1\end{smallmatrix}\right],\]
and $P(i):=e_iB$ for $i=0,1,2$. Then the following assertions hold.
\begin{enumerate}[\rm(a)]
\item We have
$S_{[P(0)],[P(1)]-[P(2)]}=\Z_{\ge0}\setminus\{1\}$. For $\ell\ge0$, $E(\ell[P(1)]-\ell[P(2)],[P(0)])=0$ holds if and only if $\ell\neq1$.
\item Let $\eta:=[P(0)]+[P(1)]-[P(2)]$. Then for $\ell\ge1$, we have a canonical decomposition 
\[\ell\eta=\left\{\begin{array}{ll}
\eta&\ell=1\\ 
{[P(0)]}^{\oplus\ell}\oplus (\ell[P(1)]-\ell[P(2)])& \ell\ge2.\end{array}\right.\]
\item The algebra $B$ does not satisfy the ray condition. More precisely, $\eta$ is indecomposable wild, but $\ell\eta$ is not indecomposable for each $\ell\ge2$.
\item For each $\ell\ge2$, we have
\[\cone(\ind\eta)\subsetneq\cone(\ind\ell\eta)=\cone(\ind\N\eta).\]
\item The element $\eta$ does not satisfy the equality \eqref{span TF and W_theta 2}. More explicitly, $\WW_{\eta}=\WW_{[P(0)]}\cap \WW_{[P(1)]-[P(2)]}$ holds. \end{enumerate}
\end{example}

\begin{proof}
(a) We apply Proposition \ref{compare monoid}(b) to our $B$ and $e:=e_1+e_2$. Since $P(0)e=e_0Be=X$, we obtain $S^B_{[P(0)],[P(1)]-[P(2)]}=S^A_{X,[P_A(1)]-[P_A(2)]}=\Z_{\ge0}\setminus\{1\}$ by Example \ref{kronecker}.

(b) Assume $\ell=1$. If $\eta$ is not indecomposable, at least one of $[P(0)]$, $[P(1)]$ and $-[P(2)]$ is a direct summand of $\eta$.
But this is impossible since $E([P(1)]-[P(2)],[P(0)])\neq0$ holds by (a), and $E([P(0)]-[P(2)],[P(1)])\neq0$ and $E(-[P(2)],[P(0)]+[P(1)])\neq0$ clearly hold.

Assume $\ell\ge 2$. Then $[P(0)]\oplus(\ell[P(1)]-\ell[P(2)])$ holds by (a). Moreover, since $n\ge3$, $\ell[P_A(1)]-\ell[P_A(2)]$ is indecomposable and so is $\ell[P(1)]-\ell[P(2)]$. Thus the assertion follows.

(c)(d) Immediate from (b).

(e) Since $2\eta=[P(0)]^{\oplus 2}\oplus (2[P(1)]-2[P(2)])$ holds by (b), we have
$\WW_\eta=\WW_{2\eta}=\WW_{[P(0)]}\cap \WW_{2[P(1)]-2[P(2)]}=\WW_{[P(0)]}\cap \WW_{[P(1)]-[P(2)]}$ by Theorem \ref{decomposition and T}(a).
\end{proof}

For example, for $n=3$, the algebra $B$ is $k\left[\xymatrix{0\ar@<4pt>[r]^{a'}\ar[r]|{b'}\ar@<-4pt>[r]_{c'}&1\ar@<4pt>[r]^a\ar[r]|b\ar@<-4pt>[r]_c&2}\right]/\left\langle a'b+b'a, b'c+c'b, c'a+a'c\right\rangle$.

\section{$E$-tame algebras and TF equivalence classes}

\subsection{$g$-tame and $E$-tame algebras}

The following classes of algebras are most basic from the point of  view of tilting theory.

\begin{definition}\label{define g-finite}
Let $A$ be a finite dimensional algebra.
\begin{enumerate}[\rm(a)]
\item \cite[Proposition 3.9]{DIJ} We say that $A$ is \emph{$\tau$-tilting finite} if $\#\twosilt A<\infty$.
\item \cite[Definition 3.23]{BST} We say that $A$ is \emph{$\tau$-tilting tame} if $\overline{\Wall}$ has measure zero.
\item We say that $A$ is \textit{$g$-finite} if $\Cone=K_0(\proj A)_\R$.
\item We say that $A$ is \textit{$g$-tame} if $\Cone$ is dense in $K_0(\proj A)_\R$. 
\end{enumerate}
\end{definition}

The conditions (a) and (c) are known to be equivalent.

\begin{proposition}\label{g-fin}\cite{ZZ}\cite[Theorem 4.7]{A}
A finite dimensional algebra is $\tau$-tilting finite if and only if it is $g$-finite.
\end{proposition}

The notion of $E$-invariants gives the following similar notions.

\begin{definition}\label{define E-finite}
Let $A$ be a finite dimensional algebra.
\begin{enumerate}[\rm(a)]
\item We say that $A$ is \textit{$E$-finite} if any $\theta\in K_0(\proj A)$ is rigid, that is, there exists a 2-term presilting complex $T$ such that $[T]=\theta$.
\item We say that $A$ is \textit{$E$-tame} if any $\theta\in K_0(\proj A)$ is tame, that is, $E(\theta,\theta)=0$ holds.
\end{enumerate}
These conditions are equivalent to that any indecomposable element is rigid or tame respectively. 
\end{definition}

These properties are preserved under the following operations.

\begin{proposition}
The following assertions hold.
\begin{enumerate}[\rm(a)]
\item If $A$ is $E$-tame (respectively, $E$-finite), then so is $eAe$ for all idempotents $e$ of $A$.
\item If $A$ is $E$-tame (respectively, $E$-finite), then so is $A/I$ for all ideals $I$ of $A$.
\end{enumerate}
\end{proposition}

\begin{proof}
(a)
This is clear since the functor $- \otimes_{eAe} eA \colon \proj eAe \to \proj A$
is fully faithful. 

(b)
Let $\theta \in K_0(\proj A)$.
It suffices to show $\theta \otimes_A (A/I) \in K_0(\proj (A/I))$ is tame.
Since $A$ is $E$-tame, there exist $f,g \in \Hom_A(\theta)$ such that $E(f,g)=E(g,f)=0$.
By the same argument as the proof of Example \ref{pi theta}(c),
we have $E(f',g')=E(g',f')=0$, where $f':=f \otimes_A (A/I)$ and $g':=g \otimes_A (A/I)$.
Thus $\theta \otimes_A (A/I) \in K_0(\proj (A/I))$.
\end{proof}

Plamondon and Yurikusa proved that any representation-tame algebras (including representation-finite algebras) are both $E$-tame and $g$-tame based on results in  \cite{CB-tame} and \cite[Theorem 3.2]{GLFS}.

\begin{proposition}\label{tameness}
Let $A$ be a representation-tame algebra.
\begin{enumerate}[\rm(a)]
\item \cite[Theorem 4.1]{PY} $A$ is $g$-tame.
\item \cite[Theorem 3.8]{PY} $A$ is $E$-tame. Moreover, if $\theta\in K_0(\proj A)$ is indecomposable non-rigid, then for any general $f\in\Hom(\theta)$, $C_f\simeq K_{\nu f}$ are bricks.
\end{enumerate}
\end{proposition}

Figure \ref{tame figure} shows connections between various finiteness and tameness introduced in Definitions \ref{define g-finite} and \ref{define E-finite}.
We conjecture that the unknown implications $\xymatrix{\ar@{=>}[r]|?&}$ also hold true.
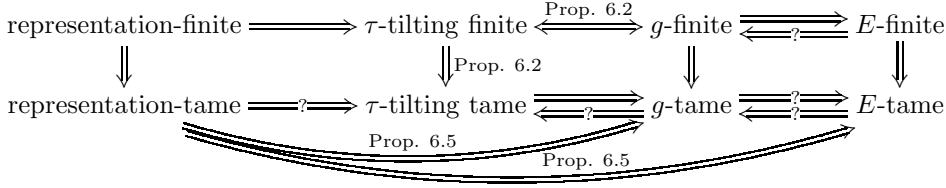
\begin{figure}
\begin{align*}\xymatrix@R=1.5em@C=4em{
\text{representation-finite}\ar@{=>}[r]\ar@{=>}[d]&\text{$\tau$-tilting finite}\ar@{<=>}[r]^-{\text{Prop. \ref{g-fin}}}\ar@{=>}[d]^-{\text{Prop. \ref{g-fin}}}&\text{$g$-finite}\ar@<3pt>@{=>}[r]\ar@{=>}[d]&\text{$E$-finite}\ar@<3pt>@{=>}[l]|?\ar@{=>}[d]\\
\text{representation-tame}\ar@{=>}[r]|-?\ar@/_20pt/@{=>}[rr]^-{\text{Prop. \ref{tameness}}}\ar@<-3pt>@/_25pt/@{=>}[rrr]^(.6){\text{Prop. \ref{tameness}}}&\text{$\tau$-tilting tame}\ar@<3pt>@{=>}[r]&\text{$g$-tame}\ar@<3pt>@{=>}[r]|?\ar@<3pt>@{=>}[l]|-?&\text{$E$-tame}\ar@<3pt>@{=>}[l]|?}
\end{align*}\caption{Relationship between finiteness and tameness}\label{tame figure}\end{figure}

\subsection{TF equivalence classes of $E$-tame algebras}
Let $A$ be a finite dimensional algebra and $\theta=\theta_1 \oplus \cdots \oplus \theta_\ell$ a canonical decomposition.
In Theorem \ref{decomposition and T}, we proved that 
$\cone^\circ\{\theta_1,\ldots,\theta_\ell\}$ is contained in the TF equivalence class $[\theta]_\mathrm{TF}$. In Conjecture \ref{canonical and TF}, we conjectured that these sets coincide. The following main result of this section gives a positive answer for $E$-tame algebras.

\begin{theorem}\label{describe TF for E-tame}
Assume that $A$ is a finite dimensional $E$-tame algebra over an algebraically closed field $k$. 
Let $\theta \in K_0(\proj A)$ and 
$\theta=\bigoplus_{i=1}^\ell \theta_i$ be the canonical decomposition. Then
\begin{align*}[\theta]_{\rm TF}=\cone^\circ\{\theta_1,\ldots,\theta_\ell\}.\end{align*}
\end{theorem}

It suffices to prove the ``$\subseteq$'' part. 
We start with the following basic properties.

\begin{proposition}\label{linear independence2}
Let $A$ be a finite dimensional algebra which is $E$-tame.
\begin{enumerate}[\rm(a)]
\item $A$ satisfies the ray condition.
\item Let $\theta=\theta_1\oplus\cdots\oplus\theta_\ell$ be a canonical decomposition such that $\theta_i \ne \theta_j$ for each  $i \ne j$. Then $\theta_1,\ldots,\theta_\ell$ are linearly independent.
In particular, $\ell\le |A|$ holds.
\end{enumerate}
\end{proposition}

\begin{proof}
(a) is clear since there is no wild element in $K_0(\proj A)$. (b) follows from Proposition \ref{linear independence}.
\end{proof}

The main tools in this section are the following subsets of $K_0(\proj A)_\R$. 

\begin{definition}\label{D f}
For each $\eta \in K_0(\proj A)$ and $f \in \Hom(\eta)$, we set
\begin{align*}
D_f&:=\{ \theta \in K_0(\proj A)_\R \mid 
\TT_f \subseteq \overline{\TT}_\theta, \ \FF_f\subseteq \overline{\FF}_\theta\}=\{\theta\in K_0(\proj A)_\R \mid C_f\in\overline{\TT}_\theta,\ K_{\nu f}\in\overline{\FF}_\theta\},\\
D_\eta&:=\bigcup_{f\in\Hom(\eta)}D_f.
\end{align*}
\end{definition}

We collect basic properties of $D_f$.

\begin{lemma}\label{properties of D_f}
For each $\eta \in K_0(\proj A)$ and $f \in \Hom(\eta)$, the following assertions hold.
\begin{enumerate}[\rm(a)]
\item $D_f$ is a union of some TF equivalence classes.
\item $D_f$ is a rational polyhedral cone in $K_0(\proj A)_\R$.
\item The set $\{D_f \mid f \in \Hom(\eta)\}$ is finite.
\item If $C_f\simeq K_{\nu f}$, then $D_f=\Theta_{C_f}$.
\end{enumerate}
\end{lemma}

\begin{proof}
(a) is clear from definition. To prove (b) and (c), let 
\begin{align*}S_f:=\{\dimv  Y\mid\text{$Y$ is a factor module of $C_f$}\}\ \text{ and }\ S'_f:=\{\dimv  Y\mid\text{$Y$ is a submodule of $K_{\nu f}$}\}.\end{align*}
Then (b) follows from
\begin{equation}\label{describe D_f}
D_f=\{\theta\in K_0(\proj A)_\R\mid\theta(d)\ge0\ge\theta(d')\ \text{ for all }d\in S_f,\ d'\in S'_f\}.
\end{equation}
Since the set $\{\dim_kC_f,\dim_kK_{\nu f}\mid f\in\Hom(\eta)\}$ is finite, the set $\{S_f,S'_f\mid f\in\Hom(\eta)\}$ is also finite. Thus (c) follows from \eqref{describe D_f}.
Finally (d) follows from
\begin{align*}&D_f=\{\theta\in K_0(\proj A)\mid C_f\in\overline{\TT}_\theta,\ K_{\nu f}\in\overline{\FF}_\theta\}=\{\theta\in K_0(\proj A)\mid C_f\in\WW_\theta\}=\Theta_{C_f}.\qedhere\end{align*}
\end{proof}

Using results in Section \ref{section 4}, we are able to prove the following key properties.

\begin{proposition}\label{direct sum D eta}
Assume that $A$ is $E$-tame.  Let $\eta, \theta \in K_0(\proj A)$.
\begin{enumerate}[\rm (a)]
\item $\eta \oplus \theta$ holds if and only if $\theta \in D_\eta$ if and only if $\eta\in D_\theta$.
\item $D_\theta$ is a union of some TF equivalence classes and contains $[\theta]_{\rm TF}$.
\end{enumerate}
\end{proposition}

\begin{proof}
(a) Since $A$ is $E$-tame, $\theta$ is tame. Thus the assertion follows from Corollary \ref{C f T bar}(c)$\Leftrightarrow$(a).

(b) The first assertion  is immediate from Lemma \ref{properties of D_f}(a). Since $\theta\in D_\theta$ holds by (a), the second assertion follows.
\end{proof}

The following observation is crucial.

\begin{theorem}\label{summand not TF}
For a finite dimensional algebra $A$, the following conditions are equivalent.
\begin{enumerate}[\rm(a)]
\item $A$ is $E$-tame. 
\item Let $\eta, \theta \in K_0(\proj A)$. Then $\eta$ and $\theta$ are TF equivalent if and only if $\ind\eta=\ind\theta$.
\end{enumerate}
\end{theorem}

\begin{proof}
(a)$\Rightarrow$(b) The ``if part'' is Corollary \ref{ind and TF}. We prove the ``only if'' part.
Since $\theta \in [\eta]_{\rm TF}\subset D_\eta$ by Proposition \ref{direct sum D eta}(b), there exists $f \in \Hom(\eta)$ such that $\theta\in D_f$, 
that is, $C_f \in \overline{\TT}_\theta$ and $K_{\nu f} \in \overline{\FF}_\theta$.
By Lemma \ref{extend opposite}(c) and its dual, there exists sufficiently large $\ell\in\N$ such that $C_f \in \overline{\TT}_{\theta-\eta/\ell}$ and $K_{\nu f} \in \overline{\FF}_{\theta-\eta/\ell}$.
Thus $\theta-\eta/\ell \in D_f\subset D_\eta$ and hence $\eta\oplus (\ell\theta-\eta)$ by Proposition \ref{direct sum D eta}(a). The equality
\begin{align*}\theta^{\oplus\ell}=\ell\theta=\eta\oplus (\ell\theta-\eta)\end{align*}
and the uniqueness of canonical decompositions implies $\ind\eta\subset\ind\theta$. By symmetry, we have $\ind\eta=\ind\theta$.

(b)$\Rightarrow$(a) Let $\theta\in K_0(\proj A)$ be indecomposable. Since $\theta$ and $2\theta$ are TF equivalent, $2\theta$ has $\theta$ as a direct summand. Thus $2\theta=\theta\oplus\theta$ holds, and $A$ is $E$-tame.
\end{proof}

Immediately we obtain the following result.

\begin{lemma}\label{TF in subspace}
Assume that $A$ is $E$-tame. 
Let $\bigoplus_{i=1}^m\theta_i$ be a canonical decomposition such that $\theta_i\neq\theta_j$ for each $i\neq j$, and $\theta=\bigoplus_{i=1}^\ell \theta_i$ for $0\le \ell\le m$.
Then we have
\begin{align*}[\theta]_{\rm TF}\cap \cone\{\theta_1,\ldots,\theta_m\}=\cone^\circ\{\theta_1,\ldots,\theta_\ell\}.\end{align*}
\end{lemma}

\begin{proof}
Since ``$\supseteq$'' holds by Theorem \ref{decomposition and T}, it suffices to prove ``$\subseteq$''.
Take any element $\eta=\sum_{i=1}^m a_i\theta_i$ with $a_i\ge0$ in the left-hand side. By Theorem \ref{decomposition and T}, $\eta$ is TF equivalent to $\bigoplus_{a_i\neq0}\theta_i$, and so is $\theta$. By Theorem \ref{summand not TF}(a)$\Rightarrow$(b), $a_i\neq0$ holds if and only if $1\le i\le\ell$. Thus $\eta\in\cone^\circ\{\theta_1,\ldots,\theta_\ell\}$.
\end{proof}

We call $\theta\in K_0(\proj A)$ \textit{maximal} if, for each $\eta\in K_0(\proj A)$ such that $\theta\oplus\eta$, each indecomposable direct summand of $\eta$ appears in a canonical decomposition of $\theta$.

The following is clear from Proposition \ref{linear independence2}(b).

\begin{lemma}\label{extend to maximal}
Assume that $A$ is $E$-tame. Then for any $\theta\in K_0(\proj A)$, there exists $\theta'\in K_0(\proj A)$ such that $\theta\oplus\theta'$ holds and $\theta\oplus\theta'$ is maximal.
\end{lemma}

Now we are ready to prove Theorem \ref{describe TF for E-tame}.

\begin{proof}[Proof of Theorem \ref{describe TF for E-tame}]
(i) First, we prove the assertion for the case $\theta\in K_0(\proj A)$ is maximal.

By Proposition \ref{direct sum D eta}(a), we have $D_\theta\cap K_0(\proj A)=\{\eta\in K_0(\proj A)\mid \eta\oplus\theta\}$. Since $\theta$ is maximal, the right-hand side is contained in $\cone\{\theta_1,\ldots,\theta_\ell\}$. Since $D_\theta$ is a union of rational polyhedral cones by Lemma \ref{properties of D_f}(b), we have $D_\theta\subset\cone\{\theta_1,\ldots,\theta_\ell\}$.
By Proposition \ref{direct sum D eta}(b), we have inclusions
\begin{align*}[\theta]_{\rm TF}\subseteq D_\theta\subseteq \cone\{\theta_1,\ldots,\theta_\ell\}.\end{align*}
By Lemma \ref{TF in subspace}, we obtain $[\theta]_{\rm TF}=\cone^\circ\{\theta_1,\ldots,\theta_\ell\}$.

(ii) We prove the assertion for general cases.

By Lemma \ref{extend to maximal}, there exists $\theta'=\theta_{\ell+1}\oplus\cdots\oplus\theta_m$ such that $\theta\oplus\theta'$ is maximal. Applying (i) to $\theta\oplus\theta'$, we have
$\overline{[\theta\oplus\theta']_{\rm TF}}=\cone\{\theta_1,\ldots,\theta_m\}\ni\theta$.
By Proposition \ref{closure of TF}, we have
\begin{align*}[\theta]_{\rm TF}\subset\overline{[\theta\oplus\theta']_{\rm TF}}=\cone\{\theta_1,\ldots,\theta_m\}.\end{align*}
By Lemma \ref{TF in subspace}, we obtain $[\theta]_{\rm TF}=\cone^\circ\{\theta_1,\ldots,\theta_\ell\}$.
\end{proof}

Next we prove the following result which gives a useful description of the set $D_\eta$. 

\begin{theorem}\label{D f largest}
Assume that $A$ is $E$-tame. Let $\eta \in K_0(\proj A)$. 
\begin{enumerate}[\rm (a)]
\item For any general $f\in \Hom(\eta)$, we have $D_f=D_\eta$.
Thus $D_\eta$ is a rational polyhedral cone.
\item $D_\eta$ depends only on the TF equivalence of $\eta$.
\end{enumerate}
\end{theorem}

To prove Theorem \ref{D f largest}, we need to define the following set.

\begin{definition}
Let $\eta,\theta \in K_0(\proj A)$.
We define an open subset of $\Hom(\eta) \times \Hom(\theta)$ by
\begin{align*}
X_{\eta,\theta}:=\{(f,g) \in \Hom(\eta) \times \Hom(\theta)
\mid E(f,g)=E(g,f)=0\}.
\end{align*}
\end{definition}

We denote by $\pi_1\colon \Hom(\eta) \times \Hom(\theta) \to \Hom(\eta)$ the first projection.

\begin{lemma}\label{pi X}
Let $\eta,\theta \in K_0(\proj A)$. 
Then $\eta\oplus\theta$ if and only if $X_{\eta,\theta} \ne \emptyset$. In this case, $\pi_1(X_{\eta,\theta})$ contains an open dense subset of $\Hom(\eta)$. 
\end{lemma}

\begin{proof}
The first assertion follows from Proposition \ref{decomposition}.
Now assume $X_{\eta,\theta} \ne \emptyset$. Then $X_{\eta,\theta}$ is open dense in $\Hom(\eta) \times \Hom(\theta)$, and hence $\pi_1(X_{\eta,\theta})$ is dense in $\Hom(\eta)$.
Since $\pi_1(X_{\eta,\theta})$ is a constructible subset of $\Hom(\eta)$ by Chevalley's Theorem, it contains an open dense subset of $\Hom(\eta)$.
\end{proof}

We prepare the following technical observation.

\begin{lemma}\label{D f larger}
Assume that $A$ is $E$-tame. Let $\eta\in K_0(\proj A)$ and $\theta_1,\ldots,\theta_m\in D_\eta\cap K_0(\proj A)$. Then $\theta_1,\ldots,\theta_m \in D_f$ holds for any general $f \in \Hom(\eta)$. 
\end{lemma}

\begin{proof}
By Proposition \ref{direct sum D eta}(a), $\eta\oplus\theta_i$ holds for each $i$. By Lemma \ref{pi X}, $\pi_1(X_{\eta,\theta_i})$ contains an open dense subset of $\Hom(\eta)$.
Thus $\bigcap_{i=1}^m \pi_1(X_{\eta,\theta_i})$ also contains an open dense subset $U$ of $\Hom(\eta)$. Then any $f \in U$ satisfies that $\theta_1,\ldots,\theta_m\in D_f$ by Corollary \ref{C f T bar 2}(c)$\Leftrightarrow$(a).
\end{proof}

We are ready to prove Theorem \ref{D f largest}.

\begin{proof}[Proof of Theorem \ref{D f largest}]
(a) By Lemma \ref{properties of D_f}(c), there are finitely many elements $f_1,\ldots,f_m \in \Hom(\eta)$ such that 
\begin{align*}\{D_f \mid f \in \Hom(\eta)\}=\{D_{f_1},\ldots,D_{f_m}\}.\end{align*}
For each $1\le i\le m$, since $D_{f_i}$ is a rational polyhedral cone, there exist $\theta_{i,1},\ldots,\theta_{i,\ell_i} \in K_0(\proj A)$ satisfying
$D_{f_i}=\cone\{\theta_{i,1},\ldots,\theta_{i,\ell_i}\}$.
Applying Lemma \ref{D f larger} to the set $S:=\{\theta_{i,j}\mid 1\le i\le m,1\le j\le\ell_i\}$, we obtain $S\subset D_f$ for any general $f \in \Hom(\eta)$. Clearly such $D_f$ coincides with $D_\eta$.

(b) Assume that $\theta\in K_0(\proj A)$ is TF equivalent to $\eta$. Then $D_\theta\cap K_0(\proj A)=D_\eta\cap K_0(\proj A)$ holds by Proposition \ref{direct sum D eta}(a) and Theorem \ref{summand not TF}(b). This implies $D_\theta=D_\eta$ since $D_\theta$ and $D_\eta$ are rational polyhedral cones by (a).
\end{proof}

In the later section, we need the following observation on the rational polyhedral cone $D_\eta$. 

\begin{proposition}\label{D f largest 2}
Assume that $A$ is $E$-tame. Let $\theta,\eta\in K_0(\proj A)$ such that $\eta$ is indecomposable and 
belongs to $D_\eta^\circ$.
Then the following conditions are equivalent.
\begin{enumerate}[\rm(a)]
\item $\theta=\eta \oplus (\theta-\eta)$.
\item $\theta\in D_\eta^\circ$.
\end{enumerate}
\end{proposition}

\begin{proof}
(a)$\Rightarrow$(b) By Proposition \ref{direct sum D eta}(a), $\theta-\eta\in D_\eta$. Since $\eta\in D_\eta^\circ$, we have $\theta=(\theta-\eta)+\eta\in D_\eta^\circ$.

(b)$\Rightarrow$(a) Since $\theta\in D_\eta^\circ$ and $\eta\in D_\eta$ hold, we have $\ell \theta-\eta \in D_\eta$ for sufficiently large $\ell$.
Thus Proposition \ref{direct sum D eta}(a) implies 
$\theta^{\oplus \ell}=\ell\theta=\eta \oplus (\ell\theta-\eta)$.
Since $\eta$ is indecomposable, $\theta$ has $\eta$ as a direct summand.
\end{proof}

It is an interesting question if $\eta\in D_\eta^\circ$ holds for each indecomposable $\eta\in K_0(\proj A)$.
The following observation gives a partial answer.  

\begin{proposition}\label{tame brick}
Let $\eta \in K_0(\proj A)$ be indecomposable.
\begin{enumerate}[\rm(a)] 
\item If $\eta$ is rigid, then $\eta\in D_\eta^\circ$.
\item If $A$ is representation-tame and $\eta$ is non-rigid,
then for any general $f \in \Hom(\eta)$, we have
\begin{align*}D_f=\Theta_{C_f}.\end{align*}
Moreover, $C_f$ is a simple object in $\WW_\eta$, $\dim D_f=|A|-1$ and $\eta$ belongs to $D_f^\circ \subset D_\eta^\circ$.
\end{enumerate}
\end{proposition}

\begin{proof}
(a) Take $U \in \twopresilt A$ satisfies $[U]=\eta$, 
then Proposition \ref{D f largest} and Example \ref{presilting case} implies that 
$D_\eta=\{ \theta \in K_0(\proj A)_\R \mid H^0(U) \in \overline{\TT}_\theta, \ H^{-1}(\nu U) \in \overline{\FF}_\theta \}$.
Since $H^0(U) \in \TT_\eta$ and $H^{-1}(\nu U) \in \FF_\eta$, the assertion follows.

(b) By Proposition \ref{tameness}(b), $C_f\simeq K_{\nu f}$ are bricks for any general $f \in \Hom(\eta)$. By Proposition \ref{properties of D_f}(d), we have $D_f=\Theta_{C_f}$.
By Lemma \ref{W_f and W_theta 2}(c), $C_f$ is a simple object in $\WW_\eta$.
By Lemma \ref{simple and interior}(c), the last two assertions follow.
\end{proof}

\section{Hereditary algebras and TF equivalence classes}\label{Section_hered} 

In this section, we assume that $A$ is hereditary. 
We will prove that every canonical decomposition gives a TF equivalence class also in this case.

First, hereditary algebras satisfy the same properties as Proposition \ref{linear independence2}, but the proof is different.

\begin{proposition}\label{linear independence3}
Let $A$ be a finite dimensional hereditary algebra.
\begin{enumerate}[\rm(a)]
\item $A$ satisfies the ray condition.
\item Let $\theta=\theta_1\oplus\cdots\oplus\theta_\ell$ be a canonical decomposition such that $\theta_i \ne \theta_j$ for each  $i \ne j$. Then $\theta_1,\ldots,\theta_\ell$ are linearly independent.
In particular, $\ell\le |A|$ holds.
\end{enumerate}
\end{proposition}

\begin{proof}
(a) is \cite[Theorem 3.7]{S}. 
(b) follows from (a) and Proposition \ref{linear independence}.
\end{proof}

As written in \cite{DF}, canonical decompositions of elements of $K_0(\proj A)$ is a generalization of original canonical decompositions of dimension vectors of quiver representations introduced by \cite{Kac}.
These two kinds of canonical decompositions are related as follows
for all finite dimensional algebras.

\begin{proposition}\label{tau-reduced}\cite[Theorem 1.2]{P}
Let $A$ be a finite dimensional algebra.

Assume that $\theta \in K_0(\proj A)$ has no negative direct summand.
Then the canonical decomposition of each presentation space $\Hom(\theta)$ gives 
the canonical decomposition of the corresponding $\tau$-reduced component 
of the module variety.

Moreover, if $U=U_1\oplus\cdots\oplus U_\ell\in\twopresilt A$ with $U_i$ indecomposable
and no $[U_i]$ is negative, then each general element in 
the $\tau$-reduced component containing $H^0(U) \in \mod A$ is isomorphic to 
$H^0(U_1) \oplus \cdots \oplus H^0(U_\ell)$.
\end{proposition}

$\tau$-reduced components were originally called strongly reduced components in \cite{P}.
We do not explain $\tau$-reduced components in this paper;
see \cite{P,GLFS,PY} for details.

If $A$ is hereditary, 
then the module variety $\mod(A,d)$ for each dimension vector $d$ is irreducible,
so the property above gives the following.

\begin{example}\label{real Schur}
Assume that $A$ is hereditary.
Proposition \ref{tau-reduced} gives a bijection
\begin{align*}\{\text{indecomposable elements of $K_0(\proj A)$}\}\simeq\{\text{Schur roots}\}\end{align*} 
which restricts to
\begin{align*}\{ \text{indecomposable non-negative rigid elements of $K_0(\proj A)$}\}
\simeq\{\text{real Schur roots}\}\end{align*}
given by $[U]\mapsto\dimv H^0(U)$ for each $U \in \twopresilt A$.
These are restrictions of the $\Z$-linear isomorphism $K_0(\proj A) \to K_0(\mod A)$ satisfying $[P(i)] \to [P(i)]$.
\end{example}

The following main result of this section shows that the Conjectures \ref{canonical and TF} and \ref{span TF and W_theta} hold for any hereditary algebra.

\begin{theorem}\label{describe TF for hereditary}
Assume that $A$ is a finite dimensional hereditary algebra over an algebraically closed field $k$. 
Let $\theta=\bigoplus_{i=1}^m \theta_i$ be a canonical decomposition in $K_0(\proj A)$
with $\theta_i \ne \theta_j$ if $i \ne j$. Then
\begin{align*}
\dim_\R W_\theta=n-m,\ W_\theta=\bigcap_{i=1}^m\Kernel\langle\theta_i,-\rangle\ {\rm and }\ [\theta]_{\rm TF}=\cone^\circ\{\theta_1,\ldots,\theta_m\}.
\end{align*}
\end{theorem}

To prove this, we use the union $\Theta_d:=\bigcup_{X \in \mod(A,d)} \Theta_X$
associated to each dimension vector $d \in K_0(\mod A)$.
Then $\Theta_d=\Theta_X$ holds for general $X \in \mod(A,d)$
\cite[Lemma 5.2]{A} since $\mod(A,d)$ is irreducible.
Its dimension is given by the proof of \cite[Theorem 5.1]{DW1}.

To use Proposition \ref{wall codim} for canonical decompositions in $K_0(\proj A)$,
we need the following duality lemma.

\begin{lemma}\label{W_f C_f K_f}
Assume that $A$ is hereditary.
Let $f,g$ be morphisms in $\proj A$ such that 
$P_f$ has no positive direct summand and $P_g$ has no negative direct summand. 
Then $C_g\in\WW_f$ if and only if $K_{\nu f}\in\WW_g$.
\end{lemma}

\begin{proof}
We recall that $X\in\WW_f$ if and only if $\Hom_A(C_f,X)=0=\Hom_A(X,K_{\nu f})$. 
Thus $C_g\in\WW_f$ holds if and only if
\begin{equation}\label{C_g in W_f}
\Hom_A(C_f,C_g)=0=\Hom_A(C_g,K_{\nu f}).
\end{equation} 
Since $A$ is hereditary, $\tau \colon \mod_PA\simeq\mod_IA$ is an equivalence, where $\mod_PA$ (respectively, $\mod_IA$) is a full subcategory of $\mod A$ consisting of $A$-modules without non-zero projective (respectively, injective) direct summands.
By our assumption, $\tau C_f \simeq K_{\nu f}$ and $\tau C_g \simeq K_{\nu g}$ hold, so \eqref{C_g in W_f} is equivalent to $\Hom_A(K_{\nu f},K_{\nu g})=0=\Hom_A(C_g,K_{\nu f})$. 
By the first remark again, this is equivalent to $K_{\nu f}\in\WW_g$.
\end{proof}

We denote by $\iota \colon K_0(\proj A) \to K_0(\mod A)$ the linear isomorphism corresponding to the equivalence $\KKK^{\bo}(\proj A)\simeq\DDD^{\bo}(\mod A)$.
Moreover, the Nakayama functor $\nu=-\Lotimes_ADA \colon \DDD^{\bo}(\mod A)\simeq\DDD^{\bo}(\mod A)$ induces an automorphism 
$\nu \colon K_0(\mod A) \to K_0(\mod A)$.

\begin{proposition}\label{W_f C_f K_f 2}
For each dimension vector $d\in K_0(\mod A)$, we have
\begin{align*}
\nu\circ\iota(\Theta_d)\subset W_{\iota^{-1}(d)}.
\end{align*}
\end{proposition}

\begin{proof}
It suffices to show that $\nu\circ\iota(\eta)\in W_{\iota^{-1}(d)}$ holds for each $X\in\mod(A,d)$ and $\eta\in\Theta_X$.
By Proposition \ref{wall direct summand}, we can assume that $\eta$ is indecomposable.
Also we can assume that $\eta$ is not positive since $-\eta$ also belongs to $\Theta_X$ in this case.

Let $\theta=\iota^{-1}(d)$, and take a minimal projective presentation $g \in \Hom(\theta)$ of $X$ so that $C_g=X$.
Since $X\in\WW_\eta$, by Theorem \ref{cup T_f}, there exist $\ell \in \Z_{\ge 1}$ and $f_\eta \in \Hom_A(\ell\eta)$ such that
$C_g=X \in \WW_{f_\eta}$ and $f_\eta$ has no positive direct summand. By Lemma \ref{W_f C_f K_f},
we get $K_{\nu f_\eta} \in \WW_g \subset \WW_\theta$ and hence $[K_{\nu f_\eta}] \in W_\theta$.
Thus $\ell\nu\circ\iota(\eta)=\nu\circ\iota(\ell\eta)=[K_{\nu f_\eta}]\in W_\theta$ holds since $f_\eta$ has no positive direct summand.
Consequently $\nu\circ\iota(\eta)\in W_\theta$.
\end{proof}

We prepare some terminology.
Let $A$ be a finite dimensional hereditary algebra.
As in \cite[Definition 4.1]{DW1}, a sequence of dimension vectors $(d_1,d_2,\ldots,d_m)$ in $K_0(\mod A)$ is called a \textit{Schur sequence} if
\begin{enumerate}[\rm (a)]
\item
for any $i$, $d_i$ is a Schur root; and
\item
if $i<j$, then any general $X \in \mod(A,d_i+d_j)$ admits
a unique submodule $Y \subset X$ such that $Y \in \mod(A,d_i)$.
\end{enumerate}
Then the proof of the well-definedness of the map $\psi(r)$ of 
\cite[Theorem 5.1]{DW1} actually implies the following property.

\begin{lemma}\label{Schur face}
Let $A$ be a finite dimensional hereditary algebra, 
$d \in K_0(\mod A)$ be a dimension vector,
and $(d_1,d_2,\ldots,d_m)$ be a Schur sequence with $d \in \sum_{i=1}^m \Z_{\ge 1}d_i$.
Then $\bigcap_{i=1}^m \Theta_{d_i}$ is an $(n-m)$-dimensional face of $\Theta_d$.
\end{lemma}

Then we have the following result.

\begin{proposition}\label{wall codim}
Let $A$ be a finite dimensional hereditary algebra,
and $d \in K_0(\mod A)_{\ge 0}$ be a dimension vector.
If $d=\bigoplus_{i=1}^m d_i^{\oplus s_i}$ is the canonical decomposition of 
the dimension vector $d$,
then the dimension of $\Theta_d$ as a rational polyhedral cone is $n-m$.
\end{proposition}

\begin{proof}
By the definition of canonical decompositions of dimension vectors,
we get $\Theta_d=\bigcap_{i=1}^m \Theta_{d_i}$.
Since $d_1,d_2,\ldots,d_m$ are linearly independent \cite[Corollary 4.12]{DW1}, 
we have the dimension of $\Theta_d$ is at most $n-m$.

Thus it remains to show that $\Theta_d$ has an $(n-m)$-dimensional face.
This follows from Lemma \ref{Schur face}
and that $(d_1,d_2,\ldots,d_m)$ can be reordered to a Schur sequence
by \cite[Remark 4.6]{DW1}.
\end{proof}

Now we are ready to prove Theorem \ref{describe TF for hereditary}.

\begin{proof}[Proof of Theorem \ref{describe TF for hereditary}]
By Propositions \ref{linear independence3} and \ref{dim of wide},
it suffices to prove $\dim_\R W_\theta \ge n-m$ holds for any $\theta \in K_0(\proj A)$.

(i)
We first consider the case that $\theta$ has no negative direct summand.
Let $d:=\iota(\theta)$. Then $d=\bigoplus_{i=1}^m \iota(\theta_i)^{\oplus s_i}$ is the canonical decomposition in $K_0(\mod A)$ by Example \ref{real Schur}.
Thus we have
\begin{align*}
\dim_\R W_\theta \stackrel{\text{Prop. \ref{W_f C_f K_f 2}}}{\ge} 
\dim_\R(\R\Theta_d) \stackrel{\text{Prop. \ref{wall codim}}}{=}|A|-|\theta|=n-m.
\end{align*}

(ii)
We consider general cases.
We set $\theta'$ as the maximal negative direct summand of $\theta$.
Then there uniquely exists an idempotent $e \in A$ such that $\theta' \in C^\circ(eA[1])$,
so consider the algebra $B:=A/\langle e \rangle$. 
By Theorem \ref{decomposition and T}, 
we have $\WW_\theta \subset \WW_{\theta'}=\mod B$, 
so Example \ref{pi theta} implies $\WW_\theta=\WW^B_{\theta \otimes B}$.
Since $B$ is hereditary and $\theta \otimes B$ has no negative direct summand,
we have
\begin{align*}
& \dim_\R W_\theta=\dim_\R \WW^B_{\theta \otimes B}
\stackrel{\text{(i)}}{\ge}|B|-|\theta \otimes B|=(|A|-|\theta'|)-(|\theta|-|\theta'|)
=n-m. \qedhere
\end{align*}
\end{proof}

\section{TF equivalence classes of preprojective algebras of type $\widetilde{\mathbb{A}}$}\label{Section_preproj}

\subsection{Our result}
In this section, we consider the complete preprojective algebra 
$\Pi$ of type $\widetilde{\mathbb{A}}_{n-1}$ with $n \ge 2$:
\begin{align*}
\begin{tikzpicture}[baseline=0pt,->]
\node (1) at ( 90:2cm) {$1$};
\node (2) at (150:2cm) {$2$};
\node (3) at (210:2cm) {$3$};
\node (4) at (270:2cm) {$\cdots$};
\node (5) at (330:2cm) {$n-1$};
\node (6) at ( 30:2cm) {$n$};
\draw [transform canvas={shift={(120:0.1cm)}}] (1) to [edge label'=$\scriptstyle \alpha_1$] (2);
\draw [transform canvas={shift={(180:0.1cm)}}] (2) to [edge label'=$\scriptstyle \alpha_2$] (3);
\draw [transform canvas={shift={(240:0.1cm)}}] (3) to [edge label'=$\scriptstyle \alpha_3$] (4);
\draw [transform canvas={shift={(300:0.1cm)}}] (4) to [edge label'=$\scriptstyle \alpha_{n-2}$] (5);
\draw [transform canvas={shift={(  0:0.1cm)}}] (5) to [edge label'=$\scriptstyle \alpha_{n-1}$] (6);
\draw [transform canvas={shift={( 60:0.1cm)}}] (6) to [edge label'=$\scriptstyle \alpha_n$] (1);
\draw [transform canvas={shift={(120:-0.1cm)}}] (2) to [edge label'=$\scriptstyle \beta_2$] (1);
\draw [transform canvas={shift={(180:-0.1cm)}}] (3) to [edge label'=$\scriptstyle \beta_3$] (2);
\draw [transform canvas={shift={(240:-0.1cm)}}] (4) to [edge label'=$\scriptstyle \beta_4$] (3);
\draw [transform canvas={shift={(300:-0.1cm)}}] (5) to [edge label'=$\scriptstyle \beta_{n-1}$] (4);
\draw [transform canvas={shift={(  0:-0.1cm)}}] (6) to [edge label'=$\scriptstyle \beta_n$] (5);
\draw [transform canvas={shift={( 60:-0.1cm)}}] (1) to [edge label'=$\scriptstyle \beta_1$] (6);
\end{tikzpicture}, \quad
\alpha_i\beta_{i+1}=\beta_{i+1}\alpha_i \quad (i \in \{1,2,\ldots,n\}),  
\end{align*}
where we set $\alpha_{i+n}:=\alpha_i$ and $\beta_{i+n}:=\beta_i$. 
It is well-known that $\Pi$ is infinite dimensional. The center of $\Pi$ is isomorphic to the simple surface singularity $k[[x,y,z]]/(x^{n+1}-yz)$ of type $A_n$, and $\Pi$ is its Auslander algebra, that is, the endomorphism algebra of direct sum of indecomposable Cohen-Macaulay $R$-modules.
More explicitly, $x$, $y$ and $z$ are given by
\begin{align*}x=\sum_{i=1}^n\alpha_i\beta_{i+1},\ y=\sum_{i=1}^n\alpha_i\alpha_{i+1}\cdots\alpha_{i+n-1}\ \mbox{ and }\ z=\sum_{i=1}^n\beta_i\beta_{i+1}\cdots\beta_{i+n-1}.\end{align*}
 In particular, $\KKK^{\bo}(\proj \Pi)$ is Krull-Schmidt (see also \cite[Corollary 4.6]{KM} for wild case), and an indecomposable decomposition of an object of $\KKK^{\bo}(\proj \Pi)$ is unique.
We refer to \cite{G,IK,Kimura,VG} for silting theory of Noetherian algebras.

We will determine the TF equivalence classes of $K_0(\proj \Pi)_\R$.
As usual, let $\proj \Pi$ be the category of finitely generated projective $\Pi$-modules, $K_0(\proj \Pi)$ is a free abelian group of rank $n$.
We set $\fl \Pi$ as the category of finite dimensional $\Pi$-modules, then $K_0(\fl \Pi)$ is also a free abelian group of rank $n$. The Euler form
\begin{align*}K_0(\proj\Pi)\times K_0(\fl\Pi)\to\Z,\ (X,Y)\mapsto\dim_k\Hom_\Pi(X,Y)\end{align*}
is non-degenerate, and we often regard $K_0(\proj \Pi)$ as the dual space of $K_0(\fl \Pi)$. For each $\theta\in K_0(\proj\Pi)_\R$, we have torsion pairs $(\overline{\TT}_\theta,\FF_\theta),(\TT_\theta,\overline{\FF}_\theta)$ in $\fl \Pi$ as in Definition \ref{define T_theta}, and we obtain the notion of \emph{TF equivalence} on $K_0(\proj\Pi)_\R$. 

We recall a classification of 2-term silting complexes in $\KKK^{\bo}(\proj \Pi)$ in terms of the Coxeter group $W$ of type $\widetilde{\mathbb{A}}_{n-1}$, see \cite{IR,BIRS,KM}.
Recall that $W$ is defined by generators $s_1,s_2,\ldots,s_n$ with relations $(s_is_j)^{m_{i,j}}=1_W$, where
\begin{align*}
m_{i,j}:=\begin{cases}
1 & (j=i) \\
3 & (n \ge 3, \ j=i\pm 1 + n\Z) \\
\infty & (n=2, \ j \ne i) \\
2 & \text{(otherwise)}
\end{cases}.
\end{align*}
For each $w \in W$, \cite[Theorem 3.1.9]{BIRS} constructed a tilting ideal $I_w \subset \Pi$, which we identify with its projective presentation as a $\Pi$-module.
Then we have a bijection \cite[Theorem 3.1]{KM}
\begin{equation}\label{WW}
W\sqcup W\simeq\twosilt\Pi
\end{equation}
given by maps
\begin{align*}W\to\twosilt\Pi,\ w\mapsto I_w\ \mbox{ and }\ W\to\twosilt\Pi,\ w\mapsto I_w^*[1]:=\RHom_{\Pi^{\op}}(I_w,\Pi)[1].\end{align*}
Notice that $I_w$ is a classical tilting $\Pi$-module, and $I_w^*[1]$ is a complex with $H^{-1}(I_w^*[1])=\Pi$ and $H^0(I_w^*[1])$ is a $\Pi$-module of finite length.
The bijection \eqref{WW} is compatible with
a canonical action of $W$ \cite[Subsection 4.2]{BB} on the Grothendieck group $K_0(\proj \Pi)_\R$, which is given by,
for each $1\le i,j\le n$,
\begin{align*}
s_j([P(i)])=\begin{cases}
-[P(i)]+[P(i-1)]+[P(i+1)] & (i=j) \\
[P(i)] & (i \ne j)
\end{cases}.
\end{align*}
As in \cite[Theorem 6.6]{IR} and \cite[Theorem 3.4]{KM}, we have
\begin{align*}C(I_w)=w(C(\Pi))\ \mbox{ and }\ C(I_w^*[1])=w(C(\Pi[1]))=-w(C(\Pi)).\end{align*}
We set
\begin{align*}h:=\sum_{i=1}^n [S(i)] \in K_0(\fl \Pi)\end{align*}
and the hyperplane
\begin{align*}H:=\Kernel \langle ?,h \rangle\subset K_0(\proj\Pi)_\R.\end{align*}
The argument in the proof of \cite[Proposition 3.6]{KM} actually shows that
\begin{align*}
H^+&:=\{ \theta \in K_0(\proj A)_\R \mid \theta(h)>0 \}=\bigcup_{w \in W}C(I_w)\setminus\{0\}, \\
H^-&:=\{ \theta \in K_0(\proj A)_\R \mid \theta(h)<0 \}=\bigcup_{w \in W}C(I_w^*[1])\setminus\{0\}.
\end{align*}
We give a description of TF equivalence classes of $K_0(\proj\Pi)_\R$ contained in $H^+\sqcup H^-$.
For each $J \subset \{1,2,\ldots,n\}$, we set $P_J:=\bigoplus_{j \in J} P(j)\in\proj\Pi$, and $W_J \subset W$ as the parabolic subgroup generated by $\{s_j\}_{j \notin J}$.

\begin{proposition}[{cf.\ \cite{IW}}]\label{half planes}
The following assertions hold.
\begin{enumerate}[\rm(a)]
\item We have a bijection
\begin{equation*}
\bigsqcup_{J \subset \{1,2,\ldots,n\}}(W/W_J\sqcup W/W_J)\simeq
\twopresilt\Pi
\end{equation*}
given by the maps
\begin{equation*}
W/W_J\to\twopresilt\Pi,\ w\mapsto P_J\otimes_\Pi I_w\ \mbox{ and }\ W/W_J\to\twopresilt\Pi,\ w\mapsto P_J\otimes_\Pi I_w^*[1].
\end{equation*}
\item We have a decomposition of $H^+\sqcup H^-$ into the TF equivalence classes
\begin{align*}
H^+ \sqcup H^-
=\bigsqcup_{\emptyset\neq J \subset \{1,2,\ldots,n\}, \ w \in W/W_J}w(C^\circ(P_J))\sqcup (-w(C^\circ(P_J)))=\bigsqcup_{U \in (\twopresilt \Pi) \setminus \{0\}}C^\circ(U).
\end{align*} 
\end{enumerate}
\end{proposition}

\begin{proof}
By \eqref{WW} and Bongartz completion, we have a surjection
\begin{equation}\label{WW surjection}
\bigsqcup_{J \subset \{1,2,\ldots,n\}}(W\sqcup W)\to\twopresilt\Pi,
\end{equation}
which is, for each $J$, given by maps $W\to\twopresilt\Pi,\ w\mapsto P_J\otimes_\Pi I_w\ \mbox{ and }\ W\to\twopresilt\Pi,\ w\mapsto P_J\otimes_\Pi I_w^*[1]$.
We have $C^\circ(P_J\otimes_\Pi I_w)=w(C^\circ(P_J))$ and $C^\circ(P_J\otimes_\Pi I_w^*[1])=-w(C^\circ(P_J))$.
By \cite[Section 5.13, Theorem]{H}, we have a decomposition
\begin{equation}\label{decompose H^+}
H^+=\bigsqcup_{\emptyset\neq J \subset \{1,2,\ldots,n\}, \ w \in W/W_J}w(C^\circ(P_J)),
\end{equation}
where $w(C^\circ(P_J))$ depends only on the coset $wW_J$. Thus we have
\begin{equation}\label{decompose H^-}
H^-=\bigsqcup_{\emptyset\neq J \subset \{1,2,\ldots,n\}, \ w \in W/W_J}-w(C^\circ(P_J)).
\end{equation}
Since we have a surjection \eqref{WW surjection} and each element $U\in\twopresilt\Pi$ is uniquely determined by $C^\circ(U)$, the two equalities \eqref{decompose H^+} and \eqref{decompose H^-} imply our first claim (a). The second claim of (b) follows immediately.
\end{proof}

In the rest, we give an explicit description of the TF equivalence classes contained in $H$.

Our strategy is to use the factor algebra
\begin{align*}\Pi':=\Pi / \langle e_n \rangle,\end{align*}
which is the preprojective algebra of type $\mathbb{A}_{n-1}$, and the parabolic subgroup
\begin{align*}W':=\langle s_1,s_2,\ldots,s_{n-1}\rangle\subset W,\end{align*}
which is the Coxeter group of type $\mathbb{A}_{n-1}$ and hence isomorphic to the symmetric group of rank $n$.
We set $P'(i):=P_{\Pi'}(i) \in \proj \Pi'$ for $i \in \{1,2,\ldots,n-1\}$.
By \cite[Theorem 3.9]{M}, there exists a bijection
\begin{align*}W' \to \twosilt \Pi',\ w \mapsto I'_w:=I_w\Lotimes_{\Pi}\Pi'\end{align*}
such that $C(I'_w)=w(C(\Pi')) \subset K_0(\proj \Pi')$, where a canonical action of $W'$ on $K_0(\proj \Pi')_\R$ is given by, for each $1\le i,j\le n-1$,
\begin{align*}
s_j[P'(i)]:=\begin{cases}
-[P'(1)]+[P'(2)] & (i=j=1) \\
-[P'(n-1)]+[P'(n-2)] & (i=j=n-1) \\
-[P'(i)]+[P'(i-1)]+[P'(i+1)] & (i=j \notin \{1,n-1\}) \\
[P'(i)] & (i \ne j).
\end{cases}
\end{align*}
Since $\Pi'$ is $\tau$-tilting finite, $K_0(\proj \Pi')=\bigsqcup_{U \in \twopresilt \Pi'}C^\circ(U)$ holds.

We set $P'_J:=\bigoplus_{j \in J} P'(j)\in\proj\Pi'$, and $W'_J \subset W'$ as the parabolic subgroup generated by $\{s_j\}_{j \notin J,\ j \ne n}$
for each $J \subset \{1,2,\ldots,n-1\}$.
As in the case of $\Pi$, for each $U \in \twopresilt \Pi'$, there exist $w \in W'$ and $J \subset \{1,2,\ldots,n-1\}$ such that
$C^\circ(U)=w(C^\circ(P'_J))$, and we have the following description of TF equivalence classes.

\begin{proposition}\label{Pi'}
We have a decomposition of $K_0(\proj \Pi')_\R$ into the TF equivalence classes
\begin{align}\label{decompose Pi'}
K_0(\proj \Pi')_\R=\bigsqcup_{J \subset \{1,2,\ldots,n-1\}, \ w \in W'/W'_J}w(C^\circ(P'_J))=\bigsqcup_{U \in \twopresilt \Pi'}C^\circ(U).
\end{align}
\end{proposition}

We will prove that \eqref{decompose Pi'} gives the decomposition of $H \subset K_0(\proj \Pi)_\R$. For this purpose, let
\begin{align*}\pi:-\otimes_\Pi\Pi'  \colon K_0(\proj \Pi)_\R \to K_0(\proj \Pi')_\R.\end{align*}
It restricts to an isomorphism $H\simeq K_0(\proj \Pi')_\R$, whose inverse is given by
\begin{align*}\iota \colon K_0(\proj \Pi')_\R \simeq H\subset K_0(\proj \Pi)_\R,\ [P'(i)] \to [P(i)]-[P(n)]\ \mbox{ for each }\ i \in \{1,2,\ldots,n-1\}.\end{align*}
The action of $W'(\subset W)$ commutes with $\pi$ and $\iota$.

We are ready to state our main result.

\begin{theorem}\label{preproj TF}
We have a decomposition of $H \subset K_0(\proj \Pi)_\R$ into the TF equivalence classes
\begin{align*}
H=\bigsqcup_{J \subset \{1,2,\ldots,n-1\}, \ w \in W'/W'_J}w
\left(\cone^\circ\{[P(j)]-[P(n)] \mid j \in J\}\right)=\bigsqcup_{U \in \twopresilt \Pi'}\iota(C^\circ(U)).
\end{align*}
\end{theorem}

Let $n=3$ and $\eta_{i,j}:=[P(i)]-[P(j)]$ for $i,j \in \{1,2,3\}$ with $i \ne j$.
In the following picture,
the dashed hexagon is contained in the hyperplane $H$,
and $D_{\eta_{1,2}}$ is the gray region:
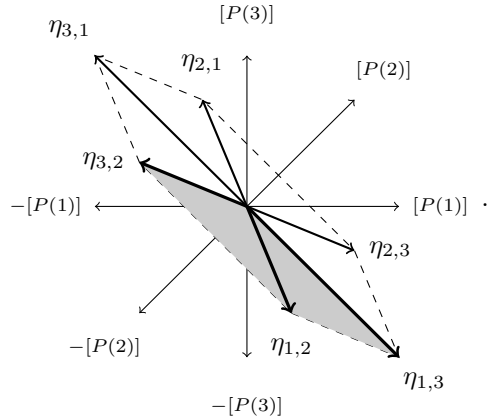
\begin{align*}
\begin{tikzpicture}[every node/.style={circle},baseline=0pt,->]
\node (0) [coordinate] at ( 0, 0) {};
\node (1) [coordinate,label=  0:{$\scriptstyle{ [P(1)]}$}] at (  0:2cm) {};
\node (2) [coordinate,label= 45:{$\scriptstyle{ [P(2)]}$}] at ( 45:2cm) {};
\node (3) [coordinate,label= 90:{$\scriptstyle{ [P(3)]}$}] at ( 90:2cm) {};
\node (-1)[coordinate,label=180:{$\scriptstyle{-[P(1)]}$}] at (180:2cm) {};
\node (-2)[coordinate,label=225:{$\scriptstyle{-[P(2)]}$}] at (225:2cm) {};
\node (-3)[coordinate,label=270:{$\scriptstyle{-[P(3)]}$}] at (270:2cm) {};
\node (12)[coordinate,label=270:{$\eta_{1,2}$}] at ($(1)-(2)$) {};
\node (13)[coordinate,label=315:{$\eta_{1,3}$}] at ($(1)-(3)$) {};
\node (23)[coordinate,label=  0:{$\eta_{2,3}$}] at ($(2)-(3)$) {};
\node (21)[coordinate,label= 90:{$\eta_{2,1}$}] at ($(2)-(1)$) {};
\node (31)[coordinate,label=135:{$\eta_{3,1}$}] at ($(3)-(1)$) {};
\node (32)[coordinate,label=180:{$\eta_{3,2}$}] at ($(3)-(2)$) {};
\draw (0) to (-1);
\draw (0) to (-2);
\draw (0) to (-3);
\draw[dashed] (12)--(13)--(23)--(21)--(31)--(32)--cycle;
\fill[black!20] (0)--(13)--(12)--(32)--cycle;
\draw (0) to (1);
\draw (0) to (2);
\draw (0) to (3);
\draw[very thick] (0) to (12);
\draw[very thick] (0) to (13);
\draw[thick] (0) to (23);
\draw[thick] (0) to (21);
\draw[thick] (0) to (31);
\draw[very thick] (0) to (32);
\end{tikzpicture}.
\end{align*}

\subsection{Proof of Theorem \ref{preproj TF}}

To study preprojective algebras of type $\widetilde{\mathbb{A}}$ in this subsection, we apply representation theory of string algebras, which are representation-tame. A classification of indecomposable modules over string algebras are given in \cite{BR,WW}, and homomorphisms between indecomposable modules are also known \cite{CB,Krause}.

A \emph{brick band} is a band $b$ such that the corresponding band module $M(b,\lambda)$ is a brick. For a band $b$, we denote by $P_1^b\to P_0^b\to M(b,\lambda)\to0$ a minimal projective presentation of $M(b,\lambda)$, and let $\eta^b:=[P_0^b]-[P_1^b]$. Then $P_i^b$ and $\eta^b$ are independent of a choice of parameter $\lambda\in k^\times$ \cite{Krause}.
For brick bands $b,b'$,
we write $b \sim b'$ if $b$ and $b'$ are isomorphic as bands;
more precisely, if $b'$ is a cyclic permutation of $b$ or $b^{-1}$.

\begin{proposition}\label{string brick band}
Let $A=KQ/I$ be a special biserial algebra.
\begin{enumerate}[\rm (a)]
\item
$A$ is $E$-tame.
\item
For any indecomposable rigid $\theta \in K_0(\proj A)$
which corresponds to $U \in \indtwopresilt A$,
$H^0(U)$ is not a band module.
\item Let $b$ be a brick band and $\eta=\eta^b$. For any general $f\in\Hom(\eta)$, there exists $\lambda_f\in k^\times$ such that $C_f\simeq K_{\nu f}\simeq M(b,\lambda_f)$.
In particular, $\eta$ is indecomposable non-rigid.
\item In (c), $D_\eta=\Theta_{M(b,\lambda)}$ holds. Moreover, $M(b,\lambda)$ is a simple object of $\WW_\eta$.
\item There exists a bijection
\begin{align*}\{\text{brick bands}\}/{\sim}\to\{\text{indecomposable non-rigid elements}\}
\quad \text{given by}\quad b\mapsto\eta^b.\end{align*}
\end{enumerate}
\end{proposition}

\begin{proof}
(a) Since $A$ is representation-tame, it is $E$-tame by Proposition \ref{tameness}(b).

(b) Each band module $X$ satisfies $X\simeq \tau X$. Since $H^0(U)$ is $\tau$-rigid, it is not a band module.

(c) Let $\eta:=\eta^b$, and write $b=p_1^{-1}q_1p_2^{-1}q_2\cdots p_\ell^{-1}q_\ell$ for paths $p_i$ and $q_i$ of length $\ge 1$ in the quiver $Q$.
For each $i \in \{1,2,\ldots,\ell\}$,
if $p_i$ is a path starting at $x_i$ and ending at $y_i$, 
then we set $P_{0,i}=P(x_i)$ and $P_{1,i}=P(y_i)$.
Then $p_i$ and $q_i$ give morphisms $p_i \colon P_{1,i}\to P_{0,i}$ and $q_i \colon P_{1,i+1}\to P_{0,i}$, where $\ell+1:=1$.
We can check $P_s^b=\bigoplus_{j=1}^\ell P_{s,j}$ holds for $s=0,1$.
Each $\lambda\in k^\times$ gives a morphism
\begin{align*}\Hom(\eta)\ni f_\lambda:=\begin{bmatrix}
p_1&q_1&0&\cdots&0&0\\
0&p_2&q_2&\cdots&0&0\\
\vdots&\vdots&\vdots&\ddots&\vdots&\vdots\\
0&0&0&\cdots&p_{\ell-1}&q_{\ell-1}\\
\lambda q_\ell&0&0&\cdots&0&p_\ell
\end{bmatrix}\colon P_1^b=\bigoplus_{j=1}^\ell P_{1,j}\to P_0^b=\bigoplus_{j=1}^\ell P_{0,j}\end{align*}
such that $C_{f_\lambda}\simeq M(b,\lambda)$.
Then $K_{\nu f_\lambda}=\tau M(b,\lambda)\simeq M(b,\lambda)$ holds.

Consider the action of $G=\Aut_A(P_0^b)\times\Aut_A(P_1^b)$ on $\Hom(\eta)$ given by $(g,h)f:=gfh^{-1}$. Then
\begin{equation}\label{codim}
{\rm codim}\,Gf_\lambda\stackrel{\text{Prop. \ref{codim=E}}}{=}E(f_\lambda,f_\lambda)\stackrel{\eqref{Serre duality}}{=}\dim\Hom_A(C_{f_\lambda},K_{\nu f_\lambda})=\dim_k\End_A(M(b,\lambda))=1.
\end{equation}
Consider a morphism of algebraic varieties $F \colon G\times k^\times\to\Hom(\eta^b)$, $F((g,h),\lambda):=(g,h)f_\lambda$.
Since $G\times k^\times$ is irreducible, we have irreducible closed subsets
\begin{align*}\overline{Gf_\lambda}\subsetneq X:=\overline{F(G\times k^\times)}\subset\Hom(\eta).\end{align*}
By \eqref{codim}, we obtain $X=\Hom(\eta)$.

(d) By Theorem \ref{D f largest}(a) and Proposition \ref{tame brick}(b), we have $D_\eta=D_f=\Theta_{M(b,\lambda_f)}$.
Since the dimension vectors of submodules of $M(b,\lambda)$ are independent of $\lambda \in k^\times$ by \cite{Krause}, we have $D_\eta=\Theta_{M(b,\lambda)}$ for each $\lambda\in k^\times$.
The last assertion follows from Proposition \ref{tame brick}(b).

(e) The map is well-defined by (c). 
The injectivity also follows from (c); more explicitly, if two brick bands $b,b'$ satisfy $\eta^b=\eta^{b'}=:\eta$,
then (c) implies that, for any general $f \in \Hom(\eta)$, there exist some $\lambda_f,\lambda'_f \in k^\times$ such that $M(b,\lambda_f) \simeq C_f \simeq M(b',\lambda'_f)$,
which yields $b \sim b'$.

To prove the surjectivity, let $\eta$ be an indecomposable non-rigid element in $K_0(\proj A)$.
By Proposition \ref{tameness}(b), for any general $f\in\Hom(\eta)$, $C_f\simeq K_{\nu f}$ are bricks. Then there exist a brick band $b$ and $\lambda_f\in k^\times$ such that $C_f\simeq K_{\nu f}\simeq M(b,\lambda_f)$.
Since any general $f\in\Hom(\eta)$ is a minimal projective presentation of $C_f \simeq M(b,\lambda_f)$,
we have $\eta=\eta^b$ as desired.
\end{proof}

Now we define factor algebras of $\Pi$ and $\Pi'$ by
\begin{align*}A:=\Pi/\langle x,y,z\rangle,\quad A':=\Pi'/\langle x,y,z\rangle.\end{align*}
In terms of quiver with relations, $A$ is obtained from $\Pi$ by factoring out the following relations:
\begin{itemize}
\item
$\alpha_i\alpha_{i+1}\cdots\alpha_{i+n-1}=\beta_i\beta_{i+1}\cdots\beta_{i+n-1}=0$,
\item
$\alpha_i\beta_{i+1}=\beta_{i+1}\alpha_i=0$
\end{itemize}
for $i \in \{1,2,\ldots,n\}$. Using an isomorphism
\begin{align*}-\otimes_\Pi A:K_0(\proj \Pi)_\R\simeq K_0(\proj A)_\R,\end{align*}
we identify $K_0(\proj \Pi)_\R$ and $K_0(\proj A)_\R$. 

The reduction theorem by \cite{Kimura} (cf.~\cite{EJR,IK,VG}) allows us to treat $A$ instead of $\Pi$ as follows.

\begin{proposition}\label{infinite reduction}
Under the setting above, we have the following properties. 
\begin{enumerate}[\rm (a)]
\item
\cite[Theorem 5.4]{Kimura}
The torsion classes in $\operatorname{\mathsf{fl}} \Pi$ 
bijectively correspond to those in $\mod A$ preserving inclusions;
namely $\TT \mapsto \TT \cap \mod A$.
Similarly, the torsion classes in $\operatorname{\mathsf{fl}} \Pi'$ 
bijectively correspond to those in $\mod A'$ preserving inclusions;
namely $\TT \mapsto \TT \cap \mod A'$.
\item
The TF equivalence classes on $K_0(\proj \Pi)_\R$ coincide with those on $K_0(\proj A)_\R$.
Similarly, the TF equivalence classes on $K_0(\proj \Pi')_\R$ coincide with those on $K_0(\proj A')_\R$.
\end{enumerate}
\end{proposition}

\begin{proof}
(a)
We can regard $\Pi$ and $\Pi'$ as $k[[x,y,z]]$-algebras.
Thus we can apply \cite[Theorem 5.4]{Kimura}.

(b) 
For any $\theta \in K_0(\proj A)_\R$, we have
$\overline{\TT}_\theta \cap \mod A=\overline{\TT}^A_\theta$ and
$\TT_\theta \cap \mod A=\TT^A_\theta$ (cf. Example \ref{pi theta}(b)). Thus (a) implies the first assertion. The proof of the second assertion is the same.
\end{proof}

Since $A$ is a string algebra, we can use Proposition \ref{string brick band}.
The following combinatorial observation is crucial.

\begin{proposition}\label{1111}
Let $b$ be a brick band for the string algebra $A$.
Then $\dimv M(b,\lambda)=(1,1,\ldots,1)$ for any $\lambda \in k^\times$.
\end{proposition}

\begin{proof}
By the definition of the string algebra $A$,
there exists $m \in \Z_{\ge 1}$ such that $\dimv M(b,\lambda)=m(1,1,\ldots,1)$. We need to show $m=1$. We may assume that $b$ consists only of arrows in $\{\alpha_i,\beta_i^{-1}\mid 1\le i\le n\}$. We define $I_\pm \subset \{1,2,\ldots,n\}$ by
\begin{align*}
I_+&:=\{ i \in \{1,2,\ldots,n\} \mid \text{$b$ contains $\alpha_i \colon i \to i+1$}\}, \\
I_-&:=\{ i \in \{1,2,\ldots,n\} \mid \text{$b$ contains $\beta_i^{-1} \colon i \to i+1$}\}.
\end{align*}
Since the quiver of $A$ is a double of $\widetilde{\mathbb A}_{n-1}$, we have $I_+ \cup I_-=\{1,2,\ldots,n\}$.

We prove $I_+ \cap I_-=\emptyset$.
If $i \in I_+ \cap I_-$, then there exists some string $s$ such that $\alpha_i s \beta_i^{-1}$ is a substring of $b$. Then the string module $M(s)$ corresponding to $s$ is a proper submodule of 
the band module $M(b,\lambda)$ with $\dimv M(s)=m'(1,1,\ldots,1)$ for some $m'$. Thus $M(b,\lambda)$ is not a simple object of $\WW_{\eta^b}$, a contradiction to Proposition \ref{string brick band}(d). Thus $I_+ \cap I_-=\emptyset$ holds.

Consequently, $b$ is of the form $c^m$, where $c$ is a string of length $n$. This implies $m=1$.
\end{proof}

We also need the following observations.
 
\begin{lemma}\label{TF for Pi and Pi'}
In $K_0(\proj A)_\R$, the following assertions hold.
\begin{enumerate}[\rm(a)] 
\item Each indecomposable element in $H\cap K_0(\proj A)$ is non-rigid.
\item For each $\theta\in H\cap K_0(\proj A)$, we have $\ind\theta\subset H$.
\item Let $\theta,\theta' \in H$. Then $\theta$ and $\theta'$ are TF equivalent if and only if $\pi(\theta)$ and $\pi(\theta')$ are TF equivalent in $K_0(\proj A')_\R$.
\end{enumerate}
\end{lemma}

\begin{proof}
(a) Since $\Cone\cap H=\{0\}$ holds by Proposition \ref{half planes}, the assertion holds.

(b) Let $\eta,\theta\in K_0(\proj A)$ be indecomposable elements such that $\eta\oplus\theta$. It suffices to show that $\eta\in H^+$ implies $\theta\in H^+$, and $\eta\in H^-$ implies $\theta\in H^-$.
We only prove the assertion for $H^+$ since the proof of $H^-$ is the same. 
Proposition \ref{half planes} implies that, if the closure $\overline{C}$ of a TF equivalence class $C$ intersects with $H^+$, then $\overline{C}\subset H^+\cup\{0\}$. 
Let $\eta\in H^+$. Since $C:=\cone^\circ\{\eta,\theta\}$ is a TF equivalence class by Theorem \ref{describe TF for E-tame},
we have $\overline{C} \cap H^+ \ne \emptyset$ and hence $\theta\in\overline{C}\subset H^+\cup\{0\}$.

(c) We prove the ``only if'' part. If $\theta$ and $\theta'$ are TF equivalent, then $\overline{\TT}^A_\theta=\overline{\TT}^A_{\theta'}$ and $\overline{\FF}^A_\theta=\overline{\FF}^A_{\theta'}$. By Example \ref{pi theta}, $\overline{\TT}^{A'}_{\pi(\theta)}=\overline{\TT}^{A'}_{\pi(\theta')}$ and $\overline{\FF}^{A'}_{\pi(\theta)}=\overline{\FF}^{A'}_{\pi(\theta')}$. Thus $\pi(\theta)$ and $\pi(\theta')$ are TF equivalent.

It remains to show the ``if'' part. Assume that $\theta$ and $\theta'$ are not TF equivalent. By Corollary \ref{ind and TF}, we have $\ind\theta\neq\ind\theta'$. Without loss of generality, we can assume that there exists $\eta\in\ind\theta\setminus\ind\theta'$. By Proposition \ref{D f largest 2}, we have $\theta\in D_\eta^\circ\not\ni\theta'$.
On the other hand, $\eta\in H$ holds by (b) and hence $\eta$ is non-rigid by (a). Take a brick band $b$ in Proposition \ref{string brick band}(e) satisfying $\eta=\eta^b$, and let $X:=M(b,\lambda)$ for a fixed $\lambda\in k^\times$. Then
\begin{align*}\theta\in \Theta_X^\circ\not\ni\theta'\end{align*}
holds by Proposition \ref{string brick band}(d). Note that $\dimv X=(1,1,\ldots,1)$ holds by Proposition \ref{1111}, and hence $\theta'\in H$ implies $\theta'(X)=0$. By Lemma \ref{simple and interior}(e), there exists a factor module $Z=X/Y$ of $X$ such that $\dimv Y\notin\R\dimv X$, $\theta'(Y)\ge0$ and $\theta'(Z)\le0$. Thus we have
\begin{align*}Y\in\overline{\FF}^A_\theta\setminus\overline{\FF}^A_{\theta'}\ \mbox{ and }\ Z\in\overline{\TT}^A_\theta\setminus\overline{\TT}^A_{\theta'}.\end{align*}
Since $\dimv X=(1,1,\ldots,1)$, the $n$th entry of either $Y$ or $Z$ is zero. Thus $Y$ or $Z$ belongs to $\mod A'$. It gives an object of $\overline{\FF}^{A'}_{\pi(\theta)}\setminus\overline{\FF}^{A'}_{\pi(\theta')}$ or $\overline{\TT}^{A'}_{\pi(\theta)}\setminus\overline{\TT}^{A'}_{\pi(\theta')}$ by Example \ref{pi theta}. Thus $\pi(\theta)$ and $\pi(\theta')$ are not TF equivalent.
\end{proof}

Then we can prove Theorem \ref{preproj TF}.

\begin{proof}[Proof of Theorem \ref{preproj TF}]
By Propositions \ref{Pi'}, \ref{infinite reduction}(b) and Lemma \ref{TF for Pi and Pi'}(c), the TF equivalence classes in $H$ can be written as $\iota(w(C^\circ(P'_J)))$ for some $J\subset\{1,2,\ldots,n-1\}$ and $w\in W'/W'_J$. Since $\iota(w(C^\circ(P'_J)))=w\left(\cone^\circ\{[P(j)]-[P(n)] \mid j \in J\}\right)$ holds, we obtain the assertion.
\end{proof}

\end{document}